\documentclass[a4paper,10pt]{article}
\usepackage{amssymb,amsmath,amsthm}
\usepackage{amsfonts}
\usepackage[english]{babel}
\usepackage{hyperref}  
\usepackage{color}

\numberwithin{equation}{section}

\newtheorem{theorem}{Theorem}[section]
\newtheorem{proposition}[theorem]{Proposition}
\newtheorem{lemma}[theorem]{Lemma}
\newtheorem{corollary}[theorem]{Corollary}

\newtheorem{remark}[theorem]{Remark}
\newtheorem{remarks}[theorem]{Remark}
\newtheorem{definition}[theorem]{Definition}
\newcommand{\be}{\begin{equation}}
\newcommand{\ee}{\end{equation}}

\newcommand{\e}{\varepsilon}

\newcommand{\R}{\mathbb R}
\newcommand{\C}{\mathbb C}

\newcommand{\Z}{\mathbb Z}

\newcommand{\N}{\mathbb N}
\newcommand{\T}{\mathbb T}
\newcommand{\sign}{{\rm sign}}

\renewcommand{\b }{\beta }
\newcommand{\s }{\sigma }
\newcommand{\ii }{{\rm i} }

\newcommand{\g }{\gamma}

\newcommand{\vphi}{\varphi }

\newcommand{\pa}{\partial}

\def\ba{\begin{aligned}}
\def\ea{\end{aligned}}
\def\beginm{\begin{multline}}
\def\endm{\end{multline}}

\newcommand{\Lipg}{{\rm{Lip}(\g)}}

\begin{document}

\title{{\bf Growth of Sobolev norms for time dependent periodic Schr\"odinger equations with sublinear dispersion}}

\date{}
 \author{ Riccardo Montalto \footnote{Supported by the Swiss National Science Foundation}
 \vspace{2mm} 
\\ \small 
  Institut f\"ur Mathematik, Universit\"at Z\"urich
Winterthurerstrasse 190, CH-8057 Z\"urich, CH
\\ \small 
E-mail: riccardo.montalto@math.uzh.ch}

\maketitle

\noindent
{\bf Abstract:}
In this paper we consider Schr\"odinger equations with sublinear dispersion relation on the one-dimensional torus $\T := \R /(2 \pi \Z)$. More precisely, we deal with equations of the form $\partial_t u = \ii {\cal V}(\omega t)[u]$ where ${\cal V}(\omega t)$ is a quasi-periodic in time, self-adjoint pseudo-differential operator of the form ${\cal V}(\omega t) = V(\omega t, x) |D|^M + {\cal W}(\omega t)$, $0 < M \leq 1$, $|D| := \sqrt{- \partial_{xx}}$, $V$ is a smooth, quasi-periodic in time function and ${\cal W}$ is a quasi-periodic time-dependent pseudo-differential operator of order strictly smaller than $M$. Under suitable assumptions on $V$ and ${\cal W}$, we prove that if $\omega$ satisfies some non-resonance conditions, the solutions of the Schr\"odinger equation $\partial_t u = \ii {\cal V}(\omega t)[u]$ grow at most as $t^\eta$, $t \to + \infty$ for any $\eta > 0$. The proof is based on a reduction to constant coefficients up to smoothing remainders of the vector field $\ii {\cal V}(\omega t)$ which uses Egorov type theorems and pseudo-differential calculus. The {\it homological equations} arising in the reduction procedure involve both time and space derivatives, since the dispersion relation is sublinear. Such equations can be solved by imposing some Melnikov non-resonance conditions on the frequency vector $\omega$. 
 \\[2mm]
\\[1mm]

\tableofcontents

\section{Introduction and main result}
In this paper we consider linear quasi-periodic in time Schr\"odinger-type equations with sublinear dispersion relation of the form
\begin{equation}\label{main equation}
\partial_t u = \ii {\cal V}(\omega t)[u]  \,, \quad x \in \T
\end{equation}
where $\T := \R / (2 \pi \Z)$ is the one-dimensional torus, $\omega \in \Omega \subset \R^\nu$, ${\cal V}(\vphi)$, $\vphi \in \T^\nu$ is a $L^2$ self-adjoint, pseudo-differential  Schr\"odinger operator of the form 
\begin{equation}\label{forma iniziale cal V (t)}
{\cal V}(\vphi) := V(\vphi, x) |D|^M + {\cal W}(\vphi)\,, \quad |D| := \sqrt{- \partial_{xx}}\,,\quad 0 < M \leq 1\,.
\end{equation}
We assume that $ V$ is a real valued ${\cal C}^\infty$ function defined on $\T^\nu \times \T$ satisfying $\inf_{(\vphi, x) \in \T^\nu \times \T} V(\vphi, x) > 0$, in the case $0 < M < 1$ and {\it close} to $1$ in the case $M = 1$. The operator ${\cal W}(\vphi)$ is a time-dependent pseudo differential operator of order strictly smaller than $M$. Our main goal is to show that given $t_0 \in \R$, $s \geq 0$, $u_0 \in H^s(\T)$, for {\it most values} of the parameters $\omega \in \Omega$, the Cauchy problem 
 \begin{equation}\label{problema di cauchy intro}
 \begin{cases}
 \partial_t u = \ii {\cal V}(\omega t)[u]   \\
 u(t_0, x) = u_0(x)
 \end{cases}
 \end{equation}  
 admits a unique solution $u(t)$ satisfying, for any $\eta > 0$, the bound 
 \begin{equation}\label{crescita sobolev generale}
 \| u(t)\|_{H^s} \leq C(s, \eta) (1 + |t - t_0|)^\eta \| u_0\|_{H^s}\,, \quad \forall t \in \R
 \end{equation} 
 for some constant $C(s, \eta) > 0$. Here, $H^s(\T)$ denotes the standard Sobolev space on the 1-dimensional torus $\T$ equipped with the norm $\| \cdot \|_{H^s}$. 
 
 \noindent
 The problem of estimating the high Sobolev norms for solutions of linear Schr\"odinger-type equations of the form $\partial_t u = \ii (H + V(t)) u$, in the case when $H$ satisfies the so called {\it spectral gap condition}, has been extensively investigated. Such a condition states that the spectrum of the operator $H$ can be enclosed in disjont clusters $(\sigma_j)_{j \geq 0}$ such that the distance between $\sigma_j$ and $\sigma_{j + 1}$ tends to $+ \infty$ for $j \to + \infty$.

\noindent  
For the Schr\"odinger operator ${\cal V}(t) = - \Delta + V(t, x)$ on the $d$-dimensional torus $\T^d$, the growth $\sim t^\eta$ of the $\| \cdot\|_{H^s}$ norm of the solutions of $\partial_t u = \ii {\cal V}(t)[u]$ has been proved by Bourgain in \cite{bourgain-1} for smooth quasi-periodic in time potentials and in \cite{bourgain-2} for smooth and bounded time dependent potentials. In the case where the potential $V$ is analytic and quasi-periodic in time, Bourgain  \cite{bourgain-1} proved also that $\| u(t) \|_{H^s}$ grows like a power of $\log(t)$. Moreover, this bound is optimal, in the sense that he constructed an example for which $\| u(t)\|_{H^s}$ is bounded from below by a power of $\log(t)$. The result obtained in \cite{bourgain-2} has been extended by Delort \cite{delort} for Schr\"odinger operators on Zoll manifolds. Furthermore, the logarithmic growth of $\| u(t) \|_{H^s}$ proved in \cite{bourgain-1} has been extended by Wang \cite{wang-1} in dimension $1$, for any real analytic and bounded potential. 

\noindent
 All the aforementioned results concern Schr\"odinger equations with bounded perturbations. The first result in which the growth of $\| u(t) \|_{H^s}$ is established in the case of unbounded perturbations is due to Maspero-Robert \cite{albertino}. More precisely, they prove the growth $\sim t^\eta$ of $\| u(t)\|_{H^s}$, for Schr\"odinger equations of the form $\ii \partial_t u = L(t) [u]$ where $L(t) = H + P(t)$, $H$ is a time-independent operator of order $\mu  + 1$ satisfying the {\it spectral gap condition} and  $P(t)$ is an operator of order $\nu \leq \mu/(\mu + 1)$ (see Theorem 1.8 in \cite{albertino}).  This last paper has been generalized, independently and at the same time, by Bambusi-Grebert-Maspero-Robert  \cite{albertino2} in the case in which the order of $P(t)$ is strictly smaller than the one of $H$ and in \cite{Montalto1} for periodic $1$-dimensional Schr\"odinger equations where the order of $P(t)$ is the same as the order of $H$. 
 
 \noindent
 In \cite{albertino2}, the authors deal with some Schr\"odinger equations in which the spectral gap condition is violated. In particular, they study the non-resonant harmonic oscillator on $\R^d$, in which the gaps are dense and the relativistic Schr\"odinger equation on Zoll manifolds, in which the distance between the gaps are constants.  
 In all the papers mentioned above, the dispersion relation is at least linear.  
 \noindent
 The purpose of the present paper is to provide some results concerning the growth of Sobolev norms of the solutions of some Schr\"odinger equations with sublinear dispersion, in the case in which the order of $H$ is the same as the order of $P(t)$.

 \noindent
 We also mention that in the case of quasi-periodic systems with {\it small} perturbations, i.e. $\ii \partial_t u = L(\omega t)[u] $, $L(\omega t) = H + \e P(\omega t)$ it is often possible to prove that $\| u(t) \|_{H^s}$ is uniformly bounded in time for $\e$ small enough and for a {\it large }set of frequencies $\omega$. The general strategy to deal with these quasi-periodic systems is called {\it reducibility}. It consists in costructing, for most values of the frequencies $\omega$ and for $\e$ small enough, a bounded quasi-periodic change of variables $\Phi(\omega t)$ which transforms the equation $\ii \partial_t u = L(\omega t) u$ into a time independent system $\ii \partial_t v = {\cal D} v$ whose solution preserves the Sobolev norms $\| v(t) \|_{H^s}$. We mention the results of Eliasson-Kuksin \cite{EK1} which proved the reducibility of the Schr\"odinger equation on $\T^d$ with a small, quasi-periodic in time analytic potential and Grebert-Paturel \cite{GrebertPaturel} which proved the reducibility of the quantum harmonic oscillator on $\R^d$. Concerning KAM-reducibility with unbounded perturbations, we mention Bambusi \cite{Bambusi1}, \cite{Bambusi2} for the reducibility of the quantum harmonic oscillator with unbounded perturbations (see also \cite{albertino3} in any dimension), \cite{BBM-Airy}, \cite{BBM-auto}, \cite{Giuliani} for fully non-linear KdV-type equations, \cite{Feola}, \cite{Feola-Procesi} for fully-nonlinear Schr\"odinger equations, \cite{BM16}, \cite{BertiMontalto} for the water waves system and \cite{Montalto} for the Kirchhoff equation. Note that in \cite{BBM-Airy}, \cite{BBM-auto}, \cite{Giuliani}, \cite{BM16}, \cite{BertiMontalto}, \cite{Montalto} the reducibility of the linearized equations is obtained as a consequence of the KAM theorems proved for the corresponding nonlinear equations. 
 
 \noindent
 In the case of sublinear growth of the eigenvalues, the first KAM-reducibility result is proved in \cite{BBHM} for the pure gravity water waves equations and the technique has been extended in \cite{Montaltoreducibility} to deal with a class of linear wave equations on $\T^d$ with smoothing quasi-periodic in time perturbations. 
 
 \noindent
 We now state in a precise way the main results of this paper. First, we introduce some notations. For any function $u \in L^2(\T)$, we introduce its Fourier coefficients 
\begin{equation}\label{coefficienti fourier}
\widehat u(\xi) := \frac{1}{2 \pi} \int_\T u(x) e^{- \ii x \xi}\, d x\,, \qquad \forall \xi \in \Z\,. 
\end{equation}
For any $s \geq 0$, we introduce the Sobolev space of complex valued functions $H^s \equiv H^s(\T)$, as 
\begin{equation}\label{definizione sobolev}
\begin{aligned}
& H^s  := \Big\{ u \in L^2(\T) :  \| u \|_{H^s}^2 :=\sum_{\xi \in \Z} \langle \xi \rangle^{2 s} |\widehat u(\xi) |^2 < + \infty \Big\}\,,\quad  \langle \xi \rangle := (1 + |\xi|^2)^{\frac12}\,. 
\end{aligned}
\end{equation}

\noindent
Given two Banach spaces $(X, \| \cdot \|_X)$, $(Y, \| \cdot \|_Y)$, we denote by ${\cal B}(X, Y)$ the space of bounded linear operators from $X$ to $Y$ equipped with the usual operator norm $\| \cdot \|_{{\cal B}(X, Y)}$. If $X = Y$, we simply write ${\cal B}(X)$ for ${\cal B}(X, X)$. 

\noindent
Given a linear operator ${\cal R} \in {\cal B}(L^2(\T))$, we denote by ${\cal R}^*$ the adjoint operator of ${\cal R}$ with respect to the standard $L^2$ inner product 
\begin{equation}
\langle u, v \rangle_{L^2} := \int_\T u(x) \overline{v(x)}\, d x\,, \quad \forall u, v \in L^2(\T)\,. 
\end{equation} 
 We say that the operator ${\cal R}$ is self-adjoint if ${\cal R} = {\cal R}^*$. 
For any domain $\Omega \subset \R^d$, we also denote by ${\cal C}^\infty_b(\Omega)$ the space of the ${\cal C}^\infty$ functions on $\Omega$ with all the derivatives bounded. Given a multiindex $\alpha = (\alpha_1, \ldots, \alpha_\nu)\in \N^\nu$ we define its lenght by $|\alpha| := \alpha_1 + \ldots \alpha_\nu$ and $\partial_\vphi^\alpha = \partial_{\vphi_1}^{\alpha_1} \ldots \partial_{\vphi_\nu}^{\alpha_\nu}$. 

\noindent
Since the equation we deal with is a Hamiltonian PDE, we briefly describe the Hamiltonian formalism. We define the symplectic form $\Omega : L^2(\T ) \times L^2(\T ) \to \R$ by 
\begin{equation}\label{forma simplettica coordinate complesse}
\Omega[u_1, u_2] : =  \ii \int_{\T} (u_1 \bar u_2 -  \bar u_1 u_2)\, dx\,, \quad \forall u_1, u_2 \in L^2(\T)\,. 
\end{equation}
Given a family of linear operators ${\cal R} : \T^\nu \to {\cal B}(L^2)$ such that ${\cal R}(\vphi) = {\cal R}(\vphi)^*$ for any $\vphi \in \T^\nu$, we define the $\vphi$-dependent quadratic Hamiltonian associated to ${\cal R}$ as 
$$
{\cal H}(\vphi, u) := \langle {\cal R}(\vphi)[u]\,,\, u \rangle_{L^2_x} = \int_\T {\cal R}(\vphi)[u]\, \overline u\, d x\,, \qquad \forall u \in L^2(\T)\,. 
$$
The Hamiltonian vector field associated to the Hamiltonian ${\cal H}$ is defined by
\begin{equation}\label{campo vettoriale hamiltoniano}
X_{\cal H}(\vphi, u) := \ii \nabla_{\overline u} {\cal H}(\vphi, u) = \ii {\cal R}(\vphi)[u]
\end{equation}
where the gradient $\nabla_{\overline u}$ stands for 
$$
\nabla_{\overline u} := \frac{1}{\sqrt{2}} (\nabla_v + \ii \nabla_\psi)\,, \quad v = {\rm Re}(u)\,, \quad \psi := {\rm Im}(u)\,. 
$$
We say that $\Phi : \T^\nu \to {\cal B}(L^2(\T))$ is symplectic if and only if 
$$
\Omega\Big[ \Phi(\vphi)[u_1], \Phi(\vphi)[u_2] \Big] = \Omega [u_1, u_2 ]\,, \qquad \forall u_1, u_2 \in L^2(\T)\,, \quad \forall \vphi \in \T^\nu\,.  
$$
We recall that if $X_{\cal H}$ is a Hamiltonian vector field, then the flow map generated by $X_{\cal H}$ is symplectic.

Let us consider a vector field $X : \T^\nu \to {\cal B}(L^2(\T))$ and a differentiable family of invertible maps $\Phi : \T^\nu \to {\cal B}(L^2(\T))$. Given $\omega \in \R^\nu$, under the change of variables $u = \Phi(\omega t)[v]$, the equation $\partial_t u = X(\omega t)[u]$
transforms into the equation $\partial_t v = X_+(\omega t)[v]$
where the {\it quasi-periodic push-forward} $X_+(\vphi)$ of the vector field $X(\vphi)$ is defined by 
\begin{equation}\label{push forward}
X_+(\vphi) := \Phi_{\omega*} X(\vphi) := \Phi(\vphi)^{- 1} \Big( X(\vphi) \Phi(\vphi)  - \omega \cdot \partial_\vphi \Phi(\vphi) \Big)\,, \quad \vphi \in \T^\nu\,. 
\end{equation}
If $\Phi$ is symplectic and $X$ is a Hamiltonian vector field, then the push-forward $X_+= \Phi_{\omega *} X $ is still a Hamiltonian vector field.

\noindent
In the next two definitions, we also define the class of pseudo differential operators on $\T$ that we shall use along the paper.
\begin{definition}[{\bf The symbol class} $S^m$]\label{classe simboli}
Let $m \in \R$. We say that a ${\cal C}^\infty$ function $a :  \T^\nu \times \T \times \R \to \C$ belongs to the symbol class $S^m$ if and only if for any multiindex $\alpha \in \N^\nu$, for any $k, n \in \N$, there exists a constant $C_{\alpha, n, k} > 0$ such that 
\begin{equation}
|\partial_\vphi^\alpha\partial_x^k  \partial_\xi^n a( \vphi, x, \xi)| \leq C_{\alpha, n, k} \langle \xi \rangle^{m - n}\,, \quad \forall (\vphi, x, \xi) \in  \T^\nu \times \T \times \R\,. 
\end{equation}
We define the class of smoothing symbols $S^{- \infty} := \cap_{m \in \R} S^m$. 
\end{definition}
\begin{definition} \label{def:Ps2} {\bf (the class of operators $OPS^m$)}
Let $m \in \R$ and $a \in S^m$. We define the $\vphi$-dependent linear operator $A(\vphi) = {\rm Op}\big(a(\vphi, x, \xi) \big) = a(\vphi, x, D)$ as 
$$
A(\vphi) [u] (x) := \sum_{\xi \in \Z} a(\vphi,  x, \xi) \widehat u(\xi) e^{\ii x \xi}\,, \qquad \forall u \in {\cal C}^\infty(\T)\,.
$$
We say that the operator $A$ is in the class $OPS^m$. 

\noindent
We define the class of smoothing operators $OPS^{- \infty} := \cap_{m \in \R} OPS^m$. 
\end{definition}
Now we are ready to state the main results of this paper. 
  \begin{itemize}
  \item[ $\bf (H1)$] The operator ${\cal V}(\vphi) = V(\vphi, x) |D|^M + {\cal W}(\vphi)$ in \eqref{forma iniziale cal V (t)} is $L^2$ self-adjoint for any $\vphi \in \T^\nu$.  
  \item[ $\bf (H2)$] The operator ${\cal W}(\vphi)$ is a $\vphi$-dependent pseudo-differential operator ${\cal W}(\vphi) = {\rm Op}(w(\vphi, x, \xi))$, with symbol $w \in S^{M - \frak e}$ for some $\frak e > 0$.  
  


\end{itemize}
We also need a third hypothesis. For this last assumption, we need to distinguish between the cases $0 < M < 1$ and $M = 1$. 
In the case $0 < M <1 $ we assume 
\begin{itemize}
\item[ $\bf (H3)_{M < 1}$] The function $V(\vphi, x)$ in \eqref{forma iniziale cal V (t)} is in $ {\cal C}^\infty(\T^\nu \times \T, \R)$, strictly positive and bounded from below, i.e. $\delta := \inf_{(t, x) \in \T^\nu \times \T} V(\vphi, x) > 0$. 
\end{itemize}
In the case $M = 1$ we assume 
\begin{itemize}
\item[ $\bf (H3)_{M = 1}$] The function $V(\vphi, x)$ in \eqref{forma iniziale cal V (t)} has the form $V (\vphi, x) = 1 + \e P(\vphi, x)$ where $\e \in (0, 1)$ and  $ P \in {\cal C}^\infty(\T^\nu \times \T, \R)$.
\end{itemize}
Given $\tau > \nu - 1$ and $\gamma \in (0, 1)$, we also introduce the set of diophantine frequencies 
\begin{equation}\label{vettori diofantei}
DC(\gamma, \tau) := \Big\{ \omega \in \Omega : |\omega \cdot \ell| \geq \frac{\gamma}{|\ell|^\tau}, \quad \forall \ell \in \Z^\nu \setminus \{ 0 \} \Big\}\,. 
\end{equation}
It is well known that the Lebesgue measure of $\Omega \setminus DC(\gamma, \tau)$ is of order $  O(\gamma)$. 
The main results of this paper are the following 
\begin{theorem}[\bf Growth of Sobolev norms, the case $0 < M <  1$]\label{teo growth of sobolev norms M < 1}
Assume the hypotheses 

\noindent
${\bf (H1)}$, ${\bf (H2)}$, ${\bf (H3)_{M < 1}}$. Let $s > 0$, $\gamma \in (0, 1)$, $\tau > \nu$, $t_0 \in \R$, $u_0 \in H^s(\T)$, $\omega \in DC(\gamma, \tau)$. 
Then there exists a unique global solution $u \in {\cal C}^0(\R, H^s(\T))$ of the Cauchy problem \eqref{problema di cauchy intro} which satisfies the bound \eqref{crescita sobolev generale}. 
\end{theorem}
\begin{theorem}[\bf Growth of Sobolev norms, the case $M = 1$]\label{teo growth of sobolev norms M = 1}
Assume the hypotheses 

\noindent
${\bf (H1)}$, ${\bf (H2)}$, ${\bf (H3)_{M = 1}}$. Let $s > 0$, $u_0 \in H^s(\T)$, $t_0 \in \R$, $\gamma \in (0, 1)$, $\tau > \nu$. Then there exists a constant $\delta_0 \in (0, 1)$ such that if $\e \gamma^{- 1} \leq \delta_0$ the following holds: there exists a Borel set $\Omega_{\gamma, \tau} \subseteq \Omega$ such that for any $\omega \in \Omega_{\gamma, \tau}$ there exists a unique global solution $u \in {\cal C}^0(\R, H^s(\T))$ of the Cauchy problem \eqref{problema di cauchy intro} which satisfies the bound \eqref{crescita sobolev generale}.  
\end{theorem}
\begin{remark}\label{remark cantor set Omega gamma}
The explicit expression of the set $\Omega_{\gamma, \tau}$ in Theorem \ref{teo growth of sobolev norms M = 1} is provided in \eqref{def cal O mu pm gamma}, \eqref{Cantor finale riduzione M = 1}. It is obtained by imposing some {\it first order Melnikov conditions} on the frequency $\omega \in \Omega \subset \R^\nu$. Arguing as in the proof of Corollary 4.2 in \cite{trasporto paper}, it can be proved that the Lebesgue measure of $\Omega \setminus \Omega_{\gamma, \tau}$ is of order $O(\gamma)$. 
\end{remark}
The two theorems stated above will be deduced by the following two {\it normal form} theorems stated below. 
\begin{theorem}[\bf Normal-form theorem, the case $0 < M < 1$]\label{teorema riduzione}
Assume the hypotheses 

\noindent
${\bf (H1)}$, ${\bf (H2)}$, ${\bf (H3)_{M < 1}}$ and let $\omega \in DC(\gamma, \tau)$. For any $K > 0$ there exists a $\vphi$-dependent family of  invertible maps ${\cal T}_K(\vphi) \in {\cal B}(H^s)$, $s \geq 0$ 
such that the following holds: the vector field $\ii {\cal V}(\vphi)$ is transformed, by the map ${\cal T}_K(\vphi)$, into the vector field 
\begin{equation}\label{campo finalissimo cal VK}
\ii {\cal V}_K(\vphi) := ({\cal T}_K)_{\omega*} (\ii {\cal V})(\vphi) = \ii \Big( \lambda_K(D) + {\cal R}_K(\vphi) \Big)
\end{equation}
where $\lambda_K( D) := {\rm Op}(\lambda_K( \xi)) \in OPS^M$ is a time-independent Fourier multiplier with real symbol and ${\cal R}_K \in OPS^{- K}$.  
\end{theorem}
\begin{theorem}[\bf Normal-form theorem, the case $M = 1$]\label{teorema riduzione M = 1}
Assume ${\bf (H1)}$, ${\bf (H2)}$, ${\bf (H3)_{M = 1}}$ and let $\gamma \in (0, 1)$. Then if $\e \gamma^{- 1} \leq \delta$, for any $\omega \in \Omega_{\gamma, \tau}$ (where the constant $\delta$ and the set $\Omega_{\gamma, \tau}$ are given in Theorem \ref{teo growth of sobolev norms M = 1}), the following holds. For any $K > 0$ there exists a $\vphi$-dependent family of invertible maps ${\cal T}_K(\vphi) \in {\cal B}(H^s)$, $s \geq 0$ 
such that the push-forward of the Hamiltonian vector field $\ii {\cal V}(\vphi)$ by means of the map ${\cal T}_K(\vphi)$ has the form 
\begin{equation}\label{campo finalissimo cal VK M = 1}
\ii {\cal V}_K(\vphi) := ({\cal T}_K)_{\omega*} (\ii {\cal V})(\vphi) = \ii \Big( \lambda_K(D) + {\cal R}_K(\vphi) \Big)
\end{equation}
where $\lambda_K( D) := {\rm Op}(\lambda_K( \xi)) \in OPS^1$ is a time-independent Fourier multiplier with real symbol and ${\cal R}_K \in OPS^{- K}$.  
\end{theorem}

In the remaining part of the introduction, we shall explain the main ideas needed to prove the theorems stated above. 
\subsection{Ideas of the proof}
The proof of Theorems \ref{teo growth of sobolev norms M < 1}, \ref{teo growth of sobolev norms M = 1} is deduced in a straightforward way by Theorems \ref{teorema riduzione}, \ref{teorema riduzione M = 1}. Indeed, once we transform the equation $\partial_t u = \ii {\cal V}(\omega t)[u]$ into another one which is an arbitrarily regularizing perturbation of a constant coefficients equation, fixing the number of regularization steps $K \simeq s$, we get immediately that the Sobolev norm $\| u(t) \|_{H^s}$ grows linearly in $t$. Then in order to get the desired estimate \eqref{crescita sobolev generale}, it is enough to apply the classical Riesz-Thorin interpolation Theorem. 

\noindent
The proof of Theorems \ref{teorema riduzione}, \ref{teorema riduzione M = 1} is based on a {\it normal form} procedure, which transforms the vector field $\ii {\cal V}(\vphi)$ into another one which is an arbitrarily regularizing perturbation of a {\it space diagonal} vector field. Such a normal form method is based on symbolic calculus and Egorov type Theorems (see Propositions \ref{Teorema egorov generale}-\ref{teorema egorov campo minore di uno}) and it is performed by constructing iteratively changes of variables which reduce to constant coefficients order by order the vector field $\ii {\cal V}(\vphi)$. The procedure is different in the case $0 < M < 1$ and $M = 1$. The intuitive reason is the following: In the equation $\partial_t u = \ii {\cal V}(t)[u]$ we have the interaction between the time differential operator $\partial_t$ and the space pseudo-differential operator $|D|^M$. If $M < 1$, the effect of the operator $\partial_t$ is stronger than the one of $|D|^M$, hence in the {\it homological equations} which arise in our reduction procedure at any order, we remove first the dependence on the time variable, by requiring that the frequency $\omega$ is diophantine and then the dependence on the space variable. In the case $M = 1$, we deal with the operators $\partial_t$ and $|D|$ which are both of order 1. Hence the homological equations arising along the reduction procedure involve both time and space variable and they are solved by imposing {\it first order Melnikov conditions}, see \eqref{def cal O mu pm gamma}, \eqref{Cantor finale riduzione M = 1}. We also remark that in \cite{Montalto1}, it is studied the case where $M > 1$, in which the effect of the operator $|D|^M$ is much stronger than the effect of the time differential operator $\partial_t$ and we do not need any parameter to perform the reduction procedure.

\noindent
 In the following we shall explain more in detail these two reduction procedures. It is convenient for the sequel to introduce the following notation. Given $A = {\rm Op}(a), B = {\rm Op}(b) \in S^m$ and $m' < m$, we write
\begin{equation}\label{notazione conveniente}
\begin{aligned}
& A = B + OPS^{m'} \quad \text{if} \quad A - B \in OPS^{m'}\,,  \\
& a = b + S^{m'} \quad \text{if} \quad a - b \in S^{m'}\,. 
\end{aligned}
\end{equation}
In particular if $A - B$ is a smoothing operator, we write 
\begin{equation}\label{notazione conveniente smoothing}
A = B + OPS^{- \infty}\, \quad \text{and} \quad a = b + S^{- \infty}\,. 
\end{equation}

\subsubsection{Normal form in the case $M < 1$}\label{idea forma normale M < 1}
The normal form procedure for the vector field $\ii {\cal V}(\omega t)$ in the case $M < 1$ is developed in Section \ref{sezione riduzione M < 1}. 
\begin{enumerate}
\item{\it Reduction of the highest order term.} The first step is to reduce to constant coefficient the highest order term $\ii V(\omega t, x) |D|^M$ of the vector field $\ii {\cal V}(\omega t)$. This is done in Section \ref{sezione ordine principale M < 1}. First, in Section \ref{riduzione tempo ordine principale}, we remove the time-dependence from the operator $V(\omega t, x) |D|^M$ by conjugating the vector field $\ii {\cal V}(\omega t)$ by means of the flow map generated by a Hamiltonian vector field $\ii {\cal G}_0^{(1)}(\vphi)$ of the form ${\cal G}_0^{(1)} (\vphi) := \alpha (\vphi, x) |D|^M + |D|^M \alpha(\vphi, x) = 2 \alpha(\vphi, x) |D|^M + OPS^{M - 1}$. The expansion of the symbol of the transformed vector field $\ii {\cal V}_0^{(1)}(\vphi)$ is given in \eqref{espansione finale cal V 0} and it has the form 
$$
{\cal V}_0^{(1)}(\vphi) = \Big(2 \omega \cdot \partial_\vphi \alpha(\vphi, x) + V(\vphi, x) \Big) |D|^M  + OPS^{M - \overline{\frak e}}
$$ where the constant $\overline{\frak e} > 0$ is defined in \eqref{costante frak e bar M < 1}. By choosing $\omega \in \R^\nu$ diophantine (see \eqref{vettori diofantei}), we can determine the function $\alpha$ so that 
$$
2 \omega \cdot \partial_\vphi \alpha(\vphi, x) + V(\vphi, x) = \langle V \rangle_\vphi(x) \,, \quad  \langle V \rangle_\vphi(x) := \frac{1}{(2 \pi)^\nu} \int_{\T^\nu} V(\vphi, x)\, d \vphi
$$
so that 
$$
{\cal V}_0^{(1)}(\vphi) = \langle V \rangle_\vphi(x) |D|^M + OPS^{M - \overline{\frak e}}
$$
has the highest order $\langle V \rangle_\vphi(x) |D|^M$ independent of $\vphi$. Then, in Section \ref{riduzione spazio ordine massimo M < 1} we remove the $x$ dependence from the highest order term $\langle V \rangle_\vphi(x) |D|^M$. In order to achieve this purpose, we conjugate the vector field $\ii {\cal V}_0^{(1)}$ by means of the time-1 flow map $\Phi_0^{(2)}$ of the transport equation 
$$
\partial_\tau u = \Big( b(\tau; x) \partial_x +  \frac{b_x(\tau; x)}{2}\Big) u\,, \quad b(\tau; x) := - \frac{\beta (x)}{1 + \tau \beta_x (x)}\,. 
$$
Note that the function $b$ is $\vphi$-independent, implying that also the map $\Phi_0^{(2)}$ is $\vphi$-independent, therefore the transformed vector field is given by $\ii {\cal V}_1(\vphi)$, ${\cal V}_1(\vphi) := \Phi_0^{(2)} {\cal V}_0^{(1)} (\Phi_0^{(2)})^{- 1}$. The expansion of the operator ${\cal V}_1(\vphi) = {\rm Op}\big( v_1(\vphi, x, \xi) \big)$ is given  in Lemma \ref{espansione simbolo principale egorov} and its highest order term is given by 
$$
\Big[\langle V\rangle_\vphi( y) \big( 1 + \partial_y \widetilde \beta( y) \big)^M \Big]_{y = x + \beta( x)} |D|^M
$$
where $y \mapsto y + \widetilde \beta(y)$ is the inverse diffeomorphism of $x \mapsto x + \beta(x)$. Then we determine the function $\widetilde \beta$ and a constant $\lambda > 0$ so that 
$$
\langle V\rangle_\vphi( y) \big( 1 + \partial_y \widetilde \beta( y) \big)^M = \lambda\,. 
$$
In order to solve this equation, we use the hypothesis $\bf (H3)_{M < 1}$, i.e. $\inf_{(\vphi, x) \in \T^{\nu + 1}} V(\vphi, x) > 0$, then one gets that ${\cal V}_1(\vphi) = \lambda |D|^M + OPS^{M  - \overline{\frak e}}$. Thanks to the Hamiltonian structure one also gets that $\lambda \in \R$. 
\item{\it Reduction of the lower order terms.} 
Given $N > 0$, the next step is to transform the vector field $\ii {\cal V}_1(\vphi)$, ${\cal V}_1(\vphi) = \lambda |D|^M + {\cal W}_1(\vphi)$, ${\cal W}_1 \in OPS^{M - \overline{\frak e}}$ into another one of the form $\ii {\cal V}_N(\vphi)$, ${\cal V}_N(\vphi) = \lambda |D|^M + \mu_N(D) + OPS^{M - N \overline{\frak e}}$ where $\mu_N(D)$ is a time independent Fourier multiplier of order $M - \overline{\frak e}$. This is achieved in Section \ref{sezione riduzione ordini bassi M < 1}, by means of an iterative procedure. At the $n$-th step of such a procedure, we deal with $\ii {\cal V}_n(\vphi)$, ${\cal V}_n(\vphi) = \lambda |D|^M + \mu_n(D) + {\cal W}_n(\vphi)$ where $\mu_n(D)$ is a time independent Fourier multiplier of order $M - \overline{\frak e}$ and ${\cal W}_n(\vphi) = {\rm Op}\big( w_n(\vphi, x, \xi) \big)  \in OPS^{M - n \overline{\frak e}}$. First, in Section \ref{ordini bassi tempo M < 1}, we remove the $\vphi$-dependence from the symbol $w_n(\vphi, x, \xi)$ by conjugating $\ii {\cal V}_n(\vphi)$ with the time one flow map $\Phi_n^{(1)}(\vphi)$ of a Hamiltonian vector field $\ii {\cal G}_n^{(1)}(\vphi)$, ${\cal G}_n^{(1)}(\vphi) = {\rm Op}\big( g_n^{(1)}(\vphi, x, \xi) \big) \in OPS^{M - n \overline{\frak e}}$. The transformed vector field is given by 
$$
{\cal V}_n^{(1)}(\vphi) = \lambda |D|^M + \mu_n(D) + {\rm Op}\Big( w_n(\vphi, x, \xi) +  \omega \cdot \partial_\vphi g_{n}^{(1)}(\vphi, x, \xi) \Big) + OPS^{M - (n + 1) \overline{\frak e}}
$$
see \eqref{espansione tilde vn M < 1}. Since $\omega \in \R^\nu$ is a diophantine frequency, we choose the symbol $g_n^{(1)}$ so that 
$$
w_n(\vphi, x, \xi) +  \omega \cdot \partial_\vphi g_{n}^{(1)}(\vphi, x, \xi) = \langle w_n \rangle_\vphi (x, \xi), \quad \langle w_n \rangle_\vphi(x, \xi) := \frac{1}{(2 \pi)^\nu} \int_{\T^\nu} w_n(\vphi, x, \xi)\, d \vphi\,. 
$$
therefore we have removed the $\vphi$-dependence at the order $M - n \overline{\frak e}$ and 
$$
{\cal V}_n^{(1)}(\vphi) = \lambda |D|^M + \mu_n(D) + \langle w_n \rangle_\vphi(x, D) + OPS^{M - (n + 1) \overline{\frak e}}\,. 
$$
Then, in Section \ref{reduction lower orders space M < 1}, we remove the $x$-dependence from the symbol $\langle w_n \rangle_\vphi$, by conjugating the vector field $\ii {\cal V}_n^{(1)}(\vphi)$ with the time one flow map $\Phi_n^{(2)}$ of a Hamiltonian vector field $\ii{\cal G}_n^{(2)}$, ${\cal G}_n^{(2)} = {\rm Op}(g_n^{(2)}(x, \xi)) \in OPS^{1 - n \overline{\frak e}}$. The expansion of the transformed vector field $\ii {\cal V}_{n + 1}(\vphi)$ is given by 
$$
{\cal V}_{n + 1}(\vphi) = \Phi_n^{(2)} {\cal V}_n^{(1)}(\vphi) (\Phi_n^{(2)})^{- 1} = \lambda |D|^M + \mu_n(D) + {\rm Op}\Big( \langle w_n \rangle_\vphi + \{ g_n^{(2)}\,,\, \lambda |\xi|^M \} \Big) + OPS^{M - (n + 1) \overline{\frak e}}\,. 
$$
Then in Lemma \ref{lemma equazione omologica spazio lower orders M < 1}, we determine the symbol $g_n^{(2)}$ so that 
$$
\langle w_n \rangle_\vphi + \{ g_n^{(2)}\,,\, \lambda |\xi|^M \} = \langle w_n \rangle_{\vphi, x} + OPS^{M - (n + 1) \overline{\frak e}}, \quad \langle w_n \rangle_{\vphi, x} := \frac{1}{2 \pi} \int_\T \langle w_n  \rangle_\vphi(x, \xi)\, d x 
$$
implying that 
$$
{\cal V}_{n + 1}(\vphi) = \lambda |D|^M + \mu_{n + 1}(D) + OPS^{M - (n + 1) \overline{\frak e}}, \quad \mu_{n + 1}(D) := \mu_{n }(D) + {\rm Op}\big( \langle w_n \rangle_{\vphi, x}(\xi) \big)\,. 
$$
\end{enumerate}
\subsubsection{Normal form in the case $M = 1$}\label{idea forma normale M = 1}
The regularization procedure for the vector field $\ii{\cal V}(\vphi)$, ${\cal V}(\vphi) = V(\vphi, x) |D| + {\cal W}(\vphi)$, ${\cal W} \in OPS^{1 - \frak e}$, $\frak e > 0$, $V(\vphi, x) = 1 + \e P(\vphi, x)$ is developed in Section \ref{regolarizzazione caso M = 1}. In the following we will give a sketch of the proof of such a procedure. 
\begin{enumerate}
\item{\it Reduction of the highest order term. }
In order to reduce to constant coefficients the highest order term $\ii V(\vphi, x) |D|$, we conjugate the vector filed $\ii {\cal V}(\vphi)$ with the map 
$$
\Phi(\vphi ) := \Phi_+(\vphi)^{- 1} \Pi_+ + \Phi_-(\vphi)^{- 1} \Pi_- 
$$
where $\Pi_+ $ (resp. $\Pi_-$) are the projection operators on the positive (resp. negative) Fourier modes and $\Phi_{\pm}(\vphi)$ are the operators given by 
$$
\Phi_{\pm}(\vphi) : u (x) \mapsto \sqrt{1 + (\partial_x \alpha_\pm)(\vphi, x)}\,\, u(x + \alpha_\pm(\vphi, x))
$$
with $\alpha_+, \alpha_-$ being  ${\cal C}^\infty$ functions to be determined. The transformed vector field is $\ii {\cal V}_1(\vphi)$ with 
\begin{equation}\label{ii V0 intro}
\begin{aligned}
 {\cal V}_0(\vphi) & =   \Pi_+ \Big(\Big(  \big( 1 + \e P \big)\big( 1 + (\partial_y \widetilde \alpha_{+}) \big) - \omega \cdot \partial_\vphi \widetilde \alpha_{+} \Big)_{| y = x + \alpha_{\pm}(\vphi, x)}  |D| \Big) \Pi_+ \\
& \quad +  \Pi_- \Big(\Big(  \big( 1 + \e P  \big)\big( 1 + (\partial_y \widetilde \alpha_{-}) \big) + \omega \cdot \partial_\vphi \widetilde \alpha_{-} \Big)_{| y = x + \alpha_{\pm}(\vphi, x)}  |D| \Big) \Pi_- \\
& + OPS^{M - \overline{\frak e}}
\end{aligned}
\end{equation}
where the constant $\overline{\frak e} > 0$ is defined in \eqref{def bar e M = 1} and the map $y \mapsto y + \widetilde \alpha_\pm(\vphi, y)$ is the inverse diffeomorphism of $x \mapsto x + \alpha(\vphi, x)$. The reason for which we introduce the projectors $\Pi_+$ and $\Pi_-$ is the following: there are two terms which give a contribution to the highest order in the transformed vector field $\ii {\cal V}_0(\vphi)$. The ones coming from the conjugation $\Phi_\pm(\vphi) \ii {\cal V}(\vphi) \Phi_\pm(\vphi)^{- 1}$ are given by  $\ii \Big(  \big( 1 + \e P \big)\big( 1 + (\partial_y \widetilde \alpha_{\pm}) \big)  \Big)_{| y = x + \alpha_{\pm}(\vphi, x)}  |D|$ and the other ones coming from $\Phi_\pm(\vphi) \omega \cdot \partial_\vphi \big(\Phi_\pm(\vphi)^{- 1} \big)$ are $ \mp \Big( \omega \cdot \partial_\vphi \widetilde \alpha_\pm \Big)_{| y = x + \alpha_\pm(\vphi, x)} \partial_x$. Then, in order to reduce to constant coefficients the term of order one, we have to compare the operators $\ii |D|$ and $\partial_x$. These two operators are the same (up to a sign) if we project them to positive and negative Fourier modes, since they satisfy the elementary properties $\ii |D| \Pi_+ = \partial_x \Pi_+$, $ \ii |D| \Pi_- = - \partial_x \Pi_- $. 

\noindent
In order to reduce to constant coefficients the highest order term in \eqref{ii V0 intro}, we look for {\it small} functions $\widetilde \alpha_\pm \in {\cal C}^\infty$ and constants $\lambda_\pm \in \R$ close to $1$, so that 
$$
\big( 1 + \e P \big)\big( 1 + (\partial_y \widetilde \alpha_{\pm}) \big) \mp \omega \cdot \partial_\vphi \widetilde \alpha_{\pm} = \lambda_\pm\,. 
$$
This is a quasi-periodic transport equation, which is solved by applying Proposition \ref{proposizione principale trasporto} in the appendix. Note that this is the only point in which we require a smallness condition on the parameter $\e$. 

\item{\it Reduction of the lower order terms} : The expansion of the vector field $\ii {\cal V}_0(\vphi)$ is given in \eqref{forma finalissima cal V 0 M = 1} and it has the form ${\cal V}_0(\vphi) = \Pi_+ \big( \lambda_+ |D| + {\cal W}_{0, +}(\vphi) \big) \Pi_+ + \Pi_- \big( \lambda_- |D| + {\cal W}_{0, -}(\vphi) \big) \Pi_- + OPS^{- \infty}$ (see \eqref{forma finalissima cal V 0 M = 1}) where ${\cal W}_{0, \pm} \in OPS^{1 - \overline{\frak e}}$. Given $N > 0$, our next goal is to transform the vector field $\ii {\cal V}_0(\vphi)$ into another one $\ii {\cal V}_N(\vphi)$, which has the form 
${\cal V}_N(\vphi) = \Pi_+ \big( \lambda_+ |D| + \mu_{N, + }(D) \big) \Pi_+ + \Pi_- \big( \lambda_- |D| + \mu_{N, - }(D) \big) \Pi_- + OPS^{1 - N \overline{\frak e}} $ where $\mu_{N, \pm}(D) \in OPS^{1 - \overline{\frak e}}$ is a $\vphi$-independent Fourier multiplier with real symbol. This is achieved by means of an iterative procedure which is developed in Section \ref{sezione ordini bassi M = 1}. At the $n$-th step of such a procedure, we deal with a vector field $\ii {\cal V}_n(\vphi)$, 
${\cal V}_n(\vphi) = \Pi_+ {\cal V}_{n, +}(\vphi)  \Pi_+ + \Pi_-  {\cal V}_{n, -}(\vphi) \Pi_- + OPS^{- \infty}$, ${\cal V}_{n, \pm}(\vphi) = \lambda_\pm |D| + \mu_{n, \pm}(D) + {\cal W}_{n, \pm}(\vphi)$ where $\mu_{n, \pm}(D)$ are Fourier multipliers with real symbols of order $1 - \overline{\frak e}$ and ${\cal W}_{n, \pm}(\vphi) \in OPS^{1 - n \overline{\frak e}}$. The vector fields $\ii {\cal V}_{n, \pm}$ are Hamiltonian, i.e. ${\cal V}_{n, \pm}$ are self-adjoint operators. In order to reduce to constant coefficients the terms of order $1 - n \overline{\frak e}$, we conjugate the vector field $\ii {\cal V}_n(\vphi)$ by means of the map 
$$
\Phi_n(\vphi) = \Phi_{n, +}(\vphi)^{- 1} \Pi_+ + \Phi_{n, -}(\vphi)^{- 1} \Pi_-
$$
where $\Phi_{n, \pm}(\vphi)$ are the time one flow maps of Hamiltonian vector fields $\ii {\cal G}_{n, \pm}(\vphi) $ with 
${\cal G}_{n, \pm}(\vphi) = {\rm Op}\big( g_{n, \pm}(\vphi, x, \xi) \big) \in OPS^{1 - n \frak e}$. Note that, even if $\Phi_{n, \pm}(\vphi)$ is symplectic, the map $\Phi_n$ is not symplectic. By applying Lemma \ref{push forward splitting Pi +-}, one gets that the transformed vector field has the form $\ii {\cal V}_{n + 1}(\vphi)$ 
$$
{\cal V}_{n + 1}(\vphi) = \Pi_+ {\cal V}_{n + 1, +}(\vphi) \Pi_+ + \Pi_- {\cal V}_{n +1, -}(\vphi) \Pi_- + OPS^{- \infty}
$$
where $\ii {\cal V}_{n + 1, \pm} = (\Phi_{n, \pm}^{- 1})_{\omega*} (\ii {\cal V}_{n, \pm})$ are Hamiltonian vector fields. The symbols $v_{n + 1, \pm}$ of ${\cal V}_{n + 1, \pm}$ have the expansion 
$$
v_{n + 1, \pm} = \lambda_{\pm} |\xi |  + \mu_{n, \pm} + w_{n, \pm} + \omega \cdot \partial_\vphi g_{n, \pm} - \lambda_{\pm} \partial_x g_{n, \pm} {\rm sign}(\xi) + S^{1 - (n + 1) \overline{\frak e}}\,. 
$$
The order $1 - n \overline{\frak e}$ is reduced to constant coefficients in Lemma \ref{lemma equazione omologica ordini bassi M = 1}, by choosing the symbol $g_{n, \pm}$ so that 
$$
\begin{aligned}
& w_{n, \pm} + \omega \cdot \partial_\vphi g_{n, \pm} - \lambda_{\pm} \partial_x g_{n, \pm} {\rm sign}(\xi) = \langle w_n \rangle_{\vphi, x} + S^{1 - (n + 1) \overline{\frak e}}\,,  \\
&  \langle w_n \rangle_{\vphi, x}(\xi):= \frac{1}{(2 \pi)^{\nu + 1}} \int_{\T^{\nu + 1}} w_n(\vphi, x, \xi)\, d \vphi\, d x\,. 
\end{aligned}
$$
In order to make this choice, we impose some first order Melnikov conditions on $\omega$, i.e. we require that 
$$
|\omega \cdot \ell + \lambda_{\pm} \, j | \geq \frac{\gamma}{\langle \ell \rangle^\tau}\,, \quad \forall (\ell, j) \in \Z^{\nu + 1} \setminus \{ (0, 0) \}\,. 
$$
\end{enumerate}
Then $v_{n + 1, \pm} = \lambda_{\pm} |\xi |  + \mu_{n + 1, \pm} + S^{1 - (n + 1) \overline{\frak e}}$ with $\mu_{n + 1, \pm} = \mu_{n, \pm} + \langle w_{n, \pm}\rangle_{\vphi, x}$. Thanks to the Hamiltonian structure of  the vector fields $\ii{\cal V}_{n, \pm}$ one  gets that $\mu_{n + 1, \pm}(D)$ is a Fourier multiplier with real symbol. 

\smallskip

\noindent
The paper is organized as follows: in Section \ref{sezione pseudo diff operator} we provide some technical tools which are needed for the proof of Theorems \ref{teorema riduzione} and \ref{teorema riduzione M = 1}. In Sections \ref{sezione riduzione M < 1}, \ref{regolarizzazione caso M = 1}, we develop the normal form procedures described in Sections \ref{idea forma normale M < 1}, \ref{idea forma normale M = 1} and we prove Theorems \ref{teorema riduzione}, \ref{teorema riduzione M = 1}. Finally, in Section \ref{proof crescita} we prove Theorems \ref{teo growth of sobolev norms M < 1}, \ref{teo growth of sobolev norms M = 1}. 

\smallskip

\noindent
{\it Acknowledgements}. The author warmly thanks Giuseppe Genovese, Thomas Kappeler, Alberto Maspero and Michela Procesi for many useful discussions and comments.  

\section{Pseudo differential operators and flows of pseudo-PDEs}\label{sezione pseudo diff operator}
In this section, we recall some well-known definitions and results concerning pseudo differential operators on the torus $\T$. We always consider $\vphi$-dependent symbols $a(\vphi, x, \xi)$, $(\vphi, x, \xi) \in \T^\nu \times \T \times \R$, depending in a ${\cal C}^\infty$ way on the whole variables. 
Since $\vphi$ plays the role of a parameter, all the well known results apply to these symbols without any modification (we refer for instance to \cite{SV}, \cite{Taylor}). 
For the symbol class $S^m$ given in the definition \ref{classe simboli} and the operator class $OPS^m$ given in the definition \ref{def:Ps2}, the following standard inclusions hold: 
\begin{equation}\label{inclusioni Sm OPSm}
S^m \subseteq S^{m'}\,, \quad OPS^m \subseteq OPS^{m'}\,, \quad \forall m \leq m'\,. 
\end{equation}

\begin{theorem}[\bf Calderon-Vallancourt]\label{conitnuita pseudo}
Let $m \in \R$ and $A = a( \vphi, x, D) \in OPS^m$. Then for any $\alpha \in \N^\nu$ the operator $\partial_\vphi^\alpha A(\vphi) \in {\cal B}(H^{s + m}(\T), H^s(\T))$ with $\sup_{\vphi \in \T^\nu} \| \partial_\vphi^\alpha A(\vphi) \|_{{\cal B}(H^{s + m}, H^s)} < + \infty$. 
\end{theorem}
\begin{definition}[\bf Asymptotic expansion]\label{definizione espansione asintotica}
Let $(m_k)_{k \in \N}$ be a strictly decreasing sequence of real numbers converging to $- \infty$ and $a_k \in S^{m_k}$ for any $k \in \N$. We say that $a \in S^{m_0}$ has the asymptotic expansion $\sum_{k \geq 0} a_k$, i.e. $
a \sim \sum_{k \geq 0} a_k$ if $a - \sum_{k = 0}^{N} a_k \in S^{m_{N + 1}}$ for any $N \in \N$. 
\end{definition}
\begin{theorem}[{\bf Composition}]\label{teorema composizione pseudo}
Let $m, m' \in \R$ and $ A(\vphi) = a(\vphi, x, D) \in OPS^{m} $, $ B(\vphi) = b(\vphi, x, D) \in  OPS^{m'} $. Then the composition operator 
$ A(\vphi) B(\vphi) := A(\vphi) \circ B(\vphi) = {\rm Op}(\sigma_{AB})$ is a pseudo-differential operator in $OPS^{m + m'}$. The symbol $\sigma_{AB} $ has the following asymptotic expansion
\be\label{composition pseudo}
\sigma_{AB} \sim {\mathop \sum}_{\alpha \geq 0} \frac{1}{\ii^\alpha \alpha !} \partial_\xi^\alpha a  \partial_x^\alpha b \, , 
\ee
meaning that for any $ N \in \N$,  
\be\label{expansion symbol}
\s_{AB}  = \sum_{\b =0}^{N-1} \frac{1}{ \ii^\alpha \alpha !  }  \pa_\xi^\alpha a  \, \pa_x^\alpha b  +  S^{m + m' - N }  \, .
\ee
\end{theorem}
\begin{corollary}\label{corollario commutator}
Let $m, m' \in \R$ and let $A = {\rm Op}(a)$, $B = {\rm Op}(b)$. Then the commutator $[A, B] = {\rm Op}(a \star b)$, with $a \star b \in S^{m + m' - 1}$ having the following asymptotic expansion: 
$$
a \star b = - \ii \{a, b \} + S^{m + m' - 2}\,, \quad \{ a, b \} := \partial_\xi a \partial_x b - \partial_x a \partial_\xi b \in S^{m + m'-1} \,.
$$
\end{corollary}
\begin{theorem}[{\bf Adjoint of a pseudo-differential operator}]\label{adjoint}
If  $ A (\vphi)=  a(\vphi, x, D)  \in OPS^m $ is a pseudo-differential operator with symbol $ a \in S^m $, then its $L^2$-adjoint is a pseudo-differential operator $A^* = {\rm Op}(a^*) \in OPS^m$. 
The symbol $a^* \in S^m$ admits the asymptotic expansion 
\begin{equation}\label{espansione asintotica aggiunto}
a^*\sim \sum_{\alpha \in \N} \frac{1}{\ii^\alpha \alpha !}  \overline{\partial_x^\alpha \partial_\xi^\alpha a}
\end{equation}
meaning that for any $N \in \N$, 
$$
a^* = \sum_{\alpha = 0}^{N - 1} \frac{1}{\ii^\alpha \alpha !}  \overline{\partial_x^\alpha \partial_\xi^\alpha a} + S^{m - N}\, .
$$
\end{theorem}
\begin{lemma}\label{simbolo autoaggiunto fourier multiplier}
Let $A = {\rm Op}(a) \in OPS^m$ be self-adjoint, i.e. $a = a^* $ and let $\vphi(\xi)$ be a real Fourier multiplier of order $m'$. We define the symbol $b(\vphi, x, \xi) := \vphi(\xi) a(\vphi, x, \xi) \in S^{m + m'}$. Then $b^* =  b  +  S^{m + m' - 1}$. 
\end{lemma}
\begin{proof}
See Lemma 2.6 in \cite{Montalto1}.
\end{proof}
\begin{lemma}\label{commutatore Fourier multiplier derivata 0 simbolo generale}
Let $a\in S^m$, $m \in \R$ and let $g(\xi)$ be a Fourier multiplier in $S^0$ satisfying the following property: There exists $\delta > 0$ such that $\partial_\xi g(\xi) = 0$ for any $|\xi| \geq \delta$. Then the commutator $[ {\rm Op}(g), {\rm Op}(a)] \in OPS^{- \infty}$. 
\end{lemma}
\begin{proof}
By applying Theorem \ref{teorema composizione pseudo} the symbol $\sigma(\vphi, x, \xi)$ of the commutator $[ {\rm Op}(g), {\rm Op}(a)]\in OPS^{- \infty}$ has the asymptotic expansion 
$
\sigma \sim \sum_{\alpha \geq 1} \frac{1}{\ii^\alpha \alpha!} (\partial_\xi^\alpha g)(\partial_x^\alpha a)\,.
$
Since for any $\alpha \geq 1$, $\partial_\xi^\alpha g = 0$ for $|\xi| \geq \delta$, all the symbols $(\partial_\xi^\alpha g)(\partial_x^\alpha a) \in S^{- \infty}$, implying that $\sigma \in S^{- \infty}$.  
\end{proof}
We define the operator $\partial_x^{- 1}$ by setting 
 \begin{equation}\label{definizione partial x - 1}
 \partial_x^{- 1}[1] := 0\,, \quad \partial_x^{- 1}[e^{\ii x k  }] := \frac{e^{\ii x k }}{\ii k}\,, \quad \forall k \in  \Z \setminus \{ 0 \}
 \end{equation}
 and for any diophantine frequency vector $\omega \in DC(\gamma, \tau)$, we define the operator $(\omega \cdot \partial_\vphi)^{- 1}$ by setting 
 \begin{equation}\label{omega partial vphi - 1}
 (\omega \cdot \partial_\vphi)^{- 1}[1] = 0\,, \quad (\omega \cdot \partial_\vphi)^{- 1}[e^{\ii \ell \cdot \vphi}] = \frac{e^{\ii \ell \cdot \vphi}}{\ii \omega \cdot \ell}\,, \quad \forall \ell \in \Z^\nu \setminus \{ 0 \}\,. 
 \end{equation}
 Given $\omega \in \R^\nu$ and $\lambda \in \R$ satisfying the non-resonance condition 
 \begin{equation}\label{prima di mel sez pseudo}
 |\omega \cdot \ell + \lambda \, j| \geq \frac{\gamma}{\langle \ell \rangle^\tau}\,, \quad \forall (\ell, j ) \in \Z^{\nu + 1} \setminus \{(0,0) \}\,,
 \end{equation}
 we define the operator $(\omega \cdot \partial_\vphi + \lambda \partial_x)^{- 1}$ by setting 
 \begin{equation}\label{om lambda - 1}
 (\omega \cdot \partial_\vphi + \lambda \partial_x)^{- 1}[1] = 0\,, \quad (\omega \cdot \partial_\vphi + \lambda \partial_x)^{- 1}[e^{\ii \ell \cdot \vphi} e^{\ii j x}] := \frac{e^{\ii \ell \cdot \vphi} e^{\ii j x}}{\ii (\omega \cdot \ell + \lambda \, j)}\,, \quad \forall (\ell, j) \neq (0, 0)\,. 
 \end{equation}
 Furthermore, given a symbol $a \in S^m$, we define the averaged symbols $\langle a \rangle_\vphi, \langle a \rangle_{\vphi, x}$ by 
 \begin{equation}\label{simbolo mediato}
 \begin{aligned}
 \langle a \rangle_\vphi(x , \xi) & := \frac{1}{(2 \pi)^\nu} \int_{\T^\nu} a(\vphi, x, \xi)\, d \vphi \,, \quad \langle a \rangle_{\vphi, x}( \xi) & := \frac{1}{2 \pi} \int_\T  \langle a \rangle_\vphi(x , \xi)\, d x\,. 
 \end{aligned}
 \end{equation}
The following elementary properties hold: 
\begin{equation}\label{proprieta simbolo mediato e partial x - 1}
a \in S^m \Longrightarrow \partial_x^{- 1} a \,,\, (\omega \cdot \partial_\vphi)^{- 1} a\,,\,(\omega \cdot \partial_\vphi + \lambda \partial_x)^{- 1} a\,, \langle a \rangle_\vphi\,,\, \langle a \rangle_{\vphi, x} \in S^m\,.
\end{equation}
\begin{lemma}\label{proprieta aggiunti small divisors}
Given a symbol $a \in S^m$, the following holds. 

\noindent
$(i)$ $\langle a^* \rangle_\vphi = \big( \langle a \rangle_\vphi \big)^*$, $\langle a^* \rangle_{\vphi, x} = \overline{\langle a \rangle_{\vphi, x}} = \big( \langle a \rangle_{\vphi, x}\big)^*$. 

\noindent
$(ii)$ $\partial_x^{- 1}(a^*) = (\partial_x^{- 1} a )^*$. 

\noindent
$(iii)$ If $\omega \in DC(\gamma, \tau)$ then $(\omega \cdot \partial_\vphi)^{- 1} a^* = \Big( (\omega \cdot \partial_\vphi )^{- 1} a \Big)^*$. 

\noindent
$(iv)$ If $\omega$ satisfies the condition \eqref{prima di mel sez pseudo} then $(\omega \cdot \partial_\vphi + \lambda \partial_x)^{- 1} a^* = \Big( (\omega \cdot \partial_\vphi + \lambda \partial_x)^{- 1} a \Big)^*$. 
\end{lemma}
\begin{proof}
We prove item $(iv)$. The proof of items $(i)-(iii)$ can be done arguing similarly. The symbol $a^*$ is given by 
$$
a^*(\vphi, x, \xi) := \overline{\sum_{\eta \in \Z} \widehat a(\vphi, \eta, \xi - \eta) e^{\ii \eta x }}, \quad \widehat a(\vphi, \eta, \xi - \eta) := \frac{1}{2 \pi} \int_\T a(\vphi, x, \xi - \eta) e^{- \ii x \eta }\, dx\,. 
$$
Writing also the Fourier expansion w.r. to $\vphi \in \T^\nu$ one gets that $\widehat a(\vphi, \eta, \xi - \eta) = \sum_{\ell \in \Z^\nu} \widehat a_\ell( \eta, \xi - \eta) e^{\ii \ell \cdot \vphi} $ and by recalling \eqref{om lambda - 1}, one has  
\begin{align}
(\omega \cdot \partial_\vphi + \lambda \partial_x)^{- 1} a^* & = \overline{ \sum_{(\ell, \eta) \neq (0, 0) } \frac{\widehat a_\ell( \eta, \xi - \eta)}{\ii (\omega \cdot \ell + \lambda \eta)} e^{\ii (\ell \cdot \vphi + \eta x) }}\,. \label{forfora 0}
\end{align}
Now let $\sigma = (\omega \cdot \partial_\vphi + \lambda \partial_x)^{- 1} a $. One has that 
\begin{align}
\sigma^* & = \overline{\sum_{\eta \in \Z} \widehat \sigma(\vphi, \eta, \xi - \eta) e^{\ii \eta x }} =  \overline{\sum_{(\ell, \eta) \in \Z^{\nu + 1}} \widehat \sigma_\ell( \eta, \xi - \eta) e^{\ii (\ell \cdot \vphi + \eta x) }} \label{forfora 1}
\end{align}
and using again \eqref{om lambda - 1}, one has that $\widehat \sigma_0(0, \xi - \eta) = 0$ and $\widehat \sigma_\ell(\eta, \xi - \eta) = \frac{\widehat a_\ell(\eta, \xi - \eta)}{\ii (\omega \cdot \ell + \lambda \eta)}$ for any $(\ell, \eta) \in \Z^{\nu + 1} \setminus \{(0, 0)\}$, implying that $\eqref{forfora 0}$ coincides with \eqref{forfora 1}. 
\end{proof}

For any $\alpha \in \R$, we define the operator $|D|^\alpha$ as follows. Let $\chi \in {\cal C}^\infty(\R, \R)$ be a cut-off function satisfying 
\begin{equation}\label{definizione cut off D alpha}
\chi(\xi) := \begin{cases}
1 & \quad \text{if} \quad |\xi| \geq 1 \\
0 & \quad \text{if} \quad |\xi| \leq \frac12\,.
\end{cases}
\end{equation}
We then define for any $\alpha \in \R$
\begin{equation}\label{definizione D alpha}
|D|^\alpha := {\rm Op}\big( |\xi|^\alpha \chi(\xi) \big)\,.
\end{equation}
Clearly $|D|^\alpha \in OPS^\alpha$ and the action on any $2\pi$-periodic function $u \in L^2(\T)$ is given by 
$$
|D|^\alpha u(x) = \sum_{\xi \in \Z \setminus \{ 0 \} } |\xi|^\alpha \widehat u(\xi) e^{\ii x \xi}\,. 
$$
\subsection{Well posedness of some linear PDEs}
In this section we state some classical properties of the flow of some linear pseudo-PDEs. Before to give the statement of the following Lemma, we introduce the following notation. For any $s \in \R$, we write $A \lesssim_s B$ if there exists a constant $C(s) > 0$ depending on $s$ such that $A \leq C(s) B$. 
\begin{lemma}\label{flusso + derivate flusso}
Let ${\cal A}(\tau; \vphi) := {\rm Op}\Big( a(\tau;\vphi,  x, \xi) \Big)$, $\tau \in [0, 1]$ be a smooth $\tau$-dependent family of pseudo differential operators in $OPS^1$.  Assume that ${\cal A}(\tau; \vphi) + {\cal A}(\tau; \vphi)^* \in OPS^{0}$. Then the following holds. 

\noindent
$(i)$ Let $s \geq 0$, $u_0 \in H^s(\T)$, $\tau_0 \in [0, 1]$. Then there exists a unique solution $u \in {\cal C}^0_b\Big([0, 1], H^s(\T) \Big)$ of the Cauchy problem 
\begin{equation}\label{rafael}
\begin{cases}
\partial_\tau u = {\cal A}(\tau; \vphi)[ u ] \\
u(\tau_0; x) = u_0(x)
\end{cases}
\end{equation}
satisfying the estimate 
$$
\| u\|_{{\cal C}^0([0, 1], H^s)} \lesssim_s \| u_0\|_{H^s} \,.
$$
As a consequence, for any $\tau_0, \tau \in [0, 1]$, the flow map $\Phi(\tau_0, \tau; \vphi)$, which maps the initial datum $u(\tau_0 ) = u_0$ into the solution $u(\tau)$ of \eqref{rafael} at the time $\tau$, is in ${\cal B}(H^s)$ with $\sup_{\begin{subarray}{c}
\tau_0, \tau \in [0, 1] \\
\vphi \in \T^\nu
\end{subarray}} \| \Phi(\tau_0, \tau; \vphi)\|_{{\cal B}(H^s)} < + \infty$ for any $s \geq 0$. Moreover, the operator $\Phi(\tau_0, \tau; \vphi )$ is invertible with inverse $\Phi(\tau_0, \tau; \vphi)^{- 1} = \Phi(\tau, \tau_0; \vphi)$. 

\medskip

\noindent
$(ii)$ For any $\tau_0, \tau \in [0, 1]$,  the flow map $ \vphi \mapsto \Phi(\tau_0, \tau; \vphi)$
is differentiable and 
\begin{equation}\label{rafael - 10}
 \sup_{\begin{subarray}{c}
\tau_0, \tau \in [0, 1] \\
\vphi \in \T^\nu
\end{subarray}} \| \partial_\vphi^\alpha \Phi(\tau_0, \tau; \vphi)\|_{{\cal B}(H^{s + |\alpha|}, H^{s})} < + \infty\,, \quad \forall \alpha \in \N^\nu, \quad s \geq 0\,.  
\end{equation}
\end{lemma}
\begin{proof}
The proof of item $(i)$ is classical. We refer for instance to \cite{Taylor}, Section 0.8. The proof of item $(ii)$ can be obtained arguing as in Lemma 2.9 in \cite{Montalto1}.  
\end{proof}
\begin{lemma}\label{commutatore flusso Fourier multiplier costante}
Let $g (\xi)$ be a Fourier multiplier in $OPS^0$ which satisfies the following property: there exists $\delta > 0$ such that $\partial_\xi g(\xi) = 0$ for any $|\xi| \geq \delta$. Let ${\cal A}(\tau; \vphi) := {\rm Op}\Big( a(\tau;\vphi,  x, \xi) \Big)$, $\vphi \in \T^\nu$, $\tau \in [0, 1]$ be a smooth $\tau$-dependent family of periodic  pseudo differential operators in $OPS^\eta$ with $\eta \leq 1$.  Assume that ${\cal A}(\tau; \vphi) + {\cal A}(\tau; \vphi)^* \in OPS^{\eta - 1}$. Let $\Phi(\tau; \vphi)$ be the flow of the pseudo PDE $\partial_\tau u = {\cal A}(\tau; \vphi)[u]$, i.e. 
$$
\begin{cases}
\partial_\tau \Phi(\tau; \vphi) = {\cal A}(\tau; \vphi) \Phi(\tau; \vphi) \\
\Phi(0; \vphi) = {\rm Id}\,. 
\end{cases}
$$
Then the commutator $[\Phi(\tau; \vphi), {\rm Op}(g)] \in OPS^{- \infty}$. 
\end{lemma}
\begin{proof}
Let $\Phi_g(\tau; \vphi) := [\Phi(\tau; \vphi), {\rm Op}(g)]$. A direct calculation shows that $\Phi_g(\tau; \vphi)$ solves 
$$
\begin{cases}
\partial_\tau \Phi_g(\tau; \vphi) = {\cal A}(\tau; \vphi) \Phi_g(\tau; \vphi) + {\cal R}_g(\tau; \vphi) \\
\Phi_g(0; \vphi) = 0\,,
\end{cases}\quad {\cal R}_g(\tau; \vphi) := [{\cal A}(\tau; \vphi), {\rm Op}(g)] \Phi(\tau; \vphi)\,. 
$$
By Lemma \ref{commutatore Fourier multiplier derivata 0 simbolo generale}, $[{\cal A}(\tau; \vphi), {\rm Op}(g)] \in OPS^{- \infty}$ and since ${\cal A}(\tau; \vphi)$ satisfies the hypothesis of Lemma \ref{flusso + derivate flusso}, the flow $\Phi(\tau; \vphi)$ satisfies
 \begin{equation}\label{torta di mele 0}
 \Phi(\tau; \vphi)^{\pm 1} \in {\cal B}(H^s(\T))\,, \quad \forall s \geq 0\,, 
 \end{equation}
  implying that the operator ${\cal R}_g(\tau; \vphi) = [{\cal A}(\tau; \vphi), {\rm Op}(g)] \Phi(\tau; \vphi) \in OPS^{- \infty}$.  
Using the Duhamel principle one gets that 
$$
\Phi_g(\tau; \vphi) = \int_0^\tau \Phi(\tau; \vphi)^{- 1} \Phi(\zeta; \vphi) {\cal R}_g(\zeta; \vphi) \, d \zeta \in OPS^{- \infty} 
$$
and the proof of the lemma is then concluded. 
\end{proof}

\subsection{Some Egorov-type theorems}\label{sezione astratta Egorov}
In this section we collect some abstract egorov type theorems, namely we study how a pseudo differential operator transforms under the action of the flow of a first order hyperbolic PDE. 
 Let $\alpha :  \T^\nu \times \T \to \R$ be a ${\cal C}^\infty$ function satisfying
  \begin{equation}\label{ansatz alpha egorov}
\alpha \in {\cal C}^\infty( \T^\nu \times \T, \R) \,,\quad \inf_{(\vphi, x) \in \T^{\nu + 1}} \big( 1 + \alpha_x(\vphi, x) \big) > 0\,. 
 \end{equation} 
 We then consider the non-autonomous transport equation
\begin{equation}\label{trasporto per Egorov}
\partial_\tau u = {\cal A}(\tau; \vphi) u\,, \qquad {\cal A}(\tau; \vphi)  := b_{\alpha}(\tau; \vphi, x) \partial_x + \frac{(\partial_x b_{\alpha})(\tau; \vphi, x)}{2}\,, \quad 
\end{equation}
\begin{equation}\label{definizione b alpha per egorov}
b_\alpha (\tau; \vphi, x):= - \frac{\alpha(\vphi, x)}{1 + \tau \alpha_x(\vphi, x)} \,, \qquad \tau \in [0, 1]\,. 
\end{equation}
Note that the condition \eqref{ansatz alpha egorov} implies that $
\inf_{\begin{subarray}{c}
\tau \in [0, 1] \\
(\vphi, x) \in \T^\nu \times \T
\end{subarray}} \big( 1+ \tau \alpha_x(\vphi, x) \big) > 0$, hence the function $b \in {\cal C}^\infty([0, 1] \times \T^\nu \times \T)$ and ${\cal A}(\tau; \cdot) \in OPS^1$, $\tau \in [0, 1]$ is a smooth family of pseudo-differential operators. It is straightforward to verify that ${\cal A}(\tau; \vphi) + {\cal A}(\tau; \vphi)^* = 0$, therefore, the hypotheses of Lemma \ref{flusso + derivate flusso} are verified, implying that, for any $\tau \in [0, 1]$, the flow $\Phi(\tau; \vphi) \equiv \Phi(0, \tau; \vphi)$, $\tau \in [0, 1]$ of the equation \eqref{trasporto per Egorov}, i.e. 
\begin{equation}\label{equazione flusso operatoriale}
\begin{cases}
\partial_\tau \Phi(\tau; \vphi) = {\cal A}(\tau; \vphi) \Phi(\tau; \vphi)  \\
\Phi (0; t) = {\rm Id}
\end{cases}
\end{equation}
is a well defined map and satisfies all the properties stated in the items $(i)$, $(ii)$ of Lemma \ref{flusso + derivate flusso}. Furthermore, arguing as in Section 2.2 of \cite{Montalto1}, the map $\Phi(\tau; \vphi)$ is symplectic. We then have the following 
\begin{lemma}\label{proprieta flusso trasporto lemma astratto}
The flow $\Phi(\tau; \vphi)$ given by \eqref{equazione flusso operatoriale} is a symplectic, invertible map satisfying 
$$
\sup_{\begin{subarray}{c}
\tau \in [0, 1] \\
\vphi \in \T^\nu
\end{subarray}}\| \partial_\vphi^\alpha \Phi(\tau; \vphi)^{\pm 1}\|_{{\cal B}(H^{s + |\alpha|}, H^{s})} < + \infty\,, \quad \forall \alpha \in \N^\nu, \quad s \geq 0\,.  
$$ 
\end{lemma}
In order to state Proposition \ref{Teorema egorov generale} of this section, we need some preliminary results. 
 \begin{lemma}\label{diffeo del toro lemma astratto}
 Let $\alpha \in {\cal C}^\infty(\T^\nu \times \T, \R)$ satisfy the condition \eqref{ansatz alpha egorov}. Then for any $\vphi \in \T^\nu$, the map 
 $$
\psi_\vphi : \T \to \T\,, \quad x \mapsto x + \alpha(\vphi, x)
 $$
 is a diffeomorphism of the torus whose inverse has the form 
 \begin{equation}\label{forma vphi t - 1}
\psi_\vphi^{- 1} : \T \to \T , \quad y \mapsto y + \widetilde \alpha(\vphi, y)\,,
 \end{equation}
 with $\widetilde \alpha : \T^\nu \times \T \to \R$ satisfying 
 \begin{equation}\label{proprieta alpha tilde diffeo toro}
\widetilde \alpha \in {\cal C}^\infty(\T^\nu \times \T, \R)\,, \quad \inf_{(\vphi, y) \in \T^\nu \times \T} \big(1 + \widetilde \alpha_y(\vphi, y) \big) > 0\,. 
 \end{equation}
 Furthermore, the following identities hold: 
 \begin{equation}\label{identita 1 + alpha alpha tilde partial}
 \begin{aligned}
1 + \alpha_x(\vphi, x) = \frac{1}{1 + \widetilde \alpha_y\big(\vphi, x + \alpha(\vphi, x)\big)}\,, \quad 1 + \widetilde \alpha_y(\vphi, y) = \frac{1}{1 + \alpha_x\big(\vphi, y + \widetilde \alpha(\vphi, y)\big)}
 \end{aligned}
 \end{equation}
 \end{lemma}
 \begin{proof}
 The proof is the same as the one of Lemma 2.12 in \cite{Montalto1}. 
  \end{proof}
Now, we are ready to state the following Proposition.
\begin{proposition}\label{Teorema egorov generale}
Let $m \in \R$, ${\cal V}(\vphi) = {\rm Op}\big( v(\vphi, x, \xi)\big)$ be in the class $S^m$ and $\Phi(\tau; \vphi)$, $\tau \in[0, 1]$ be the flow map of the PDE \eqref{equazione flusso operatoriale}. Then ${\cal P}(\tau; \vphi) := \Phi(\tau; \vphi) {\cal V}(\vphi) \Phi(\tau; \vphi)^{- 1}$ is a pseudo differential operator in the class $OPS^m$, i.e. ${\cal P}(\tau; \vphi) = {\rm Op}\big(  p(\tau; \vphi, x, \xi) \big)$ with $p(\tau, \cdot, \cdot, \cdot) \in S^m$, $\tau \in [0, 1]$. Furthermore $p(\tau; \vphi, x, \xi)$ admits the expansion 
$$
p(\tau; \vphi, x, \xi) = p_0(\tau; \vphi, x, \xi) + p_{\geq 1}(\tau; \vphi, x, \xi)\,, \qquad p_0(\tau, \cdot, \cdot, \cdot) \in S^m\,, \quad p_{\geq 1}(\tau; \cdot, \cdot, \cdot) \in S^{m - 1}
$$
and the principal symbol $p_0$ has the form 
$$
\begin{aligned}
& p_0(\tau; \vphi, x, \xi) := v\Big(\vphi, x + \tau\alpha(\vphi, x), (1 + \tau\alpha_x(\vphi, x))^{- 1} \xi \Big)\,, \\
&   \qquad \forall (\vphi, x, \xi) \in \T^\nu \times \T \times \R\,, \quad \forall \tau \in  [0, 1]\,. 
\end{aligned}
$$
\end{proposition}
\begin{proof}
The proof is the same as the one of Theorem 2.14 in \cite{Montalto1}. 
\end{proof}
We also state another {\it semplified} version of the Egorov theorem in which we conjugate a symbol by means of the flow of a vector field which is a pseudo differential operator of order strictly smaller than one. 
We consider a pseudo differential operator ${\cal G}(\vphi) = {\rm Op}(g(\vphi, x, \xi))$, with $g \in S^{\eta}$, ${\cal G}(\vphi) = {\cal G}(\vphi)^*$, $\eta < 1$ and for any $\tau \in [0, 1]$, let $\Phi_{{\cal G}}(\tau; \vphi)$ be the flow of the pseudo-PDE $\partial_\tau u = \ii {\cal G}(\vphi)[ u]$, i.e. 
\begin{equation}\label{flusso pseudo ordine < 1}
\begin{cases}
\partial_\tau \Phi_{{\cal G}}(\tau; \vphi) = \ii {\cal G}(\vphi) \Phi_{{\cal G}}(\tau; \vphi) \\
\Phi_{{\cal G}}(0; \vphi)  = {\rm Id}\,.
\end{cases}
\end{equation}
which is a well-defined invertible map by Lemma \ref{flusso + derivate flusso}. The following Proposition holds. 
\begin{proposition}\label{teorema egorov campo minore di uno}
Let $m \in \R$, ${\cal V}(\vphi) = {\rm Op}\big( v(\vphi, x, \xi) \big) \in OPS^m$ and ${\cal G}(\vphi) = {\rm Op}(g(\vphi, x, \xi))$, with $g \in S^{\eta}$, $\eta < 1$. Then for any $\tau \in [0, 1]$, the operator ${\cal P}(\tau; \vphi) := \Phi_{{\cal G}}(\tau; \vphi) {\cal V}(\vphi) \Phi_{{\cal G}}(\tau; \vphi)^{- 1}$ is a pseudo differential operator of order $m$ with symbol $p(\tau; \cdot , \cdot, \cdot ) \in S^m$. The symbol $p(\tau; \vphi , x, \xi )$ admits the expansion 
\begin{equation}\label{espansione lemma egorov semplificato}
p(\tau; \vphi , x, \xi ) = v(\vphi, x, \xi) + \tau \{ g, v \}(\vphi, x, \xi) + p_{\geq 2}(\tau; \vphi, x, \xi)\,, \quad p_{\geq 2}(\tau; \vphi, x, \xi) \in S^{m - 2(1 - \eta)}\,. 
\end{equation}
\end{proposition}
\begin{proof}
 The proof is the same as the one of Theorem 2.16 in \cite{Montalto1}. 
 \end{proof}
We now consider the projection operators $\Pi_+, \Pi_- : L^2(\T) \to L^2(\T)$ given by 
\begin{equation}\label{operatori proiezione sui modi positivi negativi azione}
\Pi_{+}u (x) := \sum_{\xi \geq 0} \widehat u(\xi) e^{\ii x \xi}\,, \quad \Pi_{- } u(x) := \sum_{\xi < 0} \widehat u(\xi) e^{\ii x \xi}\,, \quad u \in L^2(\T)\,. 
\end{equation}
The following elementary properties hold: 
\begin{equation}\label{proprieta elementari Pi + -}
\Pi_+ + \Pi_- = {\rm Id}\,, \quad \Pi_{\pm}^2 = \Pi_{\pm}\,, \quad \Pi_{\pm} \Pi_{\mp} = 0\,. 
\end{equation}
Given two cut-off functions $\chi_{\pm} \in {\cal C}^\infty(\R, \R)$ satisfying 
\begin{equation}\label{propriet cut-off chi+}
\chi_+(\xi) = \begin{cases}
 1\,, \quad \forall \xi \geq 0 \\
 0 \,, \quad \forall \xi \leq - \frac12, \\
\end{cases} \quad \chi_-(\xi) = \begin{cases}
0\,, \quad \forall \xi \geq - \frac23 \\
1 \,, \quad \forall \xi \leq - 1, 
\end{cases}
\end{equation}
we often identify the operators $\Pi_{\pm}$ with ${\rm Op}(\chi_{\pm}(\xi))$, i.e. 
\begin{equation}\label{operatori proiezione sui modi positivi negativi}
\Pi_{\pm} \equiv {\rm Op}(\chi_{\pm}(\xi))\,. 
\end{equation}
Lemmata \ref{push forward splitting Pi +-}-\ref{hilbert transform cambi di variabile} below will be used in order to develop the reduction procedure of Section \ref{regolarizzazione caso M = 1}. 
\begin{lemma}\label{push forward splitting Pi +-}
Let ${\cal V}(\vphi) = {\rm Op}\big( v(\vphi, x, \xi) \big) \in OPS^m$. Let ${\cal G}_{\pm}(\tau; \vphi) = {\rm Op} \big( g_{\pm}(\tau; \vphi, x, \xi) \big) \in OPS^\eta$, $\eta \leq 1$ satisfy ${\cal G}_{\pm}(\tau; \vphi) + {\cal G}_{\pm}(\tau; \vphi)^* \in OPS^{ \eta - 1}$, $\tau \in [0, 1]$. Let $\Phi_{\pm}(\tau; \vphi)$ by the flow associated to the vector field ${\cal G}_\pm(\tau; \vphi)$, i.e. 
$$
\begin{cases}
\partial_\tau \Phi_{\pm}(\tau; \vphi) = {\cal G}_{\pm}(\tau; \vphi) \Phi_{\pm}(\tau; \vphi) \\
\Phi_{\pm}(0; \vphi) = {\rm Id}\,. 
\end{cases}
$$ 
Then the following holds.

\noindent
$(i)$
The map 
\begin{equation}\label{Phi Pi + - lemma astratto}
\Phi(\tau; \vphi) := \Phi_+(\tau; \vphi)^{- 1} \Pi_+ + \Phi_-(\tau; \vphi)^{- 1} \Pi_- 
\end{equation}
is invertible, with inverse given by 
\begin{equation}\label{Phi Pi + - lemma astratto inverso}
\Phi(\tau; \vphi)^{- 1} = \Pi_+ \Phi_+(\tau; \vphi) + \Pi_- \Phi_-(\tau; \vphi)\,. 
\end{equation}
Furthermore the quasi-periodic push forward of the vector field ${\cal V}(\vphi)$ by means of the map $\Phi$ has the form 
\begin{equation}\label{Phi pm cal V lemma astratto}
{\cal V}_1(\tau; \vphi) := \Phi_{\omega *} {\cal V}(\tau; \vphi) = \Pi_+ {\cal V}_{1, +}(\tau; \vphi) \Pi_+ + \Pi_- {\cal V}_{1, -}(\tau; \vphi) \Pi_- + OPS^{- \infty}
\end{equation}
where 
\begin{equation}\label{campo trasformato Pi + - lemma astratto}
\begin{aligned}
& {\cal V}_{1, \pm}(\tau; \vphi) := (\Phi_{\pm}^{- 1})_{\omega *} {\cal V}(\vphi) = \Phi_{\pm}(\tau; \vphi) {\cal V}(\vphi) \Phi_{\pm}(\tau; \vphi)^{- 1} - \Phi_{\pm}(\tau; \vphi) \omega \cdot \partial_\vphi \Phi_{\pm}(\tau; \vphi)^{- 1}\,. 
\end{aligned}
\end{equation}
$(ii)$ Assume that ${\cal V}(\vphi) = \Pi_+ {\cal V}_+(\vphi) \Pi_+ + \Pi_- {\cal V}_-(\vphi) \Pi_-$ where ${\cal V}_+(\vphi), {\cal V}_-(\vphi) \in OPS^m$. 
Then ${\cal V}_1(\tau; \vphi) = \Phi_{\omega*} {\cal V}(\tau; \vphi)$ has the form 
\begin{equation}\label{cal V1 split pm}
\begin{aligned}
& {\cal V}_1(\tau; \vphi) = \Pi_+ {\cal V}_{1, +}(\tau; \vphi) \Pi_+ + \Pi_- {\cal V}_{1, -}(\tau; \vphi) \Pi_- + OPS^{- \infty}\,, \\
& {\cal V}_{1, \pm}(\tau; \vphi) := \Phi_{\pm}(\tau; \vphi) {\cal V}_{\pm}(\vphi) \Phi_{\pm}(\tau; \vphi)^{- 1} - \Phi_{\pm}(\tau; \vphi) \omega \cdot \partial_\vphi \Phi_{\pm}(\tau; \vphi)^{- 1}\,.  \\
\end{aligned}
\end{equation}
\end{lemma}
\begin{proof}
{\sc Proof of $(i)$}
By the hypotheses on ${\cal G}_{\pm}(\tau; \vphi)$, we can apply Lemma \ref{flusso + derivate flusso} obtaining the the flow $\Phi_{\pm}(\tau; \vphi)$ is invertible and for any $s \geq 0$
\begin{equation}\label{Phi pm bound nel lemma astratto}
\Phi_{\pm}(\tau; \vphi), \Phi_{\pm}(\tau; \vphi)^{1} \in {\cal B}(H^s)\,, \quad \sup_{\begin{subarray}{c}
\tau \in [0, 1] \\
\vphi \in \T^\nu
\end{subarray}}\| \Phi_{\pm}(\tau; \vphi) \|_{{\cal B}(H^s)}, \| \Phi_{\pm}(\tau; \vphi)^{- 1} \|_{{\cal B}(H^s)} < + \infty \,.
\end{equation} 
 By the properties \eqref{proprieta elementari Pi + -}, one can verify that the map $\Phi(\tau; \vphi)$ is invertible and the inverse is given by the formula \eqref{Phi Pi + - lemma astratto}.  
We now compute the push forward 
\begin{equation}\label{primo cal V0 M = 1}
{\cal V}_1(\tau; \vphi) = \Phi(\tau; \vphi)^{- 1} {\cal V}(\vphi) \Phi(\tau; \vphi) - \Phi(\tau; \vphi)^{- 1} \omega \cdot \partial_\vphi \Phi(\tau; \vphi)\,. 
\end{equation}
By \eqref{Phi Pi + - lemma astratto}, \eqref{Phi Pi + - lemma astratto inverso}, one has 
\begin{align}
\Phi(\tau; \vphi)^{- 1} {\cal V}(\vphi) \Phi(\tau; \vphi) & = \Big( \Pi_+ \Phi_+(\tau; \vphi) + \Pi_{- } \Phi_-(\tau;  \vphi) \Big) {\cal V}(\vphi) \Big( \Phi_+( \tau; \vphi)^{- 1} \Pi_+ + \Phi_{-}(\tau; \vphi)^{- 1} \Pi_- \Big) \nonumber\\
& = \Pi_+ \Phi_+( \tau; \vphi) {\cal V}(\vphi) \Phi_+( \tau; \vphi)^{- 1} \Pi_+ + \Pi_+ \Phi_+(\tau;  \vphi) {\cal V}(\vphi)  \Phi_{-}( \tau; \vphi)^{- 1} \Pi_- \nonumber\\
& \quad + \Pi_{- } \Phi_-(\tau;  \vphi){\cal V}(\vphi) \Phi_+( \tau; \vphi)^{- 1} \Pi_+ + \Pi_{- } \Phi_-( \tau; \vphi) {\cal V}(\vphi) \Phi_{-}( \tau; \vphi)^{- 1} \Pi_- \nonumber\\
& = \Pi_+ \Phi_+( \tau; \vphi) {\cal V}(\vphi) \Phi_+( \tau; \vphi)^{- 1} \Pi_+ +  \Pi_- \Phi_-( \tau; \vphi) {\cal V}(\vphi) \Phi_{-}( \tau; \vphi)^{- 1} \Pi_- + {\cal R}_{ \infty}^{(1)} (\tau; \vphi)
\end{align}
where, using the properties \eqref{proprieta elementari Pi + -}, one has that 
\begin{equation}\label{definizione cal R0 infinito}
\begin{aligned}
{\cal R}_{ \infty}^{(1)}(\tau; \vphi) & := 
 [\Pi_+ , \Phi_+( \tau; \vphi) {\cal V}(\vphi)  \Phi_{-}(\tau; \vphi)^{- 1}] \Pi_- + [\Pi_-, \Phi_-( \tau; \vphi){\cal V}(\vphi) \Phi_+( \tau; \vphi)^{- 1} ] \Pi_+ \,. 
\end{aligned}
\end{equation}
Moreover, using again \eqref{proprieta elementari Pi + -}, one gets 
\begin{align}
 \Phi(\tau; \vphi)^{- 1} \omega \cdot \partial_\vphi \Phi(\tau; \vphi) & = \Big( \Pi_+ \Phi_+(\tau; \vphi) + \Pi_{- } \Phi_-(\tau;  \vphi) \Big)\Big( \omega \cdot \partial_\vphi \Phi_+( \tau; \vphi)^{- 1} \Pi_+ +  \omega \cdot \partial_\vphi \Phi_{-}(\tau;  \vphi)^{- 1} \Pi_- \Big) \nonumber\\
& = \Pi_+ \Phi_+(\tau; \vphi) \omega \cdot \partial_\vphi \Phi_+( \tau; \vphi)^{- 1}   \Pi_+ +  \Pi_-  \Phi_-( \tau; \vphi) \omega \cdot \partial_\vphi \Phi_{-}( \tau; \vphi)^{- 1}  \Pi_- \nonumber\\
& \quad + {\cal R}_{\infty}^{(2)}(\tau; \vphi) 
\end{align}
where 
\begin{equation}\label{cal R 0 infty (2) M = 1}
\begin{aligned}
{\cal R}_{ \infty}^{(2)}(\tau; \vphi)  & :=  [\Pi_+ , \Phi_+(\tau; \vphi)\omega \cdot \partial_\vphi \Phi_-(\tau; \vphi)^{- 1}] \Pi_- +  [\Pi_- , \Phi_-(\tau; \vphi)\omega \cdot \partial_\vphi \Phi_+(\tau; \vphi)^{- 1}] \Pi_+\,.
\end{aligned}
\end{equation}
Thus \eqref{primo cal V0 M = 1}-\eqref{cal R 0 infty (2) M = 1} imply that 
\begin{equation}\label{espansion cal V + - 0}
{\cal V}_1(\tau; \vphi) = \Pi_+ {\cal V}_{ +}(\tau; \vphi) \Pi_+ + \Pi_- {\cal V}_{ -}(\tau; \vphi) \Pi_- + {\cal R}_{\infty}(\tau; \vphi)
\end{equation}
where 
\begin{equation}\label{definizioni cal V+ - R 0 infty}
\begin{aligned}
& {\cal V}_{ \pm}(\tau; \vphi)  := \Phi_\pm( \tau; \vphi) {\cal V}(\vphi) \Phi_\pm( \tau; \vphi)^{- 1} -  \Phi_\pm(\tau; \vphi) \omega \cdot \partial_\vphi \Phi_\pm( \tau; \vphi)^{- 1}  \,,\\
& {\cal R}_{ \infty}(\tau; \vphi) := {\cal R}_{ \infty}^{(1)}(\tau; \vphi) - {\cal R}_{ \infty}^{(2)}(\tau; \vphi) \,. 
\end{aligned}
\end{equation}
It remains only to prove that the operator ${\cal R}_\infty(\tau; \vphi)$ satisfies the property \eqref{campo trasformato Pi + - lemma astratto}. All the terms in \eqref{definizione cal R0 infinito}, \eqref{cal R 0 infty (2) M = 1} can be analyzed in the same way, then we consider only the operator 
\begin{equation}\label{analisi solo un pezzo cal R infty}
{\cal R}(\tau; \vphi) :=  [\Pi_+ , \Phi_+( \tau; \vphi) {\cal V}(\vphi)  \Phi_{-}(\tau; \vphi)^{- 1}] \Pi_- \stackrel{\eqref{proprieta elementari Pi + -}}{=} \Pi_+ \Phi_+( \tau; \vphi) {\cal V}(\vphi)  \Phi_{-}(\tau; \vphi)^{- 1} \Pi_-  \,. 
\end{equation}
Uisng the property \eqref{proprieta elementari Pi + -}, we write 
\begin{align}
{\cal R}(\tau; \vphi) & = [\Pi_+ , \Phi_+(\tau; \vphi)] {\cal V}(\vphi) \Phi_-(\tau; \vphi)^{- 1} + \Phi_+(\tau; \vphi) [\Pi_+, {\cal V}(\vphi)] \Phi_-(\tau; \vphi)^{- 1} \Pi_- \nonumber\\
& \quad + \Phi_+(\tau; \vphi) {\cal V}(\vphi) [\Pi_+ , \Phi_-(\tau; \vphi)^{- 1}] \Pi_- \,. \label{peg pippo max}
\end{align}
By \eqref{operatori proiezione sui modi positivi negativi}, $\Pi_{\pm} \equiv {\rm Op}(\chi_{\pm}(\xi))$. Furthermore \eqref{propriet cut-off chi+}  implies that $\partial_\xi \chi_{\pm}(\xi) = 0$ for any $|\xi| \geq 1$. Therefore we can apply Lemma \ref{commutatore flusso Fourier multiplier costante} to the operators $[\Pi_+ , \Phi_+(\tau; \vphi)]$, $[\Pi_+ , \Phi_-(\tau; \vphi)^{- 1}]$ and Lemma \ref{commutatore Fourier multiplier derivata 0 simbolo generale} to the operator $[\Pi_+, {\cal V}(\vphi)]$, which together with \eqref{Phi pm bound nel lemma astratto} imply that ${\cal R}(\tau; \vphi) \in OPS^{- \infty}$.

\noindent
{\sc Proof of $(ii)$.} Applying item $(i)$ to the operator ${\cal V}(\vphi) = \Pi_+ {\cal V}_+(\vphi) \Pi_+ + \Pi_- {\cal V}_-(\vphi) \Pi_-$, one gets that 
\begin{equation}\label{prima forma cal V1 ii lemma astratto}
\begin{aligned}
& {\cal V}_1(\tau; \vphi) = \Phi_{\omega*} {\cal V}(\tau; \vphi) = \Pi_+ {\cal P}_+(\tau; \vphi) \Pi_+ + \Pi_- {\cal P}_-(\tau; \vphi) \Pi_- + OPS^{- \infty}\,, \\
& {\cal P}_{\pm}(\tau; \vphi)  := \Phi_{\pm}(\tau; \vphi) {\cal V}(\vphi) \Phi_{\pm}(\tau; \vphi)^{- 1} -  \Phi_{\pm}(\tau; \vphi) \omega \cdot \partial_\vphi \Phi_{\pm}(\tau; \vphi)^{- 1} \,. 
\end{aligned}
\end{equation}
Using that ${\cal V} = \Pi_+ {\cal V}_+ \Pi_+ + \Pi_- {\cal V}_- \Pi_-$ one gets 
\begin{align}
\Pi_+ \Phi_{+}(\tau; \vphi) {\cal V}(\vphi) \Phi_{+}(\tau; \vphi)^{- 1} \Pi_+ & = \Pi_+ \Phi_{+}(\tau; \vphi) \Pi_+ {\cal V}_+(\vphi) \Pi_+  \Phi_{+}(\tau; \vphi)^{- 1} \Pi_+ \nonumber\\
& \quad + \Pi_+ \Phi_{+}(\tau; \vphi) \Pi_- {\cal V}_-(\vphi) \Pi_- \Phi_{+}(\tau; \vphi)^{- 1} \Pi_+ \nonumber\\
& = \Pi_+ \Phi_{+}(\tau; \vphi)  {\cal V}_+(\vphi) \Phi_{+}(\tau; \vphi)^{- 1} \Pi_+ + {\cal R}_{\infty, +}(\tau; \vphi) \label{max pippo peg 0}
\end{align}
where, using the properties \eqref{proprieta elementari Pi + -}
\begin{align}
{\cal R}_{\infty, +}(\tau; \vphi) & := \Pi_+ [\Pi_+, \Phi_+(\tau; \vphi)] {\cal V}_+(\vphi) \Pi_+ \Phi_+(\tau; \vphi)^{- 1} \Pi_+ \nonumber\\ 
& \quad + \Pi_+ \Phi_+(\tau; \vphi) {\cal V}_+(\vphi) [\Pi_+ , \Phi_+(\tau; \vphi)^{- 1}] \Pi_+ \nonumber\\
& \quad + [\Pi_+ , \Phi_{+}(\tau; \vphi) ]\Pi_- {\cal V}_-(\vphi) \Pi_- \Phi_{+}(\tau; \vphi)^{- 1}  \Pi_+\,. \label{cal R infty + vphi lemma ii}
\end{align}
Using the same arguments used to analyze the remainder ${\cal R}(\tau; \vphi)$ in \eqref{peg pippo max}, one can show that the remainder ${\cal R}_{\infty, +}(\tau; \vphi) \in OPS^{- \infty}$ implying that 
\begin{equation}\label{boston u}
\Pi_+ \Phi_{+}(\tau; \vphi) {\cal V}(\vphi) \Phi_{+}(\tau; \vphi)^{- 1} \Pi_+ = \Pi_+ \Phi_{+}(\tau; \vphi)  {\cal V}_+(\vphi) \Phi_{+}(\tau; \vphi)^{- 1} \Pi_+ + OPS^{- \infty}\,.
\end{equation} 
In a similar way, one gets that 
\begin{align}
\Pi_- \Phi_{-}(\tau; \vphi) {\cal V}(\vphi) \Phi_{-}(\tau; \vphi)^{- 1} \Pi_- &    = \Pi_- \Phi_{-}(\tau; \vphi)  {\cal V}_-(\vphi) \Phi_{-}(\tau; \vphi)^{- 1} \Pi_- + OPS^{- \infty}\,.  \label{max pippo peg 1}
\end{align}
Therefore the expansion  \eqref{cal V1 split pm} follows by \eqref{prima forma cal V1 ii lemma astratto}, \eqref{max pippo peg 0}, \eqref{boston u}, \eqref{max pippo peg 1} and the lemma is proved. 
\end{proof}
We now recall some well known properties of the Hilbert transform. The Hilbert transform is defined by 
\begin{equation}\label{definizione hilbert transform}
{\cal H}(1) := 0\,, \quad {\cal H}(e^{\ii j x} ) := - \ii \sign(j) e^{\ii j x}\,, \quad \forall j \in \Z \setminus \{ 0 \}\,.
\end{equation}
We often identify the operator ${\cal H}$ with the operator associated to the symbol $- \ii \sign(\xi) \chi(\xi)$ (recall the definition \eqref{definizione cut off D alpha}), since the action on the $2 \pi$-periodic functions of ${\cal H}$ and ${\rm Op}\big( - \ii \,\sign(\xi) \chi(\xi) \big)$ is the same, i.e. 
\begin{equation}\label{simbolo esteso hilbert transform}
{\cal H} \equiv {\rm Op}\Big(- \ii \, \sign(\xi) \chi(\xi) \Big)\,. 
\end{equation}
We now state some classical results concerning the Hilbert transform. 
\begin{lemma}\label{commutatore H moltiplicazione}
Let $a \in {\cal C}^\infty(\T^{\nu + 1})$. Then the commutator $[a, {\cal H}] \in OPS^{- \infty}$.
\end{lemma}
Let $\alpha \in {\cal C}^\infty(\T^{\nu + 1}, \R)$ satisfy the condition \eqref{ansatz alpha egorov}. Then by Lemma \ref{diffeo del toro lemma astratto}, the map $x \mapsto x + \alpha(\vphi, x)$ is a diffeomorphism of the torus, with inverse given by $y \mapsto y + \widetilde \alpha(\vphi, y)$ and $\widetilde \alpha \in {\cal C}^\infty(\T^{\nu + 1}, \R)$ satisfy the properties stated in Lemma \ref{diffeo del toro lemma astratto}. The following lemma holds: 
\begin{lemma}\label{lemma cambi di variabile}
Let $\alpha \in {\cal C}^\infty(\T^{\nu + 1}, \R)$ satisfy the condition \eqref{ansatz alpha egorov}.

\noindent
$(i)$ The operator $\Psi(\vphi)[u] := u(x + \alpha(\vphi, x))$ is a bounded linear operator $H^s \to H^s$, for any $s \geq 0$, whose inverse is given by $\Psi(\vphi)^{- 1}[u] = u(y + \widetilde \alpha(\vphi, y))$ where $y \mapsto y + \widetilde \alpha(\vphi, y)$, $\widetilde \alpha \in {\cal C}^\infty(\T^{\nu + 1}, \R)$ is the inverse diffeomorphism of $x \mapsto x + \alpha(\vphi, x)$. The map $\Phi(\vphi) := \sqrt{1 + \alpha_x(\vphi, x)} \Psi(\vphi)$ is a symplectic bounded invertible operator $H^s \to H^s$ for any $s \geq 0$ and the inverse is given by $\Phi(\vphi)^{- 1} = \sqrt{1 + \widetilde \alpha(\vphi, y)} \Psi(\vphi)^{- 1}$. Furthermore, the following holds:  
\begin{equation}\label{d-alessio}
\sup_{\vphi \in \T^\nu} \{ \| \partial_\vphi^\alpha \Psi(\vphi)^{\pm 1}\|_{{\cal B}(H^{s + |\alpha|}, H^s)}, \| \partial_\vphi^\alpha \Phi(\vphi)^{\pm 1}\|_{{\cal B}(H^{s + |\alpha|}, H^s)}  \} < + \infty\,, \quad \forall s \geq 0\,,\,\alpha \in \N^\nu\,. 
\end{equation}
$(ii)$ Let $A(\vphi) \in OPS^m$. Then $\Phi(\vphi) A(\vphi) \Phi(\vphi)^{- 1} \in OPS^m$. 
\end{lemma}
\begin{proof}
We prove the lemma for the map $\Phi(\vphi)$. the claimed statements for $\Psi(\vphi)$ follow by similar arguments. Define 
\begin{equation}\label{definizione Phi pm tau}
\Phi(\tau; \vphi)[u ] := \sqrt{1 + \tau (\partial_x \alpha)(\vphi, x)} \,\,u(x + \tau \alpha(\vphi, x))\,, \quad u \in L^2(\T)\,, \quad \tau \in [0, 1]\,. 
\end{equation} 
A direct verification shows that $\Phi(\tau; \vphi)$ is an invertible symplectic operator, whose inverse is given by 
\begin{equation}\label{definizione Phi pm tau inverso}
\Phi(\tau; \vphi)^{- 1}[u ] := \sqrt{1 + (\partial_y \widetilde \alpha )(\tau; \vphi, y)} \,\,u(y + \widetilde \alpha(\tau; \vphi, y))\,, \quad u \in L^2(\T)
\end{equation} 
where $y \mapsto y + \widetilde \alpha(\tau; \vphi, y)$ is the inverse diffeomorphism of $x \mapsto x + \tau \alpha(\vphi, x)$. 
A direct calculation shows that the map $\Phi(\tau; \vphi)$ is the flow map of the transport PDE
\begin{equation}\label{nuovo trasporto M = 1}
\begin{aligned}
& \partial_\tau u = {\cal A}(\tau; \vphi)[u]\,, \quad {\cal A}(\tau; \vphi) :=  b_1 (\tau; \vphi, x)\partial_x  + b_{ 0}(\tau; \vphi , x) \,, \\
& b_{ 1} := \frac{\alpha }{\sqrt{1 + \tau (\partial_x\alpha)}} \,, \quad b_{  0} := \frac{\partial_\tau \big( \sqrt{1 + \tau (\partial_x\alpha)} \big)}{\sqrt{1 + \tau (\partial_x \alpha)}}\,.
\end{aligned}
\end{equation}
The estimate \eqref{d-alessio} for $\Phi(\vphi) = \Phi(1; \vphi)$ follows by applying Lemma \ref{flusso + derivate flusso}. Moreover, by the classical Egorov Theorem (see Theorem in A.0.9 in \cite{Taylor}), one gets that if $A(\vphi) \in OPS^m$, $\Phi(\vphi) A(\vphi ) \Phi(\vphi)^{- 1} \in OPS^m$. 
\end{proof}
\begin{lemma}\label{hilbert transform cambi di variabile}
The following property holds:  $\Psi(\vphi) {\cal H} \Psi(\vphi)^{- 1} - {\cal H} \in OPS^{- \infty}$ and $\Phi(\vphi) {\cal H} \Phi(\vphi)^{- 1} - {\cal H} \in OPS^{- \infty}$.
\end{lemma}
\begin{proof}
The fact that $\Psi(\vphi) {\cal H} \Psi(\vphi)^{- 1} - {\cal H} \in OPS^{- \infty}$ is a classical result. For a detailed proof see for instance Lemmata 2.32, 2.36 in \cite{BertiMontalto}. Let us prove that $\Phi(\vphi) {\cal H} \Phi(\vphi)^{- 1} - {\cal H} \in OPS^{- \infty}$. To shorten notations we neglect the dependence on $\vphi$. By Lemma \ref{lemma cambi di variabile}, we have 
\begin{align}
\Phi {\cal H} \Phi^{- 1} - {\cal H} & = \sqrt{1 + \alpha_x}  \Psi {\cal H} \sqrt{1 + \widetilde \alpha_y} \Psi^{- 1} - {\cal H} \nonumber\\
& =   \sqrt{1 + \alpha_x} \circ (\Psi {\cal H} \Psi^{- 1}) \circ \Psi  \sqrt{1 + \widetilde \alpha_y} \Psi^{- 1} - {\cal H} \nonumber\\
& =  \sqrt{1 + \alpha_x} \circ (\Psi  {\cal H} \Psi^{- 1} - {\cal H}) \circ \Psi   \sqrt{1 + \widetilde \alpha_y} \Psi^{- 1} \nonumber\\
& \quad +   \sqrt{1 + \alpha_x} \circ {\cal H} \circ \Psi \sqrt{1 + \widetilde \alpha_y} \Psi^{- 1} - {\cal H}\,. \label{fuffa economica}
\end{align}
Note that $\Psi  \sqrt{1 + \widetilde \alpha_y} \Psi^{- 1} $ is a multiplication operator given by 
\begin{equation}\label{fuffa economica - 1}
\Psi  \sqrt{1 + \widetilde \alpha_y} \Psi^{- 1} = \sqrt{1 + \widetilde \alpha_y(\vphi, x +  \alpha_x(\vphi, x))}  \stackrel{\eqref{identita 1 + alpha alpha tilde partial}}{=} \frac{1}{\sqrt{1 +  \alpha_x(\vphi, x)}}\,.
\end{equation}
Hence 
\begin{align}
\sqrt{1 +  \alpha_x} \circ {\cal H} \circ \Psi  \sqrt{1 + \widetilde \alpha_y} \Psi^{- 1} - {\cal H} & = \sqrt{1 + \alpha_x} \circ {\cal H} \circ \frac{1}{\sqrt{1 +  \alpha_x}} - {\cal H} \nonumber\\
 & = \sqrt{1 + \alpha_x} \Big[{\cal H}, \frac{1}{\sqrt{1 +  \alpha_x}} \Big]\,.  \label{fuffa economica 0}
\end{align}
Finally \eqref{fuffa economica}-\eqref{fuffa economica 0}, Lemma \ref{commutatore H moltiplicazione} and using that $\Psi {\cal H} \Psi^{- 1} - {\cal H} \in OPS^{- \infty}$ one obtains that 
$$
\Phi{\cal H} \Phi^{- 1} - {\cal H} = \sqrt{1 +  \alpha_x} \circ (\Psi {\cal H} \Psi^{- 1} - {\cal H}) \circ \frac{1}{\sqrt{1 +  \alpha_x}} + \sqrt{1 + \alpha_x} \Big[{\cal H}, \frac{1}{\sqrt{1 +  \alpha_x}} \Big] \in OPS^{- \infty}\,. 
$$
\end{proof}
We conclude this section by stating an interpolation theorem, which is an immediate consequence of the classical Riesz-Thorin interpolation theorem in Sobolev spaces. 
\begin{theorem}\label{interpolazione sobolev}
Let $0 \leq s_0 < s_1$ and let $A \in {\cal B}(H^{s_0}) \cap {\cal B}(H^{s_1})$. Then for any $s_0 \leq s \leq s_1$ the operator $A \in {\cal B}(H^s)$ and 
$$
\| A \|_{{\cal B}(H^s)} \leq \| A\|_{{\cal B}(H^{s_0})}^\lambda \| A\|_{{\cal B}(H^{s_1})}^{1 - \lambda}\,, \quad \lambda := \frac{s_1 - s}{s_1 - s_0}\,. 
$$
\end{theorem}

\section{Regularization of the vector field $\ii {\cal V}(\vphi)$: the case $0 < M < 1$}\label{sezione riduzione M < 1}
In this section we develop the regularization procedure on the vector field $\ii {\cal V}(\vphi) = \ii \big(V(\vphi, x) |D|^M + {\cal W}(\vphi) \big)$, $0 < M < 1 $, see \eqref{forma iniziale cal V (t)}, which is needed to prove Theorem \ref{teorema riduzione}. We assume the hypotheses $\bf (H1)$, $\bf (H2)$, $\bf (H3)_{M < 1}$, i.e. ${\cal V}(\vphi)$ is self-adjoint, ${\cal W} \in OPS^{M - \frak e}$, $\frak e > 0$ and $V \in {\cal C}^\infty(\T^{\nu + 1}, \R)$ satisfies $\inf_{(\vphi, x) \in \T^{\nu + 1}} V(\vphi, x) > 0$. 

\noindent
In Section \ref{sezione ordine principale M < 1} we reduce to constant coefficients the highest order term $V(\vphi, x) |D|^M$, see Proposition \ref{teorema riassunto primo egorov}. Then, in Section \ref{sezione riduzione ordini bassi M < 1}, we perform the reduction of the lower order terms up to arbitrarily regularizing remainders, see Proposition \ref{descent method M  < 1}. At each step of the regularization procedure, the reduction to constant coefficients is split in two parts: first we remove the dependence on $\vphi$ and in a second step, we remove the dependence on $x$. 
\subsection{Reduction of the highest order}\label{sezione ordine principale M < 1}
In order to state precisely the main result of this section, we define 
\begin{equation}\label{costante frak e bar M < 1}
\overline{\frak e} := \min\{ 1- M, \frak e \}
\end{equation}
so that 
\begin{equation}\label{proprieta M - frak e bar M < 1}
M - \overline{\frak e} \geq M - \frak e\,,\, M - (1 - M)\,
\end{equation}
(recall that $0 < M < 1$). We prove the following 
\begin{proposition}\label{teorema riassunto primo egorov}
Let $\gamma \in (0, 1), \tau > \nu - 1$ and $\omega \in DC(\gamma, \tau)$ (recall \eqref{vettori diofantei}). There exist a symplectic family of invertible maps $\Phi_0(\vphi)$, $\vphi \in \T^\nu$ satisfying 
\begin{equation}\label{proprieta Phi 2}
\sup_{\vphi \in \T^\nu} \| \Phi_0(\vphi)^{\pm 1}\|_{{\cal B}(H^s)}  < + \infty\,, \quad \forall s  \geq 0\,,
\end{equation}
a constant $\lambda > 0$ and a self-adjoint operator ${\cal W}_1(\vphi) = {\rm Op}\Big( w_1(\vphi, x, \xi)\Big) \in OPS^{M - \overline{\frak e}}$ such that 
\begin{equation}\label{push forward cal V 1 teorema}
(\Phi_0^{- 1})_{\omega *}(\ii {\cal V})(\vphi) = \ii {\cal V}_1(\vphi) \quad \text{with} \quad {\cal V}_1(\vphi) := \lambda |D|^M + {\cal W}_1(\vphi)\,. 
\end{equation}
\end{proposition}
The rest of this subsection is devoted to the proof of Proposition \ref{teorema riassunto primo egorov}. In Section \ref{riduzione tempo ordine principale} we show how to remove the dependence on $\vphi$ from the highest order. Then, in Section \ref{riduzione spazio ordine massimo M < 1}, we remove the dependence on $x$. 
\subsubsection{Time reduction}\label{riduzione tempo ordine principale}
We consider a function $\alpha \in {\cal C}^\infty(\T^\nu \times \T, \R)$ (to be determined) and an operator of the form  
\begin{equation}\label{operator cal A M < 1}
{\cal G}_0^{(1)}(\vphi) := \alpha (\vphi, x) |D|^M + |D|^M \alpha (\vphi, x)\,.
\end{equation}
Note that ${\cal G}_0^{(1)} (\vphi) = {\cal G}_0^{(1)} (\vphi)^*$ for any $\vphi \in \T^\nu$ and by applying Theorem \ref{teorema composizione pseudo},
\begin{equation}\label{espansione g0 M < 1}
{\cal G}_0^{(1)}(\vphi) = {\rm Op}\Big( g_0^{(1)}(\vphi, x, \xi) \Big)\,, \quad g_0^{(1)}(\vphi, x, \xi) := 2 \alpha(\vphi, x) |\xi|^M\chi(\xi) +  S^{M - 1} 
\end{equation}
(recall the notation \eqref{notazione conveniente}).
Since ${\cal G}_0^{(1)}$ fullfills the hypotheses of Lemma \ref{flusso + derivate flusso}, one has that the flow $\Phi_0^{(1)}(\tau; \vphi)$, $\tau \in [0, 1]$, $\vphi \in \T^\nu$ of the autonomous PDE $\partial_\tau u = \ii {\cal G}_0^{(1)}(\vphi)[ u]$, i.e. 
\begin{equation}\label{flusso cal A M < 1}
\begin{cases}
\partial_\tau \Phi_0^{(1)}(\tau; \vphi) = \ii {\cal G}_0^{(1)}(\vphi) \Phi_0^{(1)}(\tau; \vphi) \\
\Phi_0^{(1)}(0; \vphi) = {\rm Id}
\end{cases}
\end{equation}
is an invertible map $H^s \to H^s$ and satisfies 
\begin{equation}\label{proprieta continuita Phi 0 M < 1}
\sup_{\begin{subarray}{c}
\tau \in [0, 1] \\
\vphi \in \T^\nu
\end{subarray}} \| \Phi_0^{(1)}(\tau; \vphi)\|_{{\cal B}(H^s)}, \quad \forall s \geq 0\,. 
\end{equation}
Since ${\cal G}_0^{(1)}(\vphi)$ is self-adjoint, $\Phi_0^{(1)}(\tau; \vphi)$ is a symplectic map. We set $\Phi_0^{(1)}(\vphi) := \Phi_0^{(1)}(1; \vphi)$ and we have  
$$
\big((\Phi_0^{(1)})^{- 1} \big)_{\omega *}{\cal V}(\vphi ) = \ii {\cal V}_0^{(1)}(\vphi)\,, \quad {\cal V}_0^{(1)}(\vphi) := \Phi_0^{(1)}(\vphi) {\cal V}(\vphi) \Phi_0^{(1)}(\vphi)^{- 1} + \ii \Phi_0^{(1)}(\vphi) \omega \cdot \partial_\vphi \Phi_0^{(1)}(\vphi)^{- 1}\,. 
$$
By applying Proposition \ref{teorema egorov campo minore di uno} with $m = \eta = M$, since  
$$
v(\vphi, x, \xi) = V(\vphi, x) |\xi|^M\chi(\xi) + w(\vphi, x, \xi) \quad \text{with} \quad w \in S^{M - \frak e}\,,
$$
by recalling \eqref{costante frak e bar M < 1}, \eqref{proprieta M - frak e bar M < 1}, using that by \eqref{inclusioni Sm OPSm}, $S^{M - \frak e}, S^{M - (1 - M)} \subseteq S^{M - \overline{\frak e}}$ one gets that 
\begin{equation}\label{def cal V0 (1) M < 1}
 \Phi_0^{(1)}(\vphi) {\cal V}(\vphi) \Phi_0^{(1)}(\vphi)^{- 1} = {\rm Op}\Big(V(\vphi, x) |\xi|^M\chi(\xi)  \Big) + OPS^{M - \overline{\frak e}}\,.
\end{equation}
Moreover, defining $\Psi(\tau; \vphi) := \ii  \Phi_0^{(1)}(\tau; \vphi) \omega \cdot \partial_\vphi \big( \Phi_0^{(1)}(\tau; \vphi)^{- 1} \big)$, a direct calculation shows that 
$$
\Psi(\tau; \vphi) = \ii \int_0^\tau {\cal S}(\zeta; \vphi )\, d \zeta\,, \quad  {\cal S} (\zeta; \vphi) := \Phi_0^{(1)}(\zeta; \vphi) \omega \cdot \partial_\vphi  {\cal G}_0^{(1)}(\vphi) \Phi_0^{(1)}(\zeta; \vphi)^{- 1}, \quad \zeta \in [0, 1]\,. 
$$
Since $\omega \cdot \partial_\vphi {\cal G}_0^{(1)}(\vphi) \in OPS^M$, by Proposition \ref{teorema egorov campo minore di uno} (applied with $m = \eta = M$) and using \eqref{espansione g0 M < 1}, one has 
\begin{equation}\label{ordine Psi M < 1}
\Psi(\vphi) = \Psi(1 ; \vphi) = \ii \Phi_0^{(1)}(1; \vphi) \omega \cdot \partial_\vphi \big( \Phi_0^{(1)}(1 ; \vphi)^{- 1} \big) = {\rm Op}\Big(\psi (\vphi, x, \xi) \Big) \in OPS^M 
\end{equation}
with 
\begin{equation}\label{espansione psi 0 M < 1}
\psi(\vphi, x, \xi) := 2 \omega \cdot \partial_\vphi \alpha(\vphi, x) |\xi|^M \chi(\xi) + r_1 (\vphi, x, \xi)\,, \quad r_1 \in S^{M - (1 - M)} \stackrel{\eqref{proprieta M - frak e bar M < 1}, \eqref{inclusioni Sm OPSm}}{\subseteq} S^{M - \overline{\frak e}}\,. 
\end{equation}
Therefore, \eqref{def cal V0 (1) M < 1}-\eqref{espansione psi 0 M < 1} imply that ${\cal V}_0^{(1)}(\vphi) = {\rm Op}\Big( v_0^{(1)}(\vphi, x, \xi) \Big)$ with
\begin{equation}\label{espansione finale cal V 0}
\begin{aligned}
& v_0^{(1)}(\vphi, x, \xi) = \Big(V(\vphi, x) +  2 \omega \cdot \partial_\vphi \alpha (\vphi, x) \Big) |\xi|^M \chi(\xi) + S^{M - \overline{\frak e}}\,. 
 \end{aligned}
\end{equation}
Defining 
\begin{equation}\label{media vphi V M < 1}
\langle V \rangle_\vphi(x) := \frac{1}{(2 \pi)^\nu} \int_{\T^\nu} V(\vphi, x)\, d \vphi\,,
\end{equation}
we want to choose the function $\alpha_0$ so that 
\begin{equation}\label{equazione omologica 0 M < 1}
V(\vphi, x) +  2 \omega \cdot \partial_\vphi \alpha(\vphi, x)  = \langle V \rangle_\vphi(x)\,. 
\end{equation}
Since $\langle V \rangle_\vphi - V$ has zero average w.r. to $\vphi \in \T^\nu$ and $\omega \in DC(\gamma, \tau)$, one has that  
\begin{equation}\label{definizione alpha 0 M < 1}
\alpha := (\omega \cdot \partial_\vphi)^{- 1}\Big[\frac{\langle V \rangle_\vphi - V}{2} \Big]
\end{equation}
solves the equation \eqref{equazione omologica 0 M < 1}. 
Note that $\alpha \in {\cal C}^\infty(\T^\nu \times \T, \R)$, since $V \in {\cal C}^\infty(\T^\nu \times \T, \R)$ and since $V$ satisfies the hypothesis $\bf (H3)_{M < 1}$, by \eqref{media vphi V M < 1}, one sees that
\begin{equation}\label{positivita media vphi V}
\inf_{x \in \T} \langle V \rangle_\vphi(x) > 0\,. 
\end{equation}
Finally \eqref{espansione finale cal V 0}, \eqref{equazione omologica 0 M < 1} imply that 
\begin{equation}\label{forma finalissima cal V0 M < 1}
{\cal V}_0^{(1)}(\vphi) = {\rm Op}\Big( v_0^{(1)}(\vphi, x, \xi) \Big)\,, \quad v_0^{(1)}(\vphi, x, \xi) = \langle V \rangle_\vphi(x) |\xi|^M \chi(\xi) + S^{M - \overline{\frak e}}\,. 
\end{equation}
Since ${\cal V}(\vphi)$ is self-adjoint and $\Phi_0^{(1)}(\vphi)$ is symplectic, then also ${\cal V}_0^{(1)}(\vphi)$ is self-adjoint. 
\subsubsection{Space reduction}\label{riduzione spazio ordine massimo M < 1}
In this section, our purpose is to remove the dependence on $x$ from the highest order term $\ii \langle V \rangle_\vphi(x) |D|^M$ of the vector field $\ii {\cal V}_0^{(1)}(\vphi)$ given in \eqref{forma finalissima cal V0 M < 1}. To this aim, we consider a function $\beta \in {\cal C}^\infty(\T, \R)$ (that will be fixed later) satisfying the ansatz 
\begin{equation}\label{ansatz alpha 1 M < 1}
\inf_{x \in \T}(1 + \partial_x \beta(x)) > 0\,. 
\end{equation}
We then define 
\begin{equation}\label{definizione b1}
b(\tau; x) := - \frac{\beta (x)}{1 + \tau \partial_x \beta (x)}\,, \quad (\tau, x) \in [0, 1] \times \T
\end{equation}
and we consider the $\tau$-dependent vector field 
\begin{equation}\label{def cal G1 M < 1}
{\cal G}_0^{(2)}(\tau) := b(\tau; x) \partial_x + \frac{\partial_x b}{2}\,. 
\end{equation}
As explained in Section \ref{sezione astratta Egorov}, the flow of the PDE $\partial_\tau u = {\cal G}_0^{(2)}(\tau)[u]$, i.e. the map $\Phi_0^{(2)}(\tau)$, $\tau \in [0, 1]$ which solves 
\begin{equation}\label{Phi 1 M < 1}
\begin{cases}
\partial_\tau \Phi_0^{(2)}(\tau) = {\cal G}_0^{(2)}(\tau) \Phi_0^{(2)}(\tau) \\
\Phi_0^{(2)}(0) = {\rm Id}\,. 
\end{cases}
\end{equation}
 is a bounded invertible symplectic map satisfying 
 \begin{equation}\label{proprieta Phi 1 egorov M < 1}
 \sup_{\tau \in [0, 1] } \| \Phi_0^{(2)}(\tau)^{\pm 1}\|_{{\cal B}(H^s)} < + \infty\,, \quad \forall s \geq 0\,.
 \end{equation}
 Note that ${\cal G}_0^{(2)}(\tau)$ and $\Phi_0^{(2)}(\tau)$ are independent of $\vphi \in \T^\nu$. 
We set $\Phi_0^{(2)} := \Phi_0^{(2)}(1)$ and the transformed vector field is given by 
\begin{equation}\label{primo cal V1}
\big((\Phi_0^{(2)})^{- 1} \big)_{\omega *}\ii {\cal V}_0^{(1)}(\vphi) = \ii {\cal V}_1(\vphi)\,, \quad {\cal V}_1(\vphi) := \Phi_0^{(2)} {\cal V}_0^{(1)}(\vphi) (\Phi_0^{(2)})^{- 1}\,. 
\end{equation}
By applying Proposition \ref{Teorema egorov generale}, the operator ${\cal V}_1(\vphi) = {\rm Op}\Big( v_1(\vphi, x, \xi) \Big)$ admits the expansion 
\begin{equation}\label{prima espansione v1 M < 1}
v_1(\vphi, x, \xi) = v_0^{(1)}\Big(\vphi, x + \beta(x), (1 + \partial_x \beta(x))^{- 1} \xi \Big) + S^{M - 1}\,. 
\end{equation}
\begin{lemma}\label{espansione simbolo principale egorov}
The symbol $v_1(\vphi, x, \xi)$ has the form 
\begin{equation}\label{espansione simbolo q0 trasporto}
v_1(\vphi, x, \xi) = \Big[\langle V\rangle_\vphi( y) \big( 1 + \partial_y \widetilde \beta( y) \big)^M \Big]_{y = x + \beta( x)} |\xi|^M \chi(\xi) + S^{M - \overline{\frak e}}\,
\end{equation}
where we recall the definitions \eqref{definizione cut off D alpha}, \eqref{definizione D alpha} and  $y \mapsto y + \widetilde \beta(y)$ is the inverse diffeomorphism of $x \mapsto x + \beta (x)$. 
\end{lemma}
\begin{proof}
The Lemma follows by using the same aguments used in the proof of Lemma 3.2 in \cite{Montalto1}, hence the proof is omitted. 
\end{proof}
We now determine the function $\widetilde \beta$ and a constant $\lambda > 0$ so that 
\begin{equation}\label{equazione omologica x ordine alto}
\langle V\rangle_\vphi( y) \big( 1 + \partial_y \widetilde \beta( y) \big)^M = \lambda\,. 
\end{equation}
The equation \eqref{equazione omologica x ordine alto} is equivalent to the equation 
\begin{equation}\label{equazione omologica x ordine alto A}
\partial_y \widetilde \beta(y) = \frac{\lambda^{\frac{1}{M}}}{\langle V \rangle_\vphi(y)^{\frac{1}{M}}} - 1\,.
\end{equation} 
Note that by \eqref{positivita media vphi V}, the function $\langle V\rangle_\vphi$ does not vanish. Then we choose $\lambda$ so that the right hand side of the equation \eqref{equazione omologica x ordine alto A} has zero average, i.e. 
\begin{equation}\label{definizione lambda M < 1}
\lambda := \Big( \frac{1}{2 \pi} \int_\T \langle V \rangle_\vphi(y)^{- \frac{1}{M}}\, d y\Big)^{- M}
\end{equation}
and hence we define $\widetilde \beta$ as 
\begin{equation}\label{definizione widetilde alpha 1 M < 1}
\widetilde \beta := \partial_y^{- 1}\Big[ \frac{\lambda^{\frac{1}{M}}}{\langle V \rangle_\vphi(y)^{\frac{1}{M}}} - 1 \Big]\,.
\end{equation}
By the definition \eqref{definizione lambda M < 1} and recalling the property \eqref{positivita media vphi V}, one has that 
\begin{equation}\label{segno lambda}
\lambda > 0\,. 
\end{equation}
By \eqref{equazione omologica x ordine alto A}, \eqref{positivita media vphi V} one has that $\widetilde \beta \in {\cal C}^\infty(\T)$ and satisfies 
$$
\inf_{y \in \T} (1 + \partial_y \widetilde\beta (y)) > 0\,,
$$
hence, by applying Lemma \ref{diffeo del toro lemma astratto}, the inverse diffeomorphism $x \mapsto x + \beta(x)$ satisfies the ansatz \eqref{ansatz alpha 1 M < 1}. Finally, \eqref{espansione simbolo q0 trasporto}, \eqref{equazione omologica x ordine alto} imply that 
\begin{equation}\label{espansione finalissima v_1 M < 1}
{\cal V}_1(\vphi) = {\rm Op}\Big( v_1(\vphi, x, \xi) \Big)\,, \quad v_1(\vphi, x, \xi) = \lambda |\xi|^M \chi(\xi) + w_1(\vphi, x, \xi)\,, \quad w_1 \in S^{M - \overline{\frak e}}\,. 
\end{equation}
We then define $\Phi_0(\vphi) := \Phi_0^{(1)}(\vphi) \circ \Phi_0^{(2)}$. By \eqref{proprieta continuita Phi 0 M < 1}, \eqref{proprieta Phi 1 egorov M < 1}, the symplectic map $\Phi_0(\vphi)$ satisfies the property \eqref{proprieta Phi 2}. Since $\Phi_0$ is symplectic the vector field $\ii {\cal V}_1$ is Hamiltonian, i.e. ${\cal V}_1(\vphi)$ is self-adjoint. Since $\lambda \in \R$, and hence $\lambda |D|^M$ is self-adjoint, then also ${\cal W}_1(\vphi) = {\cal V}_1(\vphi) - \lambda |D|^M$ is self-adjoint and then the proof of Proposition \ref{teorema riassunto primo egorov} is concluded. 
\subsection{Reduction of the lower order terms}\label{sezione riduzione ordini bassi M < 1}
We now prove the following 
\begin{proposition}\label{descent method M  < 1}
Let $\gamma \in (0 1), \tau > \nu - 1$, $\omega \in DC(\gamma, \tau)$ and $N \in \N$. For any $n = 1, \ldots, N$ there exists a linear Hamiltonian vector field $\ii {\cal V}_n(\vphi)$ of the form 
\begin{equation}\label{forma cal Vn teorema}
{\cal V}_n(\vphi) := \lambda  |D|^M + \mu_n(D) + {\cal W}_n(\vphi)
\end{equation}
where 
\begin{equation}\label{mu n t D teo}
\mu_n( D) := {\rm Op}\Big( \mu_n( \xi)\Big)\,, \qquad \mu_n \in S^{M - \bar{\frak e}}\,,
\end{equation}
\begin{equation}\label{cal Wn teorema}
{\cal W}_n(\vphi) := {\rm Op}\Big( w_n(\vphi, x, \xi)\Big) \,, \qquad w_n \in S^{M - n \bar{\frak e}}\,,
\end{equation}
with $\mu_n(\xi)$ real and ${\cal W}_n(\vphi)$ self-adjoint, i.e. $w_n = w_n^*$ (see Theorem \ref{adjoint}).

\noindent
For any $n \in \{ 1, \ldots, N - 1\}$, there exists a symplectic invertible map $\Phi_n(\vphi)$ satisfying
\begin{equation}\label{proprieta Phi n teo 1}
\sup_{\vphi \in \T^\nu}\| \Phi_n(\vphi)^{\pm 1} \|_{{\cal B}(H^{s})}   < + \infty\,, \quad \forall s \geq 0 
\end{equation}
and 
\begin{equation}\label{cal V n + 1 cal V n}
\ii {\cal V}_{n + 1}(\vphi) = (\Phi_n^{- 1})_{\omega*}( \ii {\cal V}_n)(\vphi)\,, \qquad \forall n \in \{ 1, \ldots, N - 1 \}\,. 
\end{equation}
\end{proposition}
The remaining part of this section is devoted to the proof of this Proposition. We describe the inductive step. Given $n \in \{ 1, \ldots, N\}$, we assume that the vector field $\ii{\cal V}_n(\vphi)$ satisfies the properties \eqref{forma cal Vn teorema}-\eqref{cal Wn teorema}. The reduction to constant coefficients of the order $M - n \bar{\frak e}$ is split in two parts: in Section \ref{ordini bassi tempo M < 1} we eliminate the dependence on $\vphi$ from the symbol $w_n$ and in Section \ref{reduction lower orders space M < 1}, we eliminate the dependence on $x$. 
\subsubsection{Time reduction}\label{ordini bassi tempo M < 1}
We consider an operator 
\begin{equation}\label{cal G n + 1 M < 1}
\begin{aligned}
& {\cal G}_{n }^{(1)}(\vphi) = {\rm Op}(g_{n }^{(1)}(\vphi, x, \xi))\,, \qquad g_{n}^{(1)} \in S^{M - n \overline{\frak e}}\,, \\
& {\cal G}_{n}^{(1)}(\vphi) = {\cal G}_{n}^{(1)}(\vphi)^*\,. \quad \forall \vphi \in \T^\nu
\end{aligned}
\end{equation}
Since the operator ${\cal G}_n^{(1)}$ satisfies the hypotheses of Lemma \ref{flusso + derivate flusso}, the flow $\Phi_{n }^{(1)}(\tau; \vphi)$ of the autonomous PDE $\partial_\tau u = \ii {\cal G}_{n}^{(1)}(\vphi) u$, i.e. 
\begin{equation}\label{Phi n + 1 (1) M < 1}
\begin{cases}
\partial_\tau \Phi_{n}^{(1)}(\tau; \vphi) = \ii {\cal G}_{n}^{(1)}(\vphi) \Phi_{n}^{(1)}(\tau; \vphi) \\
\Phi_{n }^{(1)}(0; \vphi) = {\rm Id}\,.  
\end{cases}
\end{equation}
is a well defined, invertible map satisfying 
\begin{equation}\label{proprieta Phi n (1) M < 1}
\sup_{\begin{subarray}{c}
\tau \in [0, 1] \\
\vphi \in \T^\nu
\end{subarray}} \| \Phi_n^{(1)}(\tau; \vphi)\|_{{\cal B}(H^s)}  < + \infty\,, \quad \forall s \geq 0\,. 
\end{equation}
Furthermore, since ${\cal G}_n^{(1)}$ is self-adjoint, the map $\Phi_n^{(1)}(\tau; \vphi)$ is symplectic. We define $\Phi_{n}^{(1)}(\vphi) := \Phi_{n}^{(1)}(1; \vphi)$. We then compute 
\begin{equation}\label{primo cal V n + 1}
\begin{aligned}
& (\Phi_{n}^{(1)})^{- 1}_{\omega *} \ii {\cal V}_n(\vphi) = \ii {\cal V}_{n}^{(1)}(\vphi)\,, \\
& {\cal V}_{n}^{(1)}(\vphi) := \Phi_{n }^{(1)}(\vphi) {\cal V}_{n }(\vphi)  \Phi_{n}^{(1)}(\vphi)^{- 1} + \ii \Phi_{n }^{(1)}(\vphi) \omega \cdot \partial_\vphi  \Phi_{n }^{(1)}(\vphi)^{- 1} \,.
\end{aligned}
\end{equation}
A direct calculation shows that 
\begin{equation}\label{cal S n Phi n + 1 M < 1}
\begin{aligned}
& \ii \Phi_{n }^{(1)}(\vphi) \omega \cdot \partial_\vphi  \Phi_{n }^{(1)}(\vphi)^{- 1} =  \int_0^1 {\cal S}_{n}(\tau; \vphi)\, d \tau\,, \\
& {\cal S}_{n}(\tau; \vphi) := \Phi_{n }^{(1)}(\tau; \vphi) \omega \cdot \partial_\vphi {\cal G}_{n }^{(1)}(\vphi) \Phi_{n }^{(1)}(\tau; \vphi)^{- 1}\,. 
\end{aligned}
\end{equation}
By \eqref{primo cal V n + 1}-\eqref{cal S n Phi n + 1 M < 1}, by applying Proposition \ref{teorema egorov campo minore di uno} (with $m = M$ and $\eta = M - n \overline{\frak e}$ for the term $\Phi_{n }^{(1)}(\vphi) {\cal V}_{n }(\vphi)  \Phi_{n}^{(1)}(\vphi)^{- 1}$ and $m = \eta =  M - n \overline{\frak e}$ for the term given in \eqref{cal S n Phi n + 1 M < 1}), one gets that 
\begin{equation}\label{espansione tilde vn M < 1}
\begin{aligned}
{\cal V}_{n}^{(1)}(\vphi) & = {\rm Op}\big( v_n^{(1)}(\vphi, x, \xi)\big) \in OPS^M \,, \\
v_n^{(1)}(\vphi, x, \xi)& = \lambda |\xi|^M \chi(\xi) + \mu_n(\xi) + w_n(\vphi, x, \xi) +  \omega \cdot \partial_\vphi g_{n}^{(1)}(\vphi, x, \xi)  +  S^{M - (n + 1)\overline{\frak e}}\,. 
\end{aligned}
\end{equation}
In order to eliminate the $\vphi$-dependence from the symbol $w_n(\vphi, x, \xi) +  \omega \cdot \partial_\vphi g_{n}^{(1)}(\vphi, x, \xi)$ (of order $M - n \overline{\frak e}$), we choose the symbol $g_{n }^{(1)}(\vphi, x, \xi)$ so that 
\begin{equation}\label{eq omologica n (1) M < 1}
w_n(\vphi, x, \xi) +  \omega \cdot \partial_\vphi g_{n }^{(1)}(\vphi, x, \xi) = \langle w_n \rangle_\vphi(x, \xi)
\end{equation}
where we recall the definition \eqref{simbolo mediato}. Then, since $\omega \in DC(\gamma, \tau)$, the equation \eqref{eq omologica n (1) M < 1} is solved by defining  
\begin{equation}\label{definizione g n+1 (1) M < 1}
g_{n }^{(1)}(\vphi, x, \xi) := (\omega \cdot \partial_\vphi)^{- 1}\Big[\langle w_n \rangle_\vphi(x, \xi)  - w_n(\vphi, x, \xi)  \Big]\,.
\end{equation}
Note that, since ${\cal W}_n = {\rm Op}(w_n)$ is self-adjoint, i.e. $w_n= w_n^*$, by applying Lemma \ref{proprieta aggiunti small divisors}, one gets that $g_{n }^{(1)} = (g_{n}^{(1)})^*$ implying that ${\cal G}_{n }^{(1)} = {\rm Op}(g_{n}^{(1)})$ is self-adjoint. By \eqref{espansione tilde vn M < 1}, \eqref{eq omologica n (1) M < 1} one then has 
\begin{equation}\label{espansione tilde vn finale M < 1}
\begin{aligned} 
{\cal V}_n^{(1)}  = {\rm Op}(v_n^{(1)}) \,, & \quad v_n^{(1)}(\vphi, x, \xi) = \lambda |\xi|^M \chi(\xi) + \mu_n(\xi) + \langle w_n \rangle_\vphi(x, \xi)  + S^{M - (n + 1)\overline{\frak e}}\,, \\
\mu_n \in S^{M - \overline{\frak e}}\,, & \quad \langle w_n \rangle_\vphi \in S^{M - n \overline{\frak e}}\,. 
\end{aligned}
\end{equation}
Since $\Phi_n^{(1)}(\vphi)$ is symplectic one has that ${\cal V}_n^{(1)}(\vphi)$ is self-adjoint. 
\subsubsection{Space reduction}\label{reduction lower orders space M < 1}
In order to remove the $x$-dependence from the symbol $\langle w_n \rangle_\vphi$ in \eqref{espansione tilde vn finale M < 1}, we consider a $\vphi$-independent pseudo-differential operator 
\begin{equation}\label{cal G n + 1 (2) M < 1}
\begin{aligned}
& {\cal G}_{n }^{(2)} = {\rm Op}\Big(g_{n }^{(2)}(x, \xi) \Big)\,, \quad g_{n }^{(2)} \in S^{1 - n \overline{\frak e}}\,,\quad {\cal G}_{n }^{(2)} = ({\cal G}_{n }^{(2)})^*\,. 
\end{aligned}
\end{equation}
Since the operator ${\cal G}_n^{(2)}$ satisfies the hypotheses of Lemma \ref{flusso + derivate flusso}, the flow $\Phi_{n}^{(2)}(\tau)$, $\tau \in [0, 1]$ of the PDE $\partial_\tau u = \ii {\cal G}_{n}^{(2)}[ u]$, i.e. 
\begin{equation}\label{flusso Phi n + 1 (2) M < 1}
\begin{cases}
\partial_\tau \Phi_{n }^{(2)}(\tau) = \ii {\cal G}_{n }^{(2)} \Phi_{n }^{(2)}(\tau) \\
\Phi_{n}^{(2)}(0) = {\rm Id} 
\end{cases}
\end{equation} 
is a well defined invertible map satisfying 
\begin{equation}\label{stima Phi n + 1 (2) M < 1}
\sup_{\tau \in [0, 1]} \| \Phi_{n}^{(2)}(\tau)^{\pm 1} \|_{{\cal B}(H^s)} < + \infty\,, \quad \forall s \geq 0\,,.
\end{equation}
Note that ${\cal G}_n^{(2)}$ as well as $\Phi_n^{(2)}$ is $\vphi$-independent. Since ${\cal G}_n^{(2)}$ is self-adjoint, the flow $\Phi_n^{(2)}(\tau)$ is symplectic. We set $\Phi_{n}^{(2)} := \Phi_{n}^{(2)}(1)$. The transformed vector field is then given by  
\begin{equation}\label{cal V n + 1 M < 1}
\ii {\cal V}_{n + 1} := ((\Phi_{n }^{(2)})^{- 1})_{\omega *} (\ii {\cal V}_n^{(1)}) (\vphi)\,, \quad {\cal V}_{n + 1}(\vphi) = \Phi_{n}^{(2)} {\cal V}_n^{(1)}(\vphi) (\Phi_{n }^{(2)})^{- 1}\,. 
\end{equation}
We prove the following 
\begin{lemma}\label{espansione cal V n+1 M < 1 prima eq omologica}
The operator ${\cal V}_{n + 1}(\vphi) = {\rm Op}\big( v_{n + 1}(\vphi, x, \xi)\big)$, $v_{n + 1} \in S^M$ has the expansion
\begin{equation}\label{espansione v n + 1 M < 1 A}
v_{n + 1} = \lambda |\xi|^M \chi(\xi) + \mu_n + \langle w_n \rangle_\vphi - M \lambda |\xi|^{M - 2} \xi \chi(\xi) \partial_x g_{n }^{(2)} +  S^{M - (n + 1) \overline{\frak e}}\,. 
\end{equation}
\end{lemma}
\begin{proof}
By applying Proposition \ref{teorema egorov campo minore di uno} (with $m = M$ and $\eta = 1 - n \overline{\frak e}$), one gets that 
\begin{equation}
{\cal V}_{n + 1}(\vphi) = {\rm Op}\Big( v_{n + 1}(\vphi, x, \xi) \Big)\,, \quad v_{n + 1} \in S^M
\end{equation}
and the symbol $v_{n + 1}$ admits the expansion 
\begin{equation}\label{prima espansione v n + 1 M < 1}
v_{n + 1} =  v_n^{(1)} + \{ g_{n + 1}^{(2)}, v_n^{(1)} \} + p_{n , \geq 2}\,, \quad p_{n, \geq 2} \in S^{M - 2 n \overline{\frak e}} \stackrel{\eqref{inclusioni Sm OPSm}}{\subseteq} S^{M - (n + 1) \overline{\frak e}}\,. 
\end{equation}
Recalling \eqref{espansione tilde vn finale M < 1}, one gets 
\begin{align}
\{ g_{n }^{(2)}, v_n^{(1)} \}  & = \{ g_{n}^{(2)},  \lambda   |\xi|^M \chi(\xi) + \mu_n + \langle w_n \rangle_\vphi  + w_{n}^{(1)}\} \nonumber\\
& =  \{ g_{n }^{(2)},  \lambda   |\xi|^M \chi(\xi)\} + \{ g_{n }^{(2)}, \mu_n + \langle w_n \rangle_\vphi  + w_{n}^{(1)} \} \nonumber\\
& = - \lambda (\partial_x g_{n }^{(2)} ) \Big( \partial_\xi ( |\xi|^M \chi(\xi)) \Big) + \{ g_{n }^{(2)}, \mu_n + \langle w_n \rangle_\vphi  + w_{n}^{(1)} \} \nonumber\\
& = - M \lambda |\xi|^{M - 2} \xi \chi(\xi) \partial_x g_{n }^{(2)} - \lambda |\xi|^M (\partial_\xi \chi(\xi)) (\partial_x g_{n }^{(2)} ) + \{ g_{n}^{(2)}, \mu_n + \langle w_n \rangle_\vphi  + w_{n}^{(1)} \}\,. \label{espansione parentesi di Poisson g n + 1 (2) M < 1}
\end{align}
Since $\partial_\xi \chi(\xi) = 0$ for $|\xi| \geq 1$ (see \eqref{definizione cut off D alpha}), by Corollary \ref{corollario commutator} and by \eqref{espansione tilde vn finale M < 1}, \eqref{cal G n + 1 (2) M < 1}, one gets that 
\begin{align}
&M \lambda |\xi|^{M - 2} \xi \chi(\xi) \partial_x g_{n }^{(2)} \in S^{M - n \overline{\frak e}}\,, \quad  \lambda |\xi|^M (\partial_\xi \chi(\xi)) (\partial_x g_{n }^{(2)} ) \in S^{- \infty}\,, \quad \{ g_{n }^{(2)}, \mu_n \} \in S^{M - (n + 1)\overline{\frak e}}\,,  \label{wohnung 0 M < 1}\\
&  \{g_{n }^{(2)}, \langle w_n \rangle_\vphi \} \in S^{M - 2 n \overline{\frak e}}\subseteq S^{M - (n + 1)\overline{\frak e}}\,, \quad \{g_{n }^{(2)} , w_{n }^{(1)} \in S^{M - (2 n + 1)\overline{\frak e}} \subseteq S^{M - (n + 1)\overline{\frak e}} \label{wohnung 1 M < 1}\,. 
\end{align}
Thus, the expansion \eqref{espansione v n + 1 M < 1 A} follows by \eqref{espansione tilde vn finale M < 1}, \eqref{prima espansione v n + 1 M < 1}-\eqref{wohnung 1 M < 1}.
\end{proof}
\begin{lemma}\label{lemma equazione omologica spazio lower orders M < 1}
There exists a symbol $g_n^{(2)} \in S^{1 - n \overline{\frak e}}$ with $g_n^{(2)} = (g_n^{(2)})^*$ and satisfying
\begin{equation}\label{equazione omologica spazio lower orders M < 1}
\langle w_n \rangle_\vphi - \langle w_n \rangle_{\vphi, x}  - M \lambda |\xi|^{M - 2} \xi \chi(\xi) \partial_x g_{n }^{(2)} \in S^{M - (n + 1)\overline{\frak e}}
\end{equation}
where 
\begin{equation}\label{vphi x wn lower orders M < 1}
\langle w_n \rangle_{\vphi, x}(\xi) := \frac{1}{2\pi} \int_\T \langle w_n \rangle_\vphi(x, \xi)\, d x = \frac{1}{(2 \pi)^{\nu + 1}} \int_{\T^{\nu + 1}} w_n(\vphi, x, \xi)\, d \vphi \, d x\,. 
\end{equation}
\end{lemma}
\begin{proof}
Let $\chi_0 \in {\cal C}^\infty(\R, \R)$ be a cut-off function satisfying 
\begin{equation}\label{scelta cut off chi 0}
\begin{aligned}
&\chi_0(\xi) = 1 \,, \quad \forall |\xi| \geq 2\,, \\
&\chi_0(\xi) = 0\,, \quad \forall |\xi| \leq \frac32\,. 
\end{aligned}
\end{equation}
Writing $1 = \chi_0 + 1 - \chi_0$, one gets that 
\begin{align}
& - \lambda  M \, |\xi|^{M - 2} \xi \chi(\xi) (\partial_x g_n^{(2)})( x, \xi) + \langle w_n \rangle_\vphi(x, \xi) -  \langle w_n\rangle_{\vphi, x} ( \xi) \nonumber\\
& = - \lambda M \, |\xi|^{M - 2} \xi \chi(\xi) ( \partial_x g_n^{(2)} )(x, \xi) + \chi_0(\xi)\big( \langle w_n \rangle_\vphi ( x, \xi) -  \langle w_n\rangle_{\vphi, x} ( \xi) \big) \nonumber\\
& \quad + \big(1 - \chi_0(\xi) \big)\big( \langle w_n \rangle_\vphi( x, \xi) -  \langle w_n\rangle_{\vphi, x} ( \xi) \big)\,. \label{brotchen 0}
\end{align}
By the definition of $\chi_0$ given in \eqref{scelta cut off chi 0}, one easily gets that  
\begin{equation}\label{resto 1 - chi 0}
\big(1 - \chi_0(\xi) \big)\big( \langle w_n \rangle_\vphi( x, \xi) -  \langle w_n\rangle_{\vphi, x} (\xi) \big) \in S^{- \infty}\,,
\end{equation}
therefore we look for a solution $g_n^{(2)}$ of the equation 
\begin{equation}\label{equazione omologica simbolo g1 b}
- \lambda  M \, |\xi|^{M - 2} \xi \chi(\xi) (\partial_x g_n^{(2)})( x, \xi) + \chi_0(\xi) \big( \langle w_n \rangle_\vphi(x, \xi) -  \langle w_n\rangle_{\vphi, x} ( \xi) \big) \in S^{M - (n + 1) \bar{\frak e} }\,. 
\end{equation}
Since we require that ${\cal G}_n^{(2)} = {\rm Op}(g_n^{(2)})$ is self-adjoint, we look for a symbol of the form 
\begin{equation}\label{forma autoaggiunta g1}
g_n^{(2)}( x, \xi) = \sigma_{n}( x, \xi) + \sigma_{n}^*( x, \xi) \in S^{1 - n  \bar{\frak e}}
\end{equation}
with the property that 
\begin{equation}\label{ansatz g1 equazione omologica}
\sigma_{n}^*(x, \xi) = \sigma_{n}(x, \xi) + r_{n}(x, \xi),\qquad  r_{n} \in S^{- n  \bar{\frak e}}. 
\end{equation}
Plugging the ansatz \eqref{forma autoaggiunta g1} into the equation \eqref{equazione omologica simbolo g1 b}, using \eqref{ansatz g1 equazione omologica} and since  
\begin{equation}\label{brotchen 100}
- \lambda M |\xi|^{M - 2} \xi \chi(\xi) (\partial_x r_{n})(x, \xi) \in S^{M - 1 - n  \bar{\frak e}} \subseteq S^{M - (n + 1)  \bar{\frak e}}, 
\end{equation}
we are led to solve the equation
\begin{equation}\label{equazione per sigma n}
- 2 \lambda M |\xi|^{M - 2} \xi \chi(\xi) (\partial_x \sigma_{n})(x, \xi) + \chi_0(\xi)\Big(\langle w_n \rangle_\vphi( x, \xi) - \langle w_n \rangle_{\vphi, x} ( \xi) \Big) = 0
\end{equation}
whose solution is given by 
\begin{equation}\label{scelta sigma g1}
\sigma_{n}(x, \xi) := 
 \dfrac{\chi_0(\xi) \partial_x^{- 1} \Big[  \langle w_n \rangle_\vphi( x, \xi) -  \langle w_n \rangle_{\vphi, x}( \xi)   \Big]}{2 \lambda M |\xi|^{M - 2} \xi }\,.
 \end{equation}
Since $w_n, \langle w_n\rangle_x \in S^{M - n \bar{\frak e}}$,  $\frac{\chi_0(\xi)}{ |\xi|^{M - 2} \xi} \in S^{ 1 - M} $, one gets that $\sigma_n \in S^{1 - n \bar{\frak e}}$ and hence also $g_n = \sigma_n + \sigma_n^* \in S^{1 - n \bar{\frak e}}$. We now use Lemma \ref{simbolo autoaggiunto fourier multiplier} with $\vphi( \xi) = \frac{\chi_0(\xi)}{2 \lambda M |\xi|^{M - 2} \xi}$, $a( x, \xi) = \partial_x^{- 1} \Big[ \langle w_n \rangle_\vphi ( x, \xi)  -  \langle w_n \rangle_{\vphi, x}( \xi)   \Big]$. 
Recalling that $w_n = w_n^*$, by Lemma \ref{proprieta aggiunti small divisors} one has that $a = a^*$. Hence we can apply Lemma \ref{simbolo autoaggiunto fourier multiplier}, obtaining that the ansatz \eqref{ansatz g1 equazione omologica} is satisfied. By \eqref{brotchen 0}, \eqref{forma autoaggiunta g1}, \eqref{ansatz g1 equazione omologica}, \eqref{equazione per sigma n} one then gets  
$$
\begin{aligned}
& - \lambda  M \, |\xi|^{M - 2} \xi \chi(\xi) (\partial_x g_n^{(2)})( x, \xi) + \langle w_n \rangle_\vphi( x, \xi) -  \langle w_n\rangle_{\vphi, x} (\xi)  \\
& =  \big(1 - \chi_0(\xi) \big)\big( \langle w_n \rangle_\vphi(x, \xi) -  \langle w_n\rangle_{\vphi, x} ( \xi) \big) -  \lambda M |\xi|^{M - 2} \xi \chi(\xi) (\partial_x r_{n})( x, \xi) 
\end{aligned}
$$
and recalling \eqref{resto 1 - chi 0}, \eqref{brotchen 100} one then gets \eqref{equazione omologica spazio lower orders M < 1}. 
\end{proof}
By applying Lemmata \ref{espansione cal V n+1 M < 1 prima eq omologica}, \ref{lemma equazione omologica spazio lower orders M < 1}, one obtains that 
\begin{equation}\label{simbolo v n + 1 finale}
\begin{aligned}
& {\cal V}_{n + 1}(\vphi) = {\rm Op}\Big( v_{n + 1}(\vphi, x, \xi) \Big) \\
& v_{n + 1}(\vphi, x, \xi) = \lambda |\xi|^M \chi(\xi)+ \mu_{n + 1}(\xi) + w_{n + 1}(\vphi, x, \xi) \\
& \mu_{n + 1}(\xi) := \mu_n + \langle w_n \rangle_{\vphi, x} \in S^{M - \overline{\frak e}}\,, \quad w_{n + 1} \in S^{M - (n + 1) \overline{\frak e}}
\end{aligned}
\end{equation}
Since $\Phi_{n + 1}^{(2)}$ is symplectic, then ${\cal V}_{n + 1}(\vphi)$ is self-adjoint. Since $w_n = w_n^*$ and $\mu_n(\xi)$ is real by the induction hypothesis, Lemma \ref{proprieta aggiunti small divisors} implies that $\langle w_n \rangle_{\vphi, x }(\xi)$ is real and therefore $\mu_{n + 1}(\xi) = \mu_n(\xi) + \langle w_n \rangle_{\vphi, x}(\xi)$ is real too. This implies that $${\cal W}_{n + 1}(\vphi) = {\cal V}_{n + 1}(\vphi) - {\rm Op}\Big(\lambda |\xi|^M \chi(\xi) + \mu_{n + 1}(\xi) \Big)$$ 
is self-adjoint. The proof of Proposition \ref{descent method M  < 1} is then concluded by setting $\Phi_{n }(\vphi) := \Phi_{n }^{(1)}(\vphi) \circ \Phi_{n}^{(2)}$ and by applying \eqref{proprieta Phi n (1) M < 1}, \eqref{stima Phi n + 1 (2) M < 1} to obtain \eqref{proprieta Phi n teo 1}. 
\subsection{Proof of Theorem \ref{teorema riduzione}.}
Let $K > 0$ and let us fix the integer $N_K \in \N$ so that $M - N_K \overline{\mathfrak e} < - K$, i.e. 
\begin{equation}\label{teorema principale K M < 1}
N_K := \Big[ \frac{M + K}{\overline{\frak e}} \Big] + 1
\end{equation}
(recall the definition of $\overline{\frak e}$ given in \eqref{costante frak e bar M < 1}) where for any $a \in \R$, we denote by $[a]$ its integer part. Then we define
\begin{equation}\label{trasformazione mappa teo principale M < 1}
\begin{aligned}
& {\cal T}_K(\vphi) : = \Phi_0(\vphi)^{- 1} \circ \Phi_1(\vphi)^{- 1} \circ \ldots \circ \Phi_{N_K - 1}(\vphi)^{- 1}\,, \quad \vphi \in \T^\nu \\
&  {\cal R}_K(\vphi) := {\cal W}_{N_K}(\vphi)\,, \quad \vphi \in \T^\nu  \\
& \lambda_K( D) := \lambda |D|^M + \mu_{N_K}(D)
\end{aligned}
\end{equation}
where the map $\Phi_0(\vphi)$ is given in Proposition \ref{teorema riassunto primo egorov}, $\lambda$ is defined in \eqref{definizione lambda M < 1} and for any $n \in \{1, \ldots, N_K - 1 \}$, $\Phi_n(\vphi)$, ${\cal W}_n(\vphi)$, $\mu_n(D)$ are given in Proposition \ref{descent method M  < 1}. By \eqref{proprieta Phi 2}, \eqref{proprieta Phi n teo 1}, one gets that ${\cal T}_K^{\pm 1} \in {\cal B}(H^s)$, $s \geq 0$ with ${\rm sup}_{\vphi \in \T^\nu} \| {\cal T}_K(\vphi)^{\pm 1} \|_{{\cal B}(H^s)} < + \infty$. Finally, by \eqref{push forward cal V 1 teorema}, \eqref{forma cal Vn teorema}, \eqref{cal V n + 1 cal V n} one obtains \eqref{campo finalissimo cal VK}, with $\lambda_K( D)$, ${\cal R}_K(\vphi)$ defined in \eqref{trasformazione mappa teo principale M < 1}.
\section{Regularization of the vector field $\ii {\cal V}(\vphi)$: the case $M = 1$}\label{regolarizzazione caso M = 1}
In this section we develop the regularization procedure on the vector field $\ii {\cal V}(\vphi)$, ${\cal V}(\vphi) = V(\vphi, x) |D| + {\cal W}(\vphi)$ needed to prove Theorem \ref{teorema riduzione M = 1}. Recall that in order to prove such a theorem, we assume the hypotheses ${\bf (H1)}$, ${\bf (H2)}$, ${\bf (H3)_{M = 1}}$, hence ${\cal V}(\vphi)$ is self-adjoint, ${\cal W} \in OPS^{1 - \frak e}$ for some $\frak e > 0$ and $V(\vphi, x) = 1 + \e P(\vphi, x)$ with $P \in {\cal C}^\infty(\T^{\nu + 1}, \R)$. 
In Section \ref{sezione ordine principale M = 1} we reduce to constant coefficients the highest order term $\ii V(\vphi, x) |D|$, see Proposition \ref{teorema riassunto primo egorov M = 1}. In order to perform such a reduction, we need a smallness condition on the parameter $\e$, since we shall apply Proposition \ref{proposizione principale trasporto}. Then, in Section \ref{sezione ordini bassi M = 1}, we perform the reduction of the lower order terms up to arbitrarily regularizing remainders, see Proposition \ref{iterazione lower orders M = 1 pm}. 
\subsection{Reduction of the highest order}\label{sezione ordine principale M = 1}
In this section we prove the following proposition. We recall that for a function $\mu : \Omega_o \to \R$, $\Omega_o \subseteq \R^\nu$, given $\gamma > 0$, we define 
\begin{equation}\label{Lipschitz norm}
\begin{aligned}
& |\mu|^\Lipg := |\mu|^{\rm sup} + \gamma |\mu|^{\rm lip}\,, \quad |\mu|^{\rm sup} := \sup_{\omega \in \Omega_o} |\mu(\omega)|\,, \quad |\mu|^{\rm lip} := \sup_{\begin{subarray}{c}
\omega_1 , \omega_2 \in \Omega_o \\
\omega_1 \neq \omega_2
\end{subarray}} \frac{|\mu(\omega_1) - \mu(\omega_2)|}{|\omega_1 - \omega_2|}\,. 
\end{aligned}
\end{equation}
In order to state precisely the main result of this section we define the constant $\overline{\frak e} > 0$ as 
\begin{equation}\label{def bar e M = 1}
\overline{\frak e} := {\rm min}\{\frak e, 1 \} \quad \text{so that} \quad  1 - \overline{\frak e} \geq {\rm max}\{1 - \frak e \,,\, 0 \}\,. 
\end{equation}
\begin{proposition}\label{teorema riassunto primo egorov M = 1}
Let $\gamma \in (0, 1)$ and $\tau > \nu$. Then there exists $\delta \in (0, 1)$ such that if $\e \gamma^{- 1} \leq \delta$ the following holds. There exist two Lipschitz functions $\mu_\pm : \Omega \to \R$ satisfying $|\mu_\pm - 1|^\Lipg \lesssim \e$ such that for any $\omega \in \Omega_{\gamma, \tau}^{\mu_+} \cap \Omega_{\gamma, \tau}^{\mu_-}$, where 
\begin{equation}\label{def cal O mu pm gamma}
\Omega_{\gamma, \tau}^{\mu_\pm} := \Big\{ \omega \in \Omega : |\omega \cdot \ell + \mu_\pm(\omega )\, j | \geq \frac{\gamma}{\langle \ell \rangle^\tau}\,, \quad \forall (\ell, j) \in \Z^{\nu + 1} \setminus \{(0,0) \} \Big\}
\end{equation} 
there exists an invertible map $\Phi_0(\vphi)$ satisfying 
\begin{equation}\label{trasformazione ordine principale M = 1}
\sup_{\vphi \in \T^\nu} \| \Phi_0(\vphi)^{\pm 1}\|_{{\cal B}(H^s)}  < +\infty\,, \quad \forall s \geq 0
\end{equation}
such that the push forward of the vector field $\ii {\cal V}(\vphi)$ has the form 
\begin{equation}\label{push forward ordine massimo M = 1}
\begin{aligned}
& (\Phi_0)_{\omega *} (\ii {\cal V})(\vphi) = \ii {\cal V}_1(\vphi)\,,  \\
& {\cal V}_1(\vphi) = \Pi_+ {\cal V}_{1, +}(\vphi) \Pi_+ +  \Pi_- {\cal V}_{1, -}(\vphi) \Pi_- + OPS^{- \infty}\,, \\
& {\cal V}_{1, \pm}(\vphi) = \lambda_{\pm} |D| + {\cal W}_{1, \pm}(\vphi)
\end{aligned}
\end{equation}
where the projection operators $\Pi_+, \Pi_-$ are defined in \eqref{operatori proiezione sui modi positivi negativi azione}, the functions $\lambda_{\pm} : \Omega^\gamma_{\mu_+} \cap \Omega^\gamma_{\mu_-} \to \R$ are Lipschitz and satisfy $|\lambda_{\pm} - 1|^\Lipg \lesssim \e$, ${\cal W}_{1, \pm}(\vphi) \in OPS^{1 - \overline{\frak e}}$ are self-adjoint operators. 
\end{proposition}
The rest of this section is devoted to the proof of Proposition \ref{teorema riassunto primo egorov M = 1}. 
Let $\alpha_+, \alpha_{-} \in {\cal C}^\infty(\T^{\nu + 1}, \R)$ satisfy the ansatz 
\begin{equation}\label{ansatz alpha pm}
\inf_{(\vphi, x) \in \T^{\nu + 1}} \big(1 + (\partial_x \alpha_\pm)(\vphi, x) \big) > 0\,. 
\end{equation}
Then, by applying Lemma \ref{lemma cambi di variabile}, the operators 
\begin{equation}\label{def Phi pm M = 1}
\Phi_\pm(\vphi) := \sqrt{1 + (\partial_x \alpha_\pm)(\vphi, x)} \Psi_\pm(\vphi)\,, \quad \Psi_\pm(\vphi)[u] :=  u(x + \alpha_\pm(\vphi, x))
\end{equation}
are symplectic bounded linear operators $H^s \to H^s$, $s \geq 0$ with inverse given by 
\begin{equation}\label{def Phi pm inverso M = 1}
\Phi_\pm(\vphi)^{- 1}[u] := \sqrt{1 + (\partial_x \widetilde \alpha_\pm)(\vphi, x)} \Psi_{\pm}(\vphi)^{- 1}[u]\,, \quad \Psi(\vphi)^{- 1}[u](y) = u(y + \widetilde \alpha_\pm(\vphi, y))\,, 
\end{equation}
where $y \mapsto y + \widetilde \alpha(\vphi, y)$ is the inverse diffeomorphism of $x \mapsto x + \alpha_\pm(\vphi, x)$ and. By Lemma \ref{diffeo del toro lemma astratto}, $\widetilde \alpha_\pm$ satisfies 
\begin{equation}\label{ansatz widetilde alpha pm}
\inf_{(\vphi, y) \in \T^{\nu + 1}} \big(1 + (\partial_y \widetilde \alpha_\pm)(\vphi, y) \big) > 0\,. 
\end{equation} 
We then consider the operator 
\begin{equation}\label{definizione prima coniugazione caso M = 1}
\Phi_0 ( \vphi) := \Phi_+( \vphi)^{- 1} \Pi_+ + \Phi_{-}( \vphi)^{- 1} \Pi_-\,, \quad \tau \in [0, 1]\,, 
\end{equation}
whose inverse is given by 
\begin{equation}\label{inverso Phi 0 pm M = 1}
\Phi_0 (\vphi)^{- 1} := \Pi_+ \Phi_+( \vphi) + \Pi_{- } \Phi_-( \vphi)\,,
\end{equation}
see Lemma \ref{push forward splitting Pi +-}-$(i)$. The property \eqref{trasformazione ordine principale M = 1} for $\Phi_0(\vphi)^{\pm 1}$ holds by applying Lemma \ref{lemma cambi di variabile}-$(i)$. 
By Lemma \ref{push forward splitting Pi +-}, the push forward $\ii {\cal V}_1(\vphi) := (\Phi_0)_{\omega*} (\ii {\cal V})(\vphi)$ is given by  
\begin{equation}\label{primo cal V0 M = 1}
(\Phi_0)_{\omega*} (\ii {\cal V})(\vphi) = \ii {\cal V}_1(\vphi)\,, \quad {\cal V}_1(\vphi) := \Phi_0(\vphi)^{- 1} {\cal V}(\vphi) \Phi_0(\vphi) + \ii \Phi_0 (\vphi)^{- 1} \omega \cdot \partial_\vphi \Phi_0 (\vphi)\,, 
\end{equation}
\begin{equation}\label{espansion cal V + - 0}
{\cal V}_1(\vphi) = \Pi_+ {\cal V}_{1, +}(\vphi) \Pi_+ + \Pi_- {\cal V}_{1, -}(\vphi) \Pi_- + OPS^{- \infty}
\end{equation}
where 
\begin{equation}\label{definizioni cal V+ - R 0 infty}
 {\cal V}_{1, \pm}(\vphi)  := \Phi_\pm( \vphi) {\cal V}(\vphi) \Phi_\pm( \vphi)^{- 1} + \ii  \Phi_\pm(\vphi) \omega \cdot \partial_\vphi \Phi_\pm( \vphi)^{- 1}  \,.
\end{equation}
Note that, since $\Phi_{\pm}(\vphi)$ are symplectic maps, then ${\cal V}_{1, \pm}(\vphi)$ are self-adjoint operators. This implies that even if the transformation $\Phi(\vphi)$ (see \eqref{definizione prima coniugazione caso M = 1}) is not symplectic, the transformed vector field $\ii {\cal V}_1(\vphi)$ is Hamiltonian {\it up to smoothing operators}. 
\subsubsection{Expansion of ${\cal V}_{1, \pm}(\vphi)$.}
In this section we provide an expansion of the operators ${\cal V}_{1, \pm}(\vphi)$ given in \eqref{definizioni cal V+ - R 0 infty}. 
One has 
\begin{equation}\label{nuova espansione Phi pm cal V}
 \Phi_\pm( \vphi) {\cal V}(\vphi) \Phi_\pm( \vphi)^{- 1} =  \Phi_\pm( \vphi) V(\vphi, x) |D| \Phi_\pm( \vphi)^{- 1} +  \Phi_\pm( \vphi) {\cal W}(\vphi) \Phi_\pm( \vphi)^{- 1}\,. 
\end{equation}
By applying Lemma \ref{lemma cambi di variabile}-$(ii)$, using that $V |D| \in OPS^1$, ${\cal W} \in OPS^{1 - \frak e}$,  one gets that 
\begin{equation}\label{resina fossile 0}
\begin{aligned}
& \Phi_\pm( \vphi) V(\vphi, x) |D| \Phi_\pm( \vphi)^{- 1}  \in OPS^1\,, \quad   \Phi_\pm( \vphi) {\cal W}(\vphi) \Phi_\pm( \vphi)^{- 1} \in  OPS^{1 - \frak e}\,.
\end{aligned}
\end{equation}
In order to compute the highest order term in \eqref{nuova espansione Phi pm cal V} (the term of order $1$), we need to expand the pseudo-differential operator $\Phi_\pm( \vphi) V(\vphi, x)|D| \Phi_\pm( \vphi)^{- 1} \in OPS^1$. We write $|D| = \partial_x {\cal H}$ where ${\cal H}$ is the Hilbert transform, see \eqref{definizione hilbert transform}, \eqref{simbolo esteso hilbert transform}. Hence, 
\begin{align}
\Phi_\pm( \vphi) V(\vphi, x) |D| \Phi_\pm( \vphi)^{- 1} & = \Phi_\pm( \vphi)  V(\vphi, x)\partial_x  \Phi_{\pm}(\vphi)^{- 1} \Phi_{\pm}(\vphi) {\cal H}\Phi_\pm( \vphi)^{- 1} \nonumber\\
& = \Phi_\pm( \vphi) V(\vphi, x) \partial_x  \Phi_{\pm}(\vphi)^{- 1} {\cal H} \nonumber\\
& \quad + \Phi_\pm( \vphi) V(\vphi, x) \partial_x  \Phi_{\pm}(\vphi)^{- 1}\Big( \Phi_{\pm}(\vphi) {\cal H}\Phi_\pm( \vphi)^{- 1} - {\cal H} \Big) \nonumber\\
& = \sqrt{1 + (\partial_x\alpha_{\pm})} \Psi_{\pm}\Big(  V(\vphi, x) (1 + (\partial_y \widetilde \alpha_{\pm}))^{\frac32} \Big) \partial_x{\cal H} + {\cal R}_{\pm, V}(\vphi)\,, \label{resina fossile 1}
\end{align}
\begin{equation}
\begin{aligned}
{\cal R}_{\pm, V}(\vphi) & := \sqrt{1 + (\partial_x\alpha_{\pm})} \Psi_{\pm}\Big(  V(\vphi, x) \partial_y \sqrt{1 + (\partial_y \widetilde \alpha_{\pm})} \Big) \\
& \quad + \Phi_\pm( \vphi) V(\vphi, x)\partial_x  \Phi_{\pm}(\vphi)^{- 1}\Big( \Phi_{\pm}(\vphi) {\cal H}\Phi_\pm( \vphi)^{- 1} - {\cal H} \Big) \,. 
\end{aligned}
\end{equation}
 The property \eqref{identita 1 + alpha alpha tilde partial} applied to $\alpha_\pm$, $\widetilde \alpha_\pm$ implies that 
\begin{align}
& \sqrt{1 + (\partial_x\alpha_{\pm})} \Psi_{\pm}\Big( V (1 + (\partial_y \widetilde \alpha_{\pm}))^{\frac32} \Big) = \Big( V \big( 1 + (\partial_y \widetilde \alpha_{\pm}) \big) \Big)_{| y = x + \alpha_{\pm}(\vphi, x)}\,. \label{coefficiente ordine principale |D| M = 1}
\end{align}
and therefore, $\sqrt{1 + (\partial_x\alpha_{\pm})} \Psi_{\pm}\Big(  V(\vphi, x) \partial_y \sqrt{1 + (\partial_y \widetilde \alpha_{\pm})} \Big)$ is the multiplication operator by the function $\Big( V \big( 1 + (\partial_y \widetilde \alpha_{\pm}) \big) \Big)_{| y = x + \alpha_{\pm}(\vphi, x)}$. 
Moreover, recalling Lemma \ref{hilbert transform cambi di variabile}, the operator 

\noindent
$\Phi_\pm( \vphi) V(\vphi, x)\partial_x  \Phi_{\pm}(\vphi)^{- 1}\Big( \Phi_{\pm}(\vphi) {\cal H}\Phi_\pm( \vphi)^{- 1} - {\cal H} \Big) \in OPS^{- \infty}$, hence ${\cal R}_{\pm, V} \in OPS^0$.
Then, summarizing \eqref{nuova espansione Phi pm cal V}, \eqref{resina fossile 0}, \eqref{resina fossile 1}, \eqref{coefficiente ordine principale |D| M = 1}, recalling that $1 - \overline{\frak e} \geq 1 - \frak e\,,\, 0$ (see \eqref{def bar e M = 1}), so that $OPS^{1 - \frak e} \,,\, OPS^0 \subseteq OPS^{1 - \overline{\frak e}}$, one gets that 
\begin{equation}\label{resina fossile 2}
\begin{aligned}
& \Phi_\pm( \vphi) {\cal V}(\vphi) \Phi_\pm( \vphi)^{- 1} = \Big(  V \big( 1 + (\partial_y \widetilde \alpha_{\pm}) \big) \Big)_{| y = x + \alpha_{\pm}(\vphi, x)} |D| + OPS^{1 - \overline{\frak e}}\,. 
\end{aligned}
\end{equation}
Now we compute the term $\ii  \Phi_\pm(\vphi) \omega \cdot \partial_\vphi \Phi_\pm( \vphi)^{- 1}$ appearing in the definitions of ${\cal V}_{\pm, 1}(\vphi)$ given in \eqref{definizioni cal V+ - R 0 infty}. A direct calculation shows that 
\begin{align}
\ii  \Phi_\pm(\vphi) \omega \cdot \partial_\vphi \Phi_\pm( \vphi)^{- 1} & =  \ii \sqrt{1 + (\partial_x \alpha_{\pm})} \Psi_{\pm}\Big( \omega \cdot \partial_\vphi \widetilde \alpha_{\pm} \sqrt{1 + (\partial_y \widetilde \alpha_{\pm})}  \Big) \partial_x  \nonumber\\
& \quad + \ii \sqrt{1 + (\partial_x \alpha_{\pm})} \Psi_{\pm}\Big( \omega \cdot \partial_\vphi \sqrt{1 + (\partial_y \widetilde \alpha_{\pm})}\Big)\,.  \label{resina fossile 3}
\end{align}
By applying the equality \eqref{identita 1 + alpha alpha tilde partial} to $\alpha_\pm$, $\widetilde \alpha_\pm$, one has 
\begin{equation}\label{resina fossile 4}
\sqrt{1 + (\partial_x \alpha_{\pm})} \Psi_{\pm}\Big( \omega \cdot \partial_\vphi \widetilde \alpha_{\pm} \sqrt{1 + (\partial_y \widetilde \alpha_{\pm})} \Big) = \Big( \omega \cdot \partial_\vphi \widetilde \alpha_{\pm}  \Big)_{y = x + \alpha_{\pm}(\vphi, x)}\,,
\end{equation}
thus \eqref{resina fossile 3} becomes 
\begin{equation}\label{resina fossile 5}
\begin{aligned}
\ii  \Phi_\pm(\vphi) \omega \cdot \partial_\vphi \Phi_\pm( \vphi)^{- 1} = &  \ii \Big( \omega \cdot \partial_\vphi \widetilde \alpha_{\pm}  \Big)_{y = x + \alpha_{\pm}(\vphi, x)} \partial_x + OPS^0\,. 
 \end{aligned}
\end{equation}
Therefore \eqref{definizioni cal V+ - R 0 infty}, \eqref{resina fossile 2}, \eqref{resina fossile 5} and $OPS^0 \subseteq OPS^{1 - \overline{\frak e}}$ imply that 
\begin{equation}\label{forma finale cal V pm 0 M = 1}
\begin{aligned}
{\cal V}_{\pm, 1}(\vphi) & = \Big(  V \big( 1 + (\partial_y \widetilde \alpha_{\pm}) \big) \Big)_{| y = x + \alpha_{\pm}(\vphi, x)} |D| + \ii \Big( \omega \cdot \partial_\vphi \widetilde \alpha_{\pm}  \Big)_{y = x + \alpha_{\pm}(\vphi, x)} \partial_x +  OPS^{1 - \overline{\frak e}}\,. 
\end{aligned}
\end{equation}
We now provide the final expansion of the operator ${\cal V}_1(\vphi)$ defined in \eqref{espansion cal V + - 0}. Using the elementary properties $\ii \partial_x \Pi_+ = - |D| \Pi_+$, $\ii \partial_x \Pi_- = |D| \Pi_-$, by \eqref{forma finale cal V pm 0 M = 1} one gets 
\begin{equation}\label{forma finalissima cal V 0 M = 1 prima eq omologiche}
\begin{aligned}
{\cal V}_1(\vphi)  & = \Pi_+ \Big(\Big(  V( 1 + (\partial_y \widetilde \alpha_{+}) \big) - \omega \cdot \partial_\vphi \widetilde \alpha_{+} \Big)_{| y = x + \alpha_{\pm}(\vphi, x)}  |D| + {\cal W}_{1, +}(\vphi) \Big) \Pi_+ \\
& \quad + \Pi_- \Big(\Big( V \big( 1 + (\partial_y \widetilde \alpha_{-}) \big) + \omega \cdot \partial_\vphi \widetilde \alpha_{-} \Big)_{| y = x + \alpha_{\pm}(\vphi, x)}  |D| + {\cal W}_{1, -}(\vphi) \Big) \Pi_- \\
& \quad + OPS^{- \infty}\,. \\
\end{aligned}
\end{equation} 
with ${\cal W}_{1, \pm} \in OPS^{1 -\overline{\frak e}}$. By applying Proposition \ref{proposizione principale trasporto}, given $\gamma \in (0, 1)$, for $\e \gamma^{- 1} \leq \delta$, for some $\delta \in (0, 1)$ small enough, there exist Lipschitz functions $ \mathtt c_\pm : \Omega \to \R$ such that for any $\omega \in \Omega_{\gamma, \tau}^{\mu_+} \cap \Omega_{\gamma, \tau}^{\mu_-}$ (see \eqref{def cal O mu pm gamma}) there exist two ${\cal C}^\infty$ functions $\widetilde \alpha_{\pm}$ (depending on $\omega$) such that 
\begin{equation}\label{soluzione trasporto alpha pm}
\begin{aligned}
& \mp \omega \cdot \partial_\vphi \widetilde \alpha_{\pm}  +  (1 + \e P)\partial_y \widetilde \alpha_{\pm} + \e P = \mathtt c_\pm  \\
& \| \widetilde \alpha_{\pm}\|_s \lesssim_s \e \gamma^{- 1}\,, \quad \forall s \geq 0\,, \quad |\mathtt c_{\pm}|^\Lipg \lesssim \e\,.
\end{aligned}
\end{equation} 
For $\e \gamma^{-1}$ small enough, $\widetilde \alpha_\pm$ satisfies the property \eqref{ansatz widetilde alpha pm}, therefore by applying Lemma \ref{diffeo del toro lemma astratto}, the functions $\alpha_\pm$ satisfy the ansatz \eqref{ansatz alpha pm}. The operator ${\cal V}_0(\vphi)$ defined in \eqref{forma finalissima cal V 0 M = 1 prima eq omologiche} takes the form 
\begin{equation}\label{forma finalissima cal V 0 M = 1}
\begin{aligned}
& (\Phi_0)_{\omega*} \ii {\cal V}(\vphi) = \ii {\cal V}_1(\vphi)\,, \\
& {\cal V}_1(\vphi)   = \Pi_+ \Big(\lambda_+  |D| + {\cal W}_{1, +}(\vphi) \Big) \Pi_+  + \Pi_- \Big(\lambda_{-} |D| + {\cal W}_{1, -}(\vphi) \Big) \Pi_- + OPS^{- \infty}\,, \\
&  \lambda_{\pm}  := 1 + \mathtt c_{\pm}\,,  \quad |\mathtt c_{\pm}|^\Lipg \lesssim \e\,, \qquad {\cal W}_{1, \pm} \in OPS^{1 - \overline{\frak e}}\,, 
\end{aligned}
\end{equation}
where ${\cal W}_{1, +}(\vphi), {\cal W}_{1, -}(\vphi)$ are self-adjoint. This concludes the proof of Proposition \ref{teorema riassunto primo egorov M = 1}. 
\subsection{Reduction of the lower order terms.}\label{sezione ordini bassi M = 1}
In order to state the main result of this section, for $\gamma \in (0, 1)$ and $\tau > \nu$, we introduce the set 
\begin{equation}\label{Cantor finale riduzione M = 1}
\Omega_{\gamma, \tau} := \Big\{ \omega \in \Omega_{\gamma, \tau}^{\mu_+} \cap \Omega_{\gamma, \tau}^{\mu_-} : |\omega \cdot \ell + \lambda_{\pm}(\omega )\, j| \geq \frac{\gamma}{\langle \ell\rangle^\tau}\,, \quad \forall (\ell, j) \in \Z^{\nu + 1} \setminus \{(0, 0) \} \Big\}\,. 
\end{equation}
\begin{proposition}\label{iterazione lower orders M = 1 pm}
Let $\gamma \in (0, 1)$, $\tau > \nu$, $N \in \N$. Then for any $\omega \in \Omega_{\gamma, \tau}$, for any $n = 1, \ldots, N$ there exists a linear  vector field $\ii {\cal V}_n(\vphi)$, $\vphi \in \T^\nu$ of the form 
\begin{equation}\label{cal Vn con + - nel lemma iterativo M = 1}
{\cal V}_n(\vphi) = \Pi_+ {\cal V}_{n, +}(\vphi) \Pi_+ + \Pi_- {\cal V}_{n, -}(\vphi) \Pi_- + OPS^{- \infty}
\end{equation} 
(recall \eqref{operatori proiezione sui modi positivi negativi azione}) where the  vector fields $\ii {\cal V}_{n, \pm}(\vphi)$ are Hamiltonian and have the form 
\begin{equation}\label{forma cal Vn teorema M = 1}
{\cal V}_{n, \pm}(\vphi) := \lambda_\pm |D| + \mu_{n, \pm}( D) + {\cal W}_{n, \pm}(\vphi)\,,
\end{equation}
where 
\begin{equation}\label{mu n t D teo M = 1}
\mu_{n, \pm}(D) := {\rm Op}\Big( \mu_{n, \pm}( \xi)\Big)\,, \qquad \mu_{n, \pm} \in S^{1 - \overline{\frak e}}\,,
\end{equation}
\begin{equation}\label{cal Wn teorema M = 1}
{\cal W}_{n, \pm}(\vphi) := {\rm Op}\Big( w_{n, \pm}(\vphi, x, \xi)\Big) \,, \qquad w_{n, \pm} \in S^{1 - n \overline{\frak e}}\,,
\end{equation}
 $\mu_{n, \pm}(\xi)$ is real and ${\cal W}_{n, \pm}(\vphi)$ is self-adjoint. 

\noindent
For any $n \in \{1, \ldots, N - 1 \}$, there exist an invertible map $\Phi_{n}(\vphi)$ satisfying
\begin{equation}\label{proprieta Phi n teo M = 1}
\sup_{\vphi \in \T^\nu} \| \Phi_{n}(\vphi)^{\pm 1}\|_{{\cal B}(H^s)} < + \infty\,, \quad \forall s \geq 0
\end{equation}
such that
\begin{equation}\label{cal V n + 1 cal V n M = 1}
\ii {\cal V}_{n + 1}(\vphi) = (\Phi_{n})_{\omega *} \ii {\cal V}_n (\vphi)\,, \qquad \forall n \in \{ 1, \ldots, N -1 \}\,. 
\end{equation}
\end{proposition}
The rest of this section is devoted to the proof of the above proposition, which is proved by induction. We describe the induction step of the proof. Assume that for $n \in \{1, \ldots, N - 1\}$ there exists an operator ${\cal V}_n(\vphi)$ satisfying the properties \eqref{iterazione lower orders M = 1 pm}-\eqref{cal Wn teorema M = 1}. We consider an operator 
\begin{equation}\label{cal G n pm (vphi)}
{\cal G}_{n, \pm}(\vphi) = {\rm Op}\big( g_{n, \pm}(\vphi, x, \xi)\big)\,, \quad g_{n, \pm} \in S^{1 - n \overline{\mathfrak e}}\,, \quad g_{n, \pm} = g_{n, \pm}^*\,. 
\end{equation}
Let $\Phi_{n, \pm}(\tau; \vphi)$, $\tau \in [0, 1]$ be the symplectic flow generated by the Hamiltonian vector field $\ii {\cal G}_{n, \pm}(\vphi)$, namely
\begin{equation}\label{Phi n pm}
\begin{cases}
\partial_\tau \Phi_{n , \pm}(\tau; \vphi) = \ii {\cal G}_{n, \pm}(\vphi) \Phi_{n , \pm}(\tau; \vphi)  \\
\Phi_{n , \pm}(0; \vphi) = {\rm Id}\,. 
\end{cases}
\end{equation}
By applying Lemma \ref{flusso + derivate flusso}, the maps $\Phi_{n, \pm}(\tau; \vphi)$ satisfy the property \eqref{proprieta Phi n teo M = 1}. We then define for any $\tau \in [0, 1]$ the map 
\begin{equation}
\Phi_n(\tau; \vphi) := \Phi_{n , +}(\tau; \vphi)^{- 1} \Pi_+ + \Phi_{n, -}(\tau; \vphi)^{- 1} \Pi_- \,,
\end{equation} 
whose inverse is given by 
\begin{equation}
\Phi_{n}(\tau; \vphi)^{- 1} = \Pi_+ \Phi_{n, +}(\tau; \vphi) + \Pi_- \Phi_{n, -}(\tau; \vphi)\,,
\end{equation}
see Lemma \ref{push forward splitting Pi +-}-$(i)$. We set $\Phi_n(\vphi) := \Phi_n(1; \vphi)$, $\Phi_{n, \pm}(\vphi) := \Phi_{n, \pm}(1; \vphi)$. By applying Lemma \ref{push forward splitting Pi +-}-$(ii)$ one gets that $\ii {\cal V}_{n + 1} = (\Phi_n)_{\omega *} \ii {\cal V}_n$ with 
\begin{equation}\label{prima definizione cal V n + 1 M = 1}
\begin{aligned}
& {\cal V}_{n +1}(\vphi) = \Pi_+ {\cal V}_{n + 1, +}(\vphi) \Pi_+ + \Pi_- {\cal V}_{n + 1, -}(\vphi) \Pi_- + OPS^{- \infty}\,,  \\
& {\cal V}_{n + 1, \pm}(\vphi) := \Phi_{n, \pm}(\vphi) {\cal V}_{n, \pm}(\vphi) \Phi_{n, \pm}(\vphi)^{- 1} + \ii \Phi_{n, \pm}(\vphi) \omega \cdot \partial_\vphi \Phi_{n, \pm}(\vphi)^{- 1}\,. 
\end{aligned}
\end{equation}
Since $\Phi_{n, \pm}(\vphi)$ are symplectic maps, the vector fields $\ii {\cal V}_{n + 1, \pm}(\vphi)$ are Hamiltonian, i.e. ${\cal V}_{n + 1, \pm}(\vphi)$ are self-adjoint operators. In the following, we provide an expansion of the operators ${\cal V}_{n + 1, \pm}(\vphi)$.  

\noindent
{\sc Analysis of $ \Phi_{n, \pm}(\vphi) {\cal V}_{n, \pm}(\vphi) \Phi_{n, \pm}(\vphi)^{- 1} $.} The symbol $v_{n, \pm} \in S^1$ of the operator ${\cal V}_{n, \pm}$ is given by 
\begin{equation}\label{brotchen - 1}
v_{n, \pm} (\vphi, x, \xi) = \lambda_{\pm} |\xi| \chi(\xi) + \mu_{n, \pm}(\xi) + w_{n, \pm}(\vphi, x, \xi)\,.
\end{equation}
Applying Proposition \ref{teorema egorov campo minore di uno} one gets that 
\begin{equation}\label{brotchen 1}
\Phi_{n, \pm}(\vphi) {\cal V}_{n, \pm}(\vphi) \Phi_{n, \pm}(\vphi)^{- 1}  = {\rm Op}\big( v_{n + 1 , \pm}^{(1)} \big)\,,
\end{equation}
where
\begin{equation}\label{brotchen 2}
 v_{n + 1 , \pm}^{(1)} := v_{n, \pm} + \{ g_{n, \pm}, v_{n, \pm} \} + r_{n, \pm}^{(\geq 2)}\,, \quad r_{n, \pm}^{(\geq 2)} \in S^{1 - 2 n \overline{\frak e}} \subseteq S^{1 - ( n + 1) \overline{\frak e}}\,. 
\end{equation}
We write 
\begin{align}
\{ g_{n, \pm}, v_{n, \pm} \} & = \lambda_{\pm} \{ g_{n, \pm},  |\xi| \chi(\xi)\} + \{ g_{n, \pm}, \mu_{n , \pm}\} + \{ g_{n, \pm},  w_{n, \pm}\} \nonumber\\
& = - \lambda_{\pm} \partial_x g_{n, \pm} \frac{\xi}{|\xi|} \chi(\xi) - \lambda_{\pm} \partial_x g_{n, \pm} |\xi| \partial_\xi \chi(\xi) + \{ g_{n, \pm}, \mu_{n , \pm}\} + \{ g_{n, \pm},  w_{n, \pm}\} \label{brotchen 3}
\end{align}
By Corollary \ref{corollario commutator} and since $\partial_\xi \chi(\xi) = 0$ for $|\xi| \geq 1$ (see \eqref{definizione cut off D alpha}), $\mu_{n, \pm} \in S^{1 - \overline{\frak e}}$ and $w_{n, \pm}, g_{n, \pm} \in S^{1 - n \overline{\frak e}}$, one gets that 
\begin{equation}\label{brotchen 4}
\begin{aligned}
& - \lambda_{\pm} \partial_x g_{n, \pm} |\xi| \partial_\xi \chi(\xi) \in S^{- \infty} \subseteq S^{1 - (n + 1) \overline{\frak e}}\,, \quad \{ g_{n, \pm}, \mu_{n , \pm}\} \in S^{1 - (n + 1) \overline{\frak e}}\,, \\
& \{ g_{n, \pm},  w_{n, \pm}\} \in S^{1 - 2 n \overline{\frak e}} \subseteq S^{1 - (n + 1) \overline{\frak e}}\,. 
\end{aligned}
\end{equation}
By \eqref{brotchen - 1}, \eqref{brotchen 2}--\eqref{brotchen 4} one gets 
\begin{equation}\label{brotchen 5}
\begin{aligned}
& v_{n + 1, \pm}^{( 1 )} = \lambda_{\pm} |\xi | \chi(\xi) + \mu_{n, \pm} + w_{n, \pm} - \lambda_{\pm} \partial_x g_{n, \pm} \frac{\xi}{|\xi|} \chi(\xi) + S^{1 - (n + 1) \overline{\frak e}}\,. 
\end{aligned}
\end{equation}

\medskip

\noindent
{\sc Analysis of the term $\ii \Phi_{n, \pm}(\vphi) \omega \cdot \partial_\vphi \Phi_{n, \pm}(\vphi)^{- 1}$.}  We define ${\cal V}_{n + 1, \pm}^{(2)}(\tau; \vphi) := \ii  \Phi_{n, \pm}(\tau; \vphi) \omega \cdot \partial_\vphi \big( \Phi_{n, \pm}(\tau; \vphi)^{- 1} \big)$. A direct calculation shows that 
$$
{\cal V}_{n + 1, \pm}^{(2 )}(\tau; \vphi) =  \int_0^\tau {\cal S}_{n + 1, \pm}(\zeta; \vphi )\, d \zeta\,, \quad  {\cal S}_{n+ 1, \pm}(\zeta; \vphi) := \Phi_{n, \pm}(\zeta; \vphi) \circ  \omega \cdot \partial_\vphi  {\cal G}_{n ,\pm}(\vphi) \circ \Phi_{n , \pm}(\zeta; \vphi)^{- 1}, \quad \zeta \in [0, 1]\,. 
$$
Since $\omega \cdot \partial_\vphi {\cal G}_{n , \pm}(\vphi) \in OPS^{1 - n \overline{\frak e}}$, by Proposition \ref{teorema egorov campo minore di uno} 
\begin{equation}
{\cal V}_{n + 1, \pm}^{(2 )}( \vphi) = {\cal V}_{n + 1, \pm}^{(2 )}(1 ; \vphi) = {\rm Op}\Big(v_{n + 1, \pm}^{(2)}(\vphi, x, \xi) \Big) \in OPS^{1 - n \overline{\frak e}} 
\end{equation}
with 
\begin{equation}\label{brotchen 6}
v_{n + 1, \pm}^{(2 )}(\vphi, x, \xi)  := \omega \cdot \partial_\vphi g_{n, \pm}(\vphi, x, \xi)+ S^{1 - (n + 1) \overline{\frak e}}\,. 
\end{equation}
Collecting \eqref{prima definizione cal V n + 1 M = 1}, \eqref{brotchen 1}, \eqref{brotchen 2}, \eqref{brotchen 6} one then gets 
\begin{equation}\label{brotchen 7}
\begin{aligned}
v_{n + 1, \pm} & = \lambda_{\pm} |\xi | \chi(\xi) + \mu_{n, \pm} + w_{n, \pm} + \omega \cdot \partial_\vphi g_{n, \pm} - \lambda_{\pm} \partial_x g_{n, \pm} \frac{\xi}{|\xi|} \chi(\xi) + S^{1 - (n + 1) \overline{\frak e}}\,. 
\end{aligned}
\end{equation}
In the next Lemma we show how to choose ${\cal G}_{m, \pm}(\vphi) = {\rm Op}\big( g_{n, \pm}(\vphi, x, \xi)\big)$ in order to reduce to constant coefficients the term $w_{n, \pm} + \omega \cdot \partial_\vphi g_{n, \pm} - \lambda_{\pm} \partial_x g_{n, \pm} \frac{\xi}{|\xi|} \chi(\xi) $ of order $1 - n \overline{\frak e}$ in \eqref{brotchen 7}. 
\begin{lemma}\label{lemma equazione omologica ordini bassi M = 1}
For any $\omega \in \Omega_{\gamma, \tau}$ (see \eqref{Cantor finale riduzione M = 1}) there exist symbols $g_{n + 1, \pm} \in S^{1 - n \overline{\frak e}}$ satisfying $g_{n + 1, \pm} = g_{n + 1, \pm}^*$ and 
$$
 w_{n, \pm} + \omega \cdot \partial_\vphi g_{n, \pm} - \lambda_{\pm} \partial_x g_{n, \pm} \frac{\xi}{|\xi|} \chi(\xi) - \langle w_{n, \pm} \rangle_{\vphi, x} \in S^{1 - (n + 1) \overline{\frak e}}
$$
(recall the definition \eqref{simbolo mediato}).  
\end{lemma}
\begin{proof}
Let $\chi_1 \in {\cal C}^\infty(\R, \R)$ be a cut-off function satisfying: 
\begin{equation}\label{altro cut off chi1}
\begin{aligned}
& \chi_1(\xi) = 0\,, \quad \forall |\xi| \leq \frac32\,, \\
& \chi_1(\xi ) = 1\,, \quad \forall |\xi| \geq 2\,. 
\end{aligned}
\end{equation}
Writing $1 =  \chi_1 + (1 - \chi_1) $, one has 
\begin{align}
& w_{n, \pm} + \omega \cdot \partial_\vphi g_{n, \pm} - \lambda_{\pm} \partial_x g_{n, \pm} \frac{\xi}{|\xi|} \chi(\xi) - \langle w_{n, \pm} \rangle_{\vphi, x} \nonumber\\
& = \chi_1 (w_{n, \pm} - \langle w_{n, \pm} \rangle_{\vphi, x} )+ \omega \cdot \partial_\vphi g_{n, \pm} - \lambda_{\pm} \partial_x g_{n, \pm} \frac{\xi}{|\xi|} \chi(\xi)  + (1 - \chi_1)(w_{n, \pm} - \langle w_{n, \pm} \rangle_{\vphi, x} ) \,. \label{ursula 0}
\end{align}
By the definition of $\chi_1$ given in \eqref{altro cut off chi1}, one has that 
\begin{equation}\label{ursula 1} 
(1 - \chi_1)(w_{n, \pm} - \langle w_{n, \pm} \rangle_{\vphi, x} )\in S^{- \infty},
\end{equation} 
therefore we look for a symbol $g_{n, \pm} $ satisfying 
\begin{equation}\label{ursula 2}
\chi_1 (w_{n, \pm} - \langle w_{n, \pm} \rangle_{\vphi, x} )+ \omega \cdot \partial_\vphi g_{n, \pm} - \lambda_{\pm} \partial_x g_{n, \pm} \frac{\xi}{|\xi|} \chi(\xi) \in S^{1 - (n + 1) \overline{\frak e}}\,. 
\end{equation}
Since we require that $g_{n, \pm} = g_{n, \pm}^*$, we look for a symbol of the form $g_{n, \pm} = \frac{q_{n, \pm} + q_{n, \pm}^*}{2}$ and we make the ansatz 
\begin{equation}\label{ursula 3}
q_{n, \pm}^* =  q_{n, \pm} +   S^{ - n \overline{\frak e}}\,. 
\end{equation}
By \eqref{ursula 2}, \eqref{ursula 3}, we write 
\begin{align}
& \chi_1 (w_{n, \pm} - \langle w_{n, \pm} \rangle_{\vphi, x} )+ \omega \cdot \partial_\vphi g_{n, \pm} - \lambda_{\pm} \partial_x g_{n, \pm} \frac{\xi}{|\xi|} \chi(\xi) \nonumber\\
& = \chi_1 (w_{n, \pm} - \langle w_{n, \pm} \rangle_{\vphi, x} )+ \omega \cdot \partial_\vphi q_{n, \pm} - \lambda_{\pm} \partial_x q_{n, \pm} \frac{\xi}{|\xi|} \chi(\xi) \nonumber\\
& \quad + \Big( \omega \cdot \partial_\vphi - \lambda_{\pm} \frac{\xi}{|\xi|} \chi(\xi) \partial_x \Big) \Big[ \frac{q_n^* - q_n}{2} \Big]\,.  \label{ursula 4}
\end{align}
By \eqref{ursula 3} we have that 
\begin{equation}\label{ursula 5}
\Big( \omega \cdot \partial_\vphi - \lambda_{\pm} \frac{\xi}{|\xi|} \chi(\xi) \partial_x \Big) \Big[ \frac{q_n^* - q_n}{2} \Big] \in S^{ -n  \overline{\frak e}} \stackrel{\overline{\frak e} \leq 1}{\subseteq} S^{1 - (n + 1) \overline{\frak e}}\,,
\end{equation}
hence it is enough to determine the symbol $q_{n, \pm}$ so that 
\begin{equation}\label{scelta equazione omologica q n pm}
\chi_1 (w_{n, \pm} - \langle w_{n, \pm} \rangle_{\vphi, x} )+ \omega \cdot \partial_\vphi q_{n, \pm} - \lambda_{\pm} \partial_x q_{n, \pm} \frac{\xi}{|\xi|} \chi(\xi) = 0 \,.
\end{equation}
The equation above, can be solved for any $\omega \in \Omega_{\gamma, \tau}$, by defining $q_{n, \pm}$ as 
\begin{equation}\label{definizione q n pm}
\begin{aligned}
q_{n, \pm}(\vphi, x, \xi) & := (\omega \cdot \partial_\vphi - \lambda_{\pm} \partial_x)^{- 1}\big[  \langle w_{n, \pm}  \rangle_{\vphi, x}(\xi) - w_{n, \pm}(\vphi, x, \xi) \big] \chi_1^+(\xi) \\
& \quad  + (\omega \cdot \partial_\vphi +  \lambda_{\pm} \partial_x)^{- 1}\big[ \langle w_{n, \pm}  \rangle_{\vphi, x}(\xi) - w_{n, \pm}(\vphi, x, \xi)   \big] \chi_1^-(\xi)
\end{aligned}
\end{equation}
where $\chi_1^+(\xi) := \chi_1(\xi) {\mathbb I}_{\{\xi > 0 \}}$, $\chi_1^-(\xi) := \chi_1(\xi) {\mathbb I}_{\{\xi \leq 0\}}$ where ${\mathbb I}_{\{\xi > 0 \}}$ (resp. ${\mathbb I}_{\{\xi \leq 0\}}$) is the characteristic function of the set $\{ \xi \in \R : \xi > 0 \}$ (resp. $\{ \xi \in \R : \xi \leq 0 \}$). Note that, by \eqref{altro cut off chi1}, the functions $\chi_1^+, \chi_1^-$ are $C^\infty$. 
Using that $w_{n, \pm} \in S^{1 - n \overline{\frak e}}$ and recalling the property \eqref{proprieta simbolo mediato e partial x - 1}, one has that $q_{n, \pm} \in S^{1 - n \overline{\frak e}}$ and therefore $g_{n , \pm} \in S^{1 - n \overline{\frak e}}$. Furthermore, using that $w_{n, \pm} = w_{n, \pm}^*$, by applying Lemma \ref{simbolo autoaggiunto fourier multiplier} (with $\vphi(\xi) = \chi_{1}^+(\xi)$, $a = (\omega \cdot \partial_\vphi - \lambda_{\pm} \partial_x)^{- 1}\big[  \langle w_{n, \pm}  \rangle_{\vphi, x} - w_{n, \pm} \big]$ and $\vphi(\xi) = \chi_1^{-}(\xi)$, $a = (\omega \cdot \partial_\vphi + \lambda_{\pm} \partial_x)^{- 1}\big[  \langle w_{n, \pm}  \rangle_{\vphi, x} - w_{n, \pm}  \big]$) and Lemma \ref{proprieta aggiunti small divisors}, one gets that the symbols $q_{n, \pm}$ verify the ansatz \eqref{ursula 3}. Finally collecting \eqref{ursula 0}, \eqref{ursula 1}, \eqref{ursula 4}, \eqref{ursula 5}, \eqref{scelta equazione omologica q n pm}, the claimed statement follows. 
\end{proof}  
By \eqref{prima definizione cal V n + 1 M = 1}, \eqref{brotchen 7} and Lemma \ref{lemma equazione omologica ordini bassi M = 1}, one gets that 
\begin{align}
{\cal V}_{n + 1}(\vphi) = \Pi_+ {\cal V}_{n + 1, +}(\vphi) \Pi_+ + \Pi_- {\cal V}_{n + 1, -}(\vphi) \Pi_- + OPS^{- \infty} 
\end{align}
where 
\begin{equation}\label{forma finalissima cal V n + 1 pm M = 1}
\begin{aligned}
& {\cal V}_{n + 1, \pm}(\vphi) = \lambda_{\pm} |D| + \mu_{n + 1, \pm}(D) + {\cal W}_{n + 1, \pm}(\vphi)\,, \\
& \mu_{n + 1, \pm}(D) := \mu_{n, \pm}(D) +  {\rm Op}\big( \langle w_{n, \pm}  \rangle_{\vphi, x}(\xi) \big)\,, \quad  {\cal W}_{n + 1, \pm}(\vphi) \in OPS^{1 - (n + 1) \overline{\frak e}}\,. 
\end{aligned}
\end{equation}
By the induction hypothesis $w_{n , \pm} = w_{n , \pm}^*$ and $\mu_{n, \pm}(\xi)$ is real, hence by Lemma \ref{proprieta aggiunti small divisors} one has that $\langle w_{n, \pm}  \rangle_{\vphi, x}(\xi)$ is real and therefore $\mu_{n + 1, \pm}(\xi) = \mu_{n, \pm}(\xi) + \langle w_{n, \pm}  \rangle_{\vphi, x}(\xi)$ is real too. Since $\Phi_{n, \pm}$ are symplectic maps, $\ii {\cal V}_{n \pm}$ are Hamiltonian vector fields, then also $\ii {\cal V}_{n + 1, \pm} = (\Phi_{n, \pm}^{- 1})_{\omega *} \ii {\cal V}_{n, \pm}$ are Hamiltonian vector fields, implying that 
$$
{\cal W}_{n + 1, \pm}(\vphi) =  {\cal V}_{n + 1, \pm}(\vphi) -  \lambda_{\pm} |D| - \mu_{n + 1, \pm}(D)
$$
are self-adjoint operators. The proof of Proposition \ref{iterazione lower orders M = 1 pm} is then concluded. 
\subsection{Proof of Theorem \ref{teorema riduzione M = 1}.}
Let $K > 1$ and let fix an integer $N_K$ so that  $1 - N_K \overline{\frak e} \leq - K$ 
\begin{equation}\label{costante N K nela caso M = 1}
N_K :=  \Big[ \frac{K + 1}{\overline{\frak e}}\Big] + 1\,. 
\end{equation}
We define 
\begin{equation}\label{cal TK M = 1 def nella proof}
{\cal T}_K(\vphi) := \Phi_0 (\vphi) \circ \Phi_1(\vphi) \circ \ldots \circ \Phi_{N_K - 1}(\vphi)
\end{equation}
where $\Phi_0(\vphi)$ is defined in \eqref{definizione prima coniugazione caso M = 1} and for any $n = 1, \ldots, N_K - 1$, the maps $\Phi_n(\vphi)$ are given in Lemma \ref{iterazione lower orders M = 1 pm}. By \eqref{inverso Phi 0 pm M = 1}, \eqref{trasformazione ordine principale M = 1}, \eqref{proprieta Phi n teo M = 1}, the map ${\cal T}_K(\vphi)$ is invertible with inverse given by ${\cal T}_K(\vphi)^{- 1} =  \Phi_{N_K}(\vphi)^{- 1}  \circ \ldots \circ \Phi_1(\vphi)^{- 1}\circ \Phi_0(\vphi)^{- 1}$ and ${\cal T}_K(\vphi)^{\pm 1}$ satisfy 
\begin{equation}
\sup_{\vphi \in \T^\nu} \| {\cal T}_K(\vphi)^{\pm 1}\|_{{\cal B}(H^s)} < + \infty\,, \quad \forall s \geq 0\,. 
\end{equation}
By \eqref{forma finalissima cal V 0 M = 1} and by Lemma \ref{iterazione lower orders M = 1 pm}, one gets that $({\cal T}_K)_{\omega*} \ii {\cal V}(\vphi) = \ii {\cal V}_{N_K}(\vphi)$ where ${\cal V}_{N_K}(\vphi)$ is given by formula \eqref{cal Vn con + - nel lemma iterativo M = 1} for $n = N_K$. Then we can write 
$$
{\cal V}_{N_K}(\vphi) = \lambda_K(D) + {\cal R}_K(\vphi)
$$
where, recalling that $|D| = {\rm Op}(|\xi| \chi(\xi))$ (see \eqref{definizione D alpha}) and $\Pi_{\pm} = {\rm Op}(\chi_{\pm}(\xi))$ (see \eqref{operatori proiezione sui modi positivi negativi})
$$
\begin{aligned}
& \lambda_K(D) = {\rm Op}(\lambda_K(\xi))\,, \quad \lambda_K(\xi) :=  \big(\lambda_+ |\xi| \chi(\xi) + \mu_{N_K, +}(\xi) \big) \chi_+(\xi) +  \big(\lambda_- |\xi| \chi(\xi) + \mu_{N_K, -}(\xi) \big) \chi_-(\xi) \,, \\
& {\cal R}_K(\vphi) := \Pi_+ {\cal W}_{N_K, +}(\vphi) \Pi_+ + \Pi_- {\cal W}_{N_K, -}(\vphi) \Pi_- + OPS^{- \infty} \,. 
\end{aligned}
$$
Note that $\lambda_K(\xi)$ is real. By \eqref{cal Wn teorema M = 1}  (applied with $n = N_K$), using that $1 - N_K \overline{\frak e} \leq - K$ and recalling Theorem \ref{conitnuita pseudo}, one gets that ${\cal R}_K \in OPS^{- K}$. The proof of Theorem \ref{teorema riduzione M = 1} is then concluded. 
\section{Proof of Theorems \ref{teo growth of sobolev norms M < 1}, \ref{teo growth of sobolev norms M = 1}.}\label{proof crescita}
Let $s > 0$, $t_0 \in \R$, $u_0 \in H^s(\T)$. We fix the constant $K \in \N$ appearing in Theorem \ref{teorema riduzione} so that $K > s$,  
\begin{equation}\label{relazione K s}
K = K_s :=[s] + 1\,. 
\end{equation}
By applying Theorems \ref{teorema riduzione}, \ref{teorema riduzione M = 1}, taking $\omega \in DC(\gamma, \tau)$ in the case $M < 1$ and $\omega \in \Omega_{\gamma, \tau}$, $\e \gamma^{- 1}$ small enough, in the case $M = 1$, one has that $u(t)$ is a solution of the Cauchy problem 
\begin{equation}\label{cauchy problem prova main theorem}
\begin{cases}
\partial_t u = \ii {\cal V}(\omega t)[u] \\
u(t_0) = u_0
\end{cases}
\end{equation}
if and only if $v(t) := {\cal T}_{K_s}^{- 1}(\omega t) [u(t)]$ is a solution of the Cauchy problem 
\begin{equation}\label{cauchy problem trasformato main theorem}
\begin{cases}
\partial_t v = \ii \lambda_{K_s}(D) v + \ii {\cal R}_{K_s}(\omega t)[v] \\
v(t_0) = v_0\,,
\end{cases} \qquad v_0 := {\cal T}_{K_s}^{- 1}(\omega t_0) [u_0]\,
\end{equation}
with $\lambda_{K_s}(D) = {\rm Op}(\lambda_{K_s}( \xi)) \in OPS^M$ with $\lambda_{K_s}(\xi) = \overline{\lambda_{K_s}( \xi)}$. 
Moreover, since $K_s > s > 0$, one has that ${\cal R}_{K_s} \in OPS^{- K_s} \subset OPS^{- s}$ implying that 
\begin{equation}\label{proprieta cal W Ks teorema}
 \quad \sup_{\vphi \in \T^\nu} \| {\cal 	R}_{K_s}(\vphi)\|_{{\cal B}(L^2, H^s)} < + \infty\,. 
\end{equation}
Moreover, since ${\cal T}_{K_s}(\vphi)$ is bounded and invertible, we have that 
\begin{equation}\label{u t v t}
\| u(t)\|_{H^s} \simeq_s \| v (t)\|_{H^s} \quad \text{and} \quad  \| u(t)\|_{L^2} \simeq \| v (t)\|_{L^2}\,, \quad \forall t \in \R\,.  
\end{equation}
Writing the Duhamel formula for the Cauchy problem \eqref{cauchy problem trasformato main theorem}, one obtains 
\begin{equation}\label{duhamel equazione ridotta}
v(t) = e^{\ii \lambda_{K_s}(D) t } v_0 + \int_{t_0}^t e^{\ii \big(t - \tau \big)\lambda_{K_s}( D) } {\cal R}_{K_s}(\omega \tau)[v(\tau)]\, d \tau\,. 
\end{equation} 
Using that $\lambda_{K_s}(\xi)$ is real which implies that the propagator $e^{\ii t \lambda_{K_s}( D)}$ is unitary on $H^s(\T)$ for any $t \in \R$ and by \eqref{proprieta cal W Ks teorema}, \eqref{u t v t} one gets the estimate $\| v(t) \|_{H^s}  \lesssim_s \| v_0\|_{H^s} +  |t - t_0| \| v_0\|_{L^2}$. Applying again \eqref{u t v t} we then get 
\begin{equation}\label{u t stima lineare in t}
\| u(t) \|_{H^s}  \lesssim_s \| u_0\|_{H^s} +  |t - t_0| \| u_0\|_{L^2}\,, \quad \forall t \in \R\,. 
\end{equation}
The above argument implies that the propagator ${\cal U}(t_0, t)$ of the PDE $\partial_t u = \ii {\cal V}(\omega t)[u]$, i.e. 
$$
\begin{cases}
\partial_t {\cal U}(t_0, t) = \ii {\cal V}(\omega t) {\cal U}(t_0, t) \\
{\cal U}(t_0, t_0) = {\rm Id}
\end{cases}
$$ 
satisfies 
\begin{equation}\label{propagatore cal U Hs}
\| {\cal U}(t_0, t)\|_{{\cal B}(H^s)} \lesssim_s 1 + |t - t_0|\,, \quad \forall s > 0\,, \quad \forall t, t_0 \in \R\,. 
\end{equation}
 Furthermore, since ${\cal V}(\vphi)$ is self-adjoint, the $L^2$ norm of the solutions is constant, namely 
\begin{equation}\label{propagatore cal U L2}
\| {\cal U}(t_0, t)\|_{{\cal B}(L^2)}= 1\,, \quad  \forall t, t_0 \in \R\,.
\end{equation}
Hence, for any $0 < s < S$, by applying Theorem \ref{interpolazione sobolev}, one gets that 
\begin{align}
\| {\cal U}(t_0, t)\|_{{\cal B}(H^s)} & \leq \| {\cal U}(t_0, t)\|_{{\cal B}(L^2)}^{\frac{S - s}{S}} \| {\cal U}(t_0, t)\|_{{\cal B}(H^S)}^{\frac{s}{S}} \stackrel{\eqref{propagatore cal U Hs}, \eqref{propagatore cal U L2}}{\lesssim_S} (1 + |t - t_0|)^{\frac{s}{S}}\,.
\end{align}
Then, for any $\eta > 0$, choosing $S$ large enough so that $s/S \leq \eta$, the estimate \eqref{crescita sobolev generale} follows. 

\section{Appendix: a quasi-periodic transport equation}\label{sezione riducibilita equazione trasporto}
In this appendix we state some results concerning quasi-periodic transport equations.
The following statement is a direct consequence of Corollary 4.3 in \cite{trasporto paper}. 
\begin{proposition}\label{proposizione principale trasporto}
Let $\gamma  \in (0, 1)$, $\tau > \nu$, $P \in {\cal C}^\infty(\T^{\nu + 1}, \R)$, $\mathtt m \in \R$, $\frac12 < |\mathtt m| < 2$. There exists a constant $\delta_* = \delta_*(\tau, \nu) > 0$, such that if $\e \gamma^{- 1} \leq \delta_*$, then the following holds: there exists a Lipschitz function $\mu: \Omega \to \R$ satisfying $|\mu - \mathtt m |^\Lipg \lesssim \e $ (recall the definition \eqref{Lipschitz norm}) such that for any $\omega$ in the set 
\begin{equation}\label{Cantor del trasporto}
\Omega_{\gamma, \tau}^{\mu} := \Big\{ \omega \in \Omega : |\omega \cdot \ell + \mu (\omega) \, j | \geq \frac{\gamma}{\langle \ell \rangle^\tau}\,, \quad \forall (\ell, j) \in \Z^{\nu + 1} \setminus \{(0,0) \} \Big\}
\end{equation}
there exists a ${\cal C}^\infty$ function $\alpha(\vphi, x; \omega)$ satisfying $\| \alpha\|_s \lesssim_s \e \gamma^{- 1}$, $\forall s \geq 0$ and a Lipschitz family of constants $\mathtt c(\omega)$ satisfying $|\mathtt c|^\Lipg \lesssim \e $, such that 
\begin{equation}\label{equazione trasporto per applicazione}
\omega \cdot \partial_\vphi \alpha(\vphi, x) + \Big(\mathtt m + \e P(\vphi, x) \Big) \partial_x \alpha(\vphi, x) + \e P(\vphi, x) = \mathtt c\,, \quad \forall (\vphi, x) \in \T^\nu \times \T\,. 
\end{equation}
\end{proposition}

\begin{thebibliography}{10}
\bibitem{BBM-Airy}
P. Baldi, M. Berti, R. Montalto,
\emph{KAM  for quasi-linear and fully nonlinear forced perturbations of Airy equation}. 
Math. Annalen 359, 471-536, 2014. 
  
\bibitem{BBM-auto}
P. Baldi, M. Berti, R. Montalto,
\emph{KAM for autonomous quasi-linear perturbations of KdV}.
Ann. I. H. Poincar\'e (C) Anal. Non Lin\'eaire 33, 1589-1638, 2016.  

\bibitem{BBHM} P. Baldi, M. Berti, E. Haus, R. Montalto, \emph{Time quasi-periodic gravity water waves in finite depth.} Preprint arXiv:1708.01517, 2017. 

\bibitem{Bambusi1} D. Bambusi, \emph{Reducibility of 1-d Schr\"odinger equation with time quasiperiodic unbounded perturbations, I}. Trans. Amer. Math. Soc., doi:10.1090/tran/7135, 2017. 

\bibitem{Bambusi2} D. Bambusi, \emph{Reducibility of 1-d Schr\"odinger equation with time quasiperiodic unbounded perturbations, II}. Comm. in Math. Phys. doi:10.1007/s00220-016-2825-2, 2017.

\bibitem{albertino3} D. Bambusi, B. Grebert, A. Maspero, D. Robert, \emph{Reducibility of the Quantum Harmonic Oscillator in
d-dimensions with Polynomial Time Dependent
Perturbation}, Analysis and PDEs, 11(3): 775Ð799, 2018. 

\bibitem{albertino2}D. Bambusi, B. Grebert, A. Maspero, D. Robert, \emph{Growth of Sobolev norms for abstract linear Schr\"odinger Equations}, preprint 2017. 
\bibitem{BM16} 
M. Berti, R. Montalto, 
\newblock \emph{Quasi-periodic water waves}. J. Fixed Point Theory Appl., 
 19, no. 1, 129-156, 2017.

\bibitem{BertiMontalto} M. Berti, R. Montalto, \emph{Quasi-periodic standing wave solutions for gravity-capillary water waves}, to appear on Memoirs of the Amer. Math. Society. MEMO 891. Preprint arXiv:1602.02411v1, 2016. 

\bibitem{bourgain-1} J. Bourgain, {\it Growth of Sobolev norms in linear Schr\"odinger equations with quasi periodic potential.} Comm. in Math. Phys. 204, no. 1, 207-247, 1999. 

\bibitem{bourgain-2} J. Bourgain, {\it On growth of Sobolev norms in linear Schr\"odinger equations with smooth time dependent potential.} Journal d'Analyse Math\'ematique 77, 315-348, 1999.


\bibitem{delort} J.M. Delort, {\it Growth of Sobolev Norms of Solutions of Linear Schr\"odinger Equations on Some Compact Manifolds.} Int. Math. Res. Notices, Vol. 2010, No. 12, pp. 2305-2328. 


\bibitem{EK1} L. H. Eliasson, S. Kuksin, {\it On reducibility of Schr\"odinger equations with
quasiperiodic in time potentials}. Comm. Math. Phys. 286, 125-135, 2009. 
\bibitem{Feola}
R. Feola,
\emph{KAM for quasi-linear forced hamiltonian NLS}. Preprint arXiv:1602.01341, 2016. 

\bibitem{Feola-Procesi}
R. Feola, M. Procesi,
\emph{Quasi-periodic solutions for fully nonlinear forced reversible Schr{\"o}dinger equations}. J. Differential Equations 259, no.\,7, 3389-3447, 2015. 
\bibitem{trasporto paper} R. Feola, F. Giuliani, R. Montalto, M. Procesi, \emph{Reducibility of first order linear operators on tori via Moser's theorem.} Preprint  arXiv:1801.04224, 2018.
\bibitem{Giuliani} F. Giuliani, {\it Quasi-periodic solutions for quasi-linear generalized KdV equations.} J. Differential Equations 262,  5052-5132, 2017.

\bibitem{GrebertPaturel} B. Grebert, E. Paturel, {\it On reducibility of quantum harmonic oscillator on $\R^d$ with quasiperiodic in time potential}. Preprint arXiv:1603.07455, 2016.

\bibitem{albertino} A. Maspero, D. Robert, {\it On time dependent Schr\"odinger equations: Global well-posedness and growth of Sobolev norms.} Journal of Functional analysis 273, 721-781, 2017. 

\bibitem{Montalto} R. Montalto, \emph{Quasi-periodic solutions of forced Kirchhoff equation}. Nonlinear Differ. Equ. Appl. NoDEA, 24:9, doi:10.1007/s00030-017-0432-3, 2017. 

\bibitem{Montalto1} R. Montalto, \emph{On the growth of Sobolev norms for a class of Schr\"odinger equations with superlinear dispersion}. To appear on Asymptotic Analysis. Preprint arXiv:1706.09704, 2017. 

\bibitem{Montaltoreducibility} R. Montalto, \emph{A reducibility result for a class of linear wave equations on $\T^d$.} Int. Math. Res. Notices, doi: 10.1093/imrn/rnx167, 2017. 

\bibitem{wang-1} W.-M. Wang,  {\it Logarithmic bounds on Sobolev norms for time dependent linear Schr\"odinger equations.} Comm. in Partial Differential Equations 33, no. 10-2,  2164-2179, 2008. 


\bibitem{SV}
J. Saranen, G. Vainikko, {\it Periodic Integral and Pseudodifferential Equations with Numerical Approximation.} {Springer Monographs in Mathematics}, 2002.

\bibitem{Taylor} M. Taylor, {\it Pseudo-differential operators and nonlinear PDEs}, Birkh\"auser, 1991. 




%



%
%
%



%










%

%






%


%





%



%




%


%


%






%

%





%

%





%


%

%


%

%

%

%




%


%



%


%


%

%
%
%

%

%


%



%

%

%

%


%


%


\end{thebibliography}
\end{document}